\documentclass[a4paper, 12pt, dvipdfmx]{amsart}

\usepackage{graphicx}
\usepackage{amssymb, amsmath, amsthm}
\usepackage{mathrsfs}
\usepackage{ascmac}
\usepackage{amscd}
\usepackage[all]{xy}
\usepackage{bm}
\usepackage{ulem}

\usepackage{tikz}
\usetikzlibrary{intersections, calc, arrows.meta}
\usetikzlibrary{patterns}

\theoremstyle{definition}
\newtheorem{theorem}{Theorem}[section]
\newtheorem{lemma}[theorem]{Lemma}
\newtheorem{corollary}[theorem]{Corollary}
\newtheorem{definition}[theorem]{Definition}
\newtheorem{proposition}[theorem]{Proposition}

\newtheorem{example}[theorem]{Example}
\newtheorem{remark}[theorem]{Remark}

\newtheorem*{proposition*}{Proposition}
\newtheorem*{definition*}{Definition}
\newtheorem*{remark*}{Remark}
\newtheorem*{claim*}{Claim}

\newtheorem{question}[theorem]{Question}

\newtheorem*{acknow*}{Acknowledgement}
\newtheorem*{organization*}{Organization}

\makeatletter

\@addtoreset{equation}{section}
\makeatother

\newcommand{\diam}[0]{\mathrm{diam}}
\newcommand{\Lip}[0]{\mathrm{Lip}}

\newcommand{\lip}{\mathrm{lip}}

\newcommand{\conc}{\mathrm{conc}}

\renewcommand{\epsilon}{\varepsilon}

\newcommand{\supp}{\mathrm{supp}}

\newcommand{\scr}{\mathscr}
\newcommand{\cal}{\mathcal}

\newcommand{\Hau}{\mathrm{H}}
\newcommand{\Prkh}{\mathrm{P}}

\newcommand{\Sep}{\mathrm{Sep}}
\newcommand{\ObsDiam}{\mathrm{ObsDiam}}

\title[Invariants of pyramids]{Invariants for Gromov's pyramids and their applications
}
\date{August 3, 2023.}
\author[S.~Esaki]{Syota Esaki}
\address{Department of Applied Mathematics, Fukuoka University, Fukuoka 814-0180, JAPAN}
\email{sesaki@fukuoka-u.ac.jp}

\author[D.~Kazukawa]{Daisuke Kazukawa}
\address{Faculty of Mathematics, Kyushu University, Fukuoka 819-0395, JAPAN}
\email{kazukawa@math.kyushu-u.ac.jp}

\author[A.~Mitsuishi]{Ayato Mitsuishi}
\address{Department of Applied Mathematics, Fukuoka University, Fukuoka 814-0180, JAPAN}
\email{mitsuishi@fukuoka-u.ac.jp}

\keywords{Pyramid, concentration, dissipation, separation distance}
\begin{document}

\maketitle

\begin{abstract}
Pyramids introduced by Gromov are generalized objects of metric spaces with Borel probability measures.
We study non-trivial pyramids, where non-trivial means that they are not represented as metric measure spaces.
In this paper, we establish general theory of invariants of pyramids and construct several invariants.
Using them, we distinguish concrete pyramids.
Furthermore, we study a space consisting of non-trivial pyramids and prove that the space have infinite dimension.
\end{abstract}

\setcounter{section}{-1}
\section{Introduction}

Interested objects in the present paper are pyramids introduced by Gromov.
Before starting to state something about pyramids, we shall review why pyramids are interesting briefly, by using a few paragraphs.

Following \cite{G:isop}, a pm-space $X =(X,d,m)$ means a complete separable metric space $(X,d)$ with Borel probability measure $m$.
Such an $X$ is often called a metric measure space, in several literature (\cite{G:green}, \cite{S}).
A typical example of pm-spaces is a complete Riemannian manifold with normalized volume measure.
Let us denote by $\cal X$ the space of all isomorphism classes of pm-spaces.
In view of convergence theory, it is important to study a topology on $\cal X$.
In this paper, we focus on the box and concentration topologies introduced by Gromov \cite{G:green}.

The convergence in the box topology is close to the measured Gromov-Hausdorff convergence and coincides with the weak-Gromov convergence induced by the Gromov-Prokhorov distance introduced by Greven-Pfaffelhuber-Winter (\cite{GPW}).
This topology has been well studied and understood.
The concentration topology is motivated by concentration of measure phenomenon.
For instance, L\'evy's result (\cite{Levy}, see also \cite{L}) is understood that the standard unit $n$-sphere $\mathbb S^n(1)$ with uniform measure converges to a one-point space in the concentration topology as the dimension $n$ diverges to infinity.
Here, the concentration topology is weaker than the box topology.

In \cite{G:green}, Gromov also introduced the concept of pyramids, which are generalized objects of pm-spaces.
Furthermore, he showed that the space $\Pi$ of all pyramids provides a natural compactification of $\cal X$ for the concentration topology (see Subsection \ref{subsec:box_and_conc}).
The topology on $\Pi$ is called the weak topology.
More precisely, 
a pyramid is defined as a box-closed subset of $\cal X$ satisfying certain axioms related with the Lipschitz order $\prec$, which is a partial order on $\cal X$ (Definition \ref{def:pyramid}).

Since $\Pi$ is a compactification of $\cal X$,
we can consider the limit as a pyramid of any sequence of pm-spaces by choosing a subsequence in the weak topology.
Moreover, if a sequence $\{Y_n\}$ of pm-spaces is monotone nondecreasing in the Lipschitz order:
\[
Y_1 \prec Y_2 \prec \cdots \prec Y_n \prec Y_{n+1} \prec \cdots,
\]
then $\{Y_n\}$ is known to have the limit in $\Pi$ as $n \to \infty$ (see Proposition \ref{prop:monotone_convergence}).
Typical examples of such pyramids are as follows. For each pm-space $X=(X,d,m)$ and for $t >0$, $n \in \mathbb N$ and $p \in [1,\infty]$, we set $tX$ and $X^n_{p}$ by
\[
tX = (X,td,m) \text{ and } X^n_p = (X^n, d_{\ell_p}, m^{\otimes n}),
\]
respectively,
where $d_{\ell_p}$ is the canonical $\ell_p$-product of distance function.
Then, we obtain monotone families $\{tX\}_{t > 0}$ and $\{X^n_p\}_{n \in \mathbb N}$ (see precisely, Section \ref{sec:prelim}).
Hence, their limits as $t \to \infty$ and as $n \to \infty$ exist in $\Pi$ and denoted by $\infty X$ and $X_p^\infty$, respectively.
Thus, the notion of pyramids can grasp scaling limit and infinite product limit.

Due to Shioya's result (\cite{S:sphere}), the infinite product limit of the standard Gaussian space is important in the study of pyramids.
For $\sigma \in (0, \infty)$, let us denote by $\gamma_{\sigma^2}^1$ the $1$-dimensional (centered) Gaussian measure of variance $\sigma^2$
and consider $\Gamma^1_{\sigma^2} = (\mathbb R, |\cdot|, \gamma_{\sigma^2}^1)$, which is called the $1$-dimensional Gaussian space of variance $\sigma^2$.
The infinite ($\ell_2$-)product limit of $\Gamma_{\sigma^2}^1$ is denoted by $\Gamma_{\sigma^2}^\infty := (\Gamma_{\sigma^2}^1)_2^\infty$.
This is called the ({\it virtual}) {\it infinite dimensional Gaussian space of variance $\sigma^2$} in \cite{S:sphere}.
Shioya showed that the $n$-sphere with suitable scaling converges to $\Gamma_{\sigma^2}^\infty$ in $\Pi$ as $n \to \infty$ (see Theorem \ref{thm:Shioya_sphere}).
Furthermore, he also showed that $\Gamma_{\sigma^2}^\infty$ is not represented by a pm-space. By the similar arguments, we know that $X_{p}^{\infty}$ is also not represented by a pm-space for any $X$ except for a one-point space (See \cite[Proposition 7.37]{S}).
Moreover, Ozawa and Shioya's work \cite{OS:lim} tells us that $\Gamma_{\sigma^2}^\infty$'s are mutually different for each $\sigma$. Note that, after Shioya's work \cite{S:sphere}, several sequences of manifolds like spheres are proved to have non-trivial limits as pyramids (\cite{ST}, \cite{KS}).
However, those pyramids are not known to distinguish with $\Gamma_{\sigma^2}^\infty$ or not.

Based on their studies, we aim to answer the following fundamental questions:
\begin{question} \label{question:main}
How to distinguish two given pyramids?
Moreover, how many non-trivial pyramids are?
Here, non-trivial pyramids mean that they are not represented by pm-spaces.
\end{question}

To answer the first question in Question \ref{question:main}, we establish a general theory of invariants of pyramids and give some concrete examples of invariants.
We outline our theory and examples.
Let $\cal I$ denote the set of all $1$-homogeneous monotone invariants on $\cal X$ which are lower semicontinuous in the box topology.
The cores of our theory are
\begin{itemize}
\item
each $F \in \cal I$ is canonically extended to an invariant on $\Pi$ (Corollary \ref{thm:unique_extension}), and 
\item 
$\cal I$ distinguishes any two pyramids (Theorem \ref{thm:invariant}).
\end{itemize}

As invariants in $\cal I$, there are some well-known quantities in geometric analysis.
For example, they are the ({\it observable}) {\it variance} $V$ in the sense of \cite{NS}, the $(p,q)$-{\it Poincar\'e constant} $C_{p,q}$ and the {\it logarithmic Sobolev constant} $C_\mathrm{LS}$. Here, the latter two constants mean the best constants of the associated functional inequalities. Actually, we show that these 
are in $\cal I$, and hence
these can be extended to invariants on $\Pi$ (see Section \ref{sec:distinguish}). 
The {\it observable diameter} of a pyramid introduced in \cite{OS:lim} is also an example of invariants in our sense. 

We list results for the important invariants $V$, $C_{2,2}$ and $C_{\mathrm{LS}}$ and their application.
\begin{theorem} \label{thm:invariant_of_smooth_space}
If a pm-space $X = (X,d,m)$ is smooth, that is, there exist a smooth connected manifold $M$ possibly with boundary, a complete Riemannian metric $g$ and $f \in C^\infty(M)$ such that
\[
(X,d,m) = (M,d_g,e^{-f(x)}\mathrm{vol}_g(dx)),
\]
where $d_g$ and $\mathrm{vol}_g$ are the distance function and the volume measure induced by $g$, respectively.
Then, we have
\[
C(X_p^\infty) = C(X)
\]
for $C= C_{2,2}, C_\mathrm{LS}$ and $p\ge 2$.
Furthermore, we have
\[
\sqrt{V(X^n_p)} \le \sqrt{V(X_p^\infty)} \le
C_{2,2}(X_p^\infty) \le
C_\mathrm{LS}(X_p^\infty)
\]
for every $n \ge 1$.
\end{theorem}

\begin{remark}
In Theorem \ref{thm:invariant_of_smooth_space}, if $p \in [1,2)$, then we have
\[
F(X_p^\infty) = \infty
\]
for any monotone homogeneous invariant $F \not\equiv 0$ unless $X$ consists of a single point.
This follows from Proposition \ref{prop:CLT}.
\end{remark}

Next, we see a result related with the comparison geometry. 
In comparison geometry, it is important to study convergence of sequences of Riemannian manifolds with a uniform curvature bound from below and to analyze the limit spaces. 
In the measured Gromov-Hausdorff topology and in the concentration topology, the limit is known to have
a lower Ricci curvature bound in a synthetic sense (\cite{St:I}, \cite{LV}, \cite{AGS}, \cite{FS} and \cite{OY}). 
Here, this synthetic notion of lower Ricci curvature bound was introduced by Sturm (\cite{St:I}) and by Lott-Villani (\cite{LV}) independently, which 
is called the CD$(K,\infty)$-condition, where $K$ is an arbitrary real number standing for the curvature bound.

\begin{theorem} \label{thm:CD_implies_FI}
If a pyramid $P$ is the weak limit of a sequence $\{X_n\}$ of pm-spaces such that $X_n$ is dominated by a CD($K,\infty$)-space for each $n$, where $K>0$ is a constant, then
\[
\sqrt{V(P)} \le
C_{2,2}(P) \le
C_{\mathrm{LS}}(P) \le \frac{1}{K}.
\]
\end{theorem}
In Theorem \ref{thm:CD_implies_FI}, the domination means the relation with respect to the Lipschitz order.

As an application of the theory of invariants, we actually distinguish the infinite-dimensional Gaussian space $\Gamma^\infty_{\sigma^2}$ and the infinite-dimensional cube $rI^\infty$. Here, $rI^\infty = (rI)^\infty_2$ is the infinite $\ell_2$-product of the interval $rI = ([0,r], |\cdot|, r^{-1}\mathcal{L}^1\lfloor_{[0,r]})$ of length $r$.

\begin{theorem} \label{thm:main}
The pyramids $\Gamma_{\sigma^2}^\infty$ and $rI^\infty$
are different to each other for every $\sigma, r > 0$.
\end{theorem}

Our framework provides one way to answer the first question in Question \ref{question:main}. To prove this theorem, we first calculate the $(2,2)$-Poincar\'e constant $C_{2,2}$. Then we have
\[
C_{2,2}(\Gamma^\infty_{\sigma^2}) = \sigma \quad \text{ and } \quad  C_{2,2}(rI^\infty) = \frac{r}{\pi}
\]
under our definition (Definition \ref{def:2,2-PI}). 
We next have the following:

\begin{theorem} \label{thm:F_Gamma}
There exists a monotone homogeneous lower semicontinuous invariant $F$ for pyramids such that
\[
F(\Gamma_{\sigma^2}^\infty) = \sigma \quad \text{ and } \quad F(rI^\infty) = \frac{r}{\sqrt{2 \pi}}
\]
for all $\sigma, r > 0$. 
Moreover, this invariant $F$ characterizes the Gaussian space in the following sense: if a pyramid $P$ is $\ell_2$-idempotent (see Definition \ref{def:idempotent}) and $F(P)=V(P)$, then
\[
P = \Gamma_{V(P)}^\infty.
\]
\end{theorem}

We immediately have a proof of Theorem \ref{thm:main} by using $F$ and $C_{2,2}$.

To answer the second question in Question \ref{question:main}, we prepare an infinite dimensional topological space $V$ given by
\[
V := \left\{ \alpha = (\alpha_i)_{i=1}^\infty \in [0, \infty)^{\mathbb N}\,\middle|\, \|\alpha\|_1 \le 1 \ \text{and} \ \alpha_i \ge \alpha_{i+1} \text{ for all } i \ge 1
\right\}.
\]
For each $\alpha \in V$, we construct a non-trivial pyramid $\cal X_\alpha$. 
Roughly speaking, $\cal X_\alpha$ is the set of all pm-spaces whose measure have atoms at least $\alpha$ (see Section \ref{sec:atom}).
Each $\cal X_\alpha$ is also obtained as the metric blow-up limit of a pm-space (Remark \ref{rem:X_alpha_is_blow-up_limit}) and $\cal X_\alpha$ is distinguished with $rI^\infty$ and $\Gamma_{\sigma^2}^\infty$ (Remark \ref{rem:X_alpha_is_not_I_and_Gamma}).
Then, we obtain:
\begin{theorem} \label{thm:mugen}
The correspondence $V \ni \alpha \mapsto \cal X_\alpha \in \Pi$ is a homeomorphism to the image.
In particular, the topological dimension of the space of non-trivial pyramids is infinite.
\end{theorem}

Remark that in the recent paper \cite{KNS}, Kazukawa, Nakajima and Shioya also show the same statement as Theorem \ref{thm:mugen} independently. 
Before appearing \cite{KNS} and this paper, the dimension of the space of non-trivial pyramids is known to be at least one, because $\{\Gamma_{\sigma^2}^\infty \}_{\sigma > 0}$ is a continuous injective curve (see also Remark \ref{rem:curve}).
Our proof is an application of the generalized concept of dissipation, which is different to the proof in \cite{KNS}. 
Here, the original notion of dissipation was introduced by Gromov \cite{G:green}.
To describe generalization, 
we extend the separation distance to with infinitely many real variables.
Moreover, we also consider a more wider class of pyramids than Theorem \ref{thm:mugen} in Section \ref{sec:atom} (see Theorem \ref{thm:homeo_with_delta}).
In addition, we consider algebraic properties of 
the set of $\cal X_\alpha$'s.
We introduce the $\ell_{p}$-product $P \otimes_p Q$ for $P, Q \in \Pi$ in Section \ref{sec:intrinsic_construction} and the product $\alpha\beta$ for $\alpha, \beta \in V$ in Section \ref{subsec:X_alpha-algebra}. Then, we show the consistency between these products.
\begin{theorem} \label{thm:algebra}
Let $p \in [1, \infty]$.
For $\alpha, \beta \in V$, we have
\[
\cal X_\alpha \otimes_p \cal X_\beta = \cal X_{\alpha \beta}.
\]
Furthermore, the map
\[
V \ni \alpha \mapsto \cal X_\alpha \in \{ \cal X_\beta \mid \beta \in V \}
\]
is an isomorphism as topological monoids.
\end{theorem}

\begin{organization*}
After this section, the present paper is divided into Sections \ref{sec:prelim}--\ref{sec:atom}.
In Section \ref{sec:prelim}, we recall several notions and fundamental properties to state our results.
In particular, we recall definitions of pyramids and the topology on the space $\Pi$ of all pyramids.
In Section \ref{sec:product_and_scale_change}, we explain invariant pyramids under two basic operations, the scaling and the product, respectively, and investigate their fundamental properties.
In Section \ref{sec:invariant}, we develop a general theory of invariants of pyramids.
In Section \ref{sec:distinguish}, we prove Theorems \ref{thm:CD_implies_FI} and \ref{thm:main}.
Finally, in Section \ref{sec:atom}, Theorems \ref{thm:mugen} and \ref{thm:algebra} are proved.
In this section, we generalize the notion of infinite dissipation 
and use this to prove Theorem \ref{thm:mugen}.
\end{organization*}

\begin{acknow*}
The present authors would like to thank S.~Honda,
Y.~Kitabeppu and T.~Shioya for their valuable discussions and useful comments. They also would like to thank S.~Sakata for suggesting the starting question. S.E. is supported in part by JSPS KAKENHI Grant Numbers JP17K14206, JP22H04942 and JP23K03158. 
D.K. is supported in part by JSPS KAKENHI Grant Numbers JP20J00147 and JP22K20338.
A.M. is supported in part by JSPS KAKENHI Grant Numbers JP20K03598, JP17H01091 and JP21H00977.
\end{acknow*}

\section{Preliminaries} \label{sec:prelim}

To explain what we deal in the present paper, let us recall and fix several notions.
Main references are \cite{G:green}, \cite{S} and \cite{V}.

\subsection{Basic concepts about metrics and measures} \label{subsec:basic}
Let us fix several notations.
Let $Y=(Y,d_Y)$ be a metric space.
Let us denote by $\mathscr B(Y)$ the Borel $\sigma$-field of $Y$ and by $\mathscr P(Y)$ the set of all Borel probability measures on $Y$.
For each $x \in Y$, $\delta_x$ is the Dirac measure at $x$.
For a measurable function $f : Y \to \mathbb R$ and $m \in \mathscr P(Y)$, we set
\[
E_{m}(f) := E_{(Y,m)}(f) := \int_Y f \,dm,
\]
whenever it is defined, which is called
the expectation of the random variable $f$.

For each $x \in Y$ and $r > 0$, we set
\begin{align*}
B_r(x) &:= \left\{ y \in Y \,\middle|\, d_Y(x,y) \le r\right\}, \\
U_r(x) &:= \left\{ y \in Y \,\middle|\, d_Y(x,y) < r\right\},
\end{align*}
which are the closed $r$-ball centered at $x$ and the open $r$-ball centered at $x$, respectively.
For a subset $Y' \subset Y$, we set $d_Y(Y',\cdot) = \inf_{y \in Y'}d_Y(y,\cdot)$, 
which is the distance from $Y'$.
Then we set $B_r(Y') = \{d_Y(Y',\cdot) \le r\}$ and $U_r(Y') = \{d_Y(Y',\cdot) < r \}$.

For a map $f : Y \to Z$ to another metric space $Z = (Z,d_Z)$, we define
\[
\Lip(f) := \sup_{y \ne y' \in Y} \frac{d_Z(f(y), f(y'))}{d_Y(y,y')}\,\, \in [0,\infty]
\]
which is called the {\it Lipschitz constant} of $f$.
Let us use the convention that $\Lip(f)=0$ if $Y$ is a one-point space.
When $\Lip(f) \le L$ (resp. $\Lip(f) < \infty$), we say that $f$ is {\it $L$-Lipschitz} (resp. {\it Lipschitz}).
The set of all real-valued Lipschitz functions on $Y$ is denoted by $\Lip(Y)$.
The subset of $\Lip(Y)$ consisting of $1$-Lipschitz (resp. bounded) functions is denoted by $\Lip_1(Y)$ (resp. $\Lip_b(Y)$).
For $f: Y \to Z$ and $x \in Y$, the {\it asymptotic Lipschitz constant} of $f$ at $x$ is defined by
\begin{equation} \label{def:asymp-Lip}
\lip_a(f)(x) := \lim_{r \to 0} \Lip(f|_{B_r(x)}) = \inf_{r > 0} \Lip(f|_{B_r(x)}).
\end{equation}
From the definition, $\lip_a(f)(x) = 0$ if $x$ is isolated.

For two closed subsets $A,A'$ of $Y$, the {\it Hausdorff distance} between them is defined by
\[
d_\Hau^{d_Y}(A,A') = d_\Hau^{(Y,d_Y)}(A,A') := \inf \left\{ r > 0 \,\middle|\,
B_r(A) \supset A', B_r(A') \supset A
\right\}.
\]
It is possibly $+\infty$.

When $m \in \mathscr P(Y)$ is fixed, for measurable functions $f,f' : Y \to \mathbb R$, we set
\begin{align*}
d_\mathrm{KF}(f,f') &=
d_\mathrm{KF}^m (f,f')
= d_\mathrm{KF}^{(Y,m)}(f,f')
\\
&:= \inf \left\{ \epsilon > 0 \,\middle|\, m(\{y \in Y \mid |f(y)-f'(y)| > \epsilon \}) \le \epsilon \right\}\,\, \in [0,1]
\end{align*}
which is called the {\it Ky Fan distance} between $f$ and $f'$ with respect to $m$.
Note that this is independent of $d_Y$.

For two measures $\mu, \mu' \in \mathscr P(Y)$, the {\it Prokhorov distance} between them is defined by
\begin{align*}
d_\mathrm{P}(\mu, \mu') &= d_\Prkh^{d_Y}(\mu,\mu') = d_\Prkh^{(Y,d_Y)}(\mu,\mu')  \\
&:= \inf \left\{ \epsilon >0 \,\middle|\,
\mu(B_\epsilon(A) ) + \epsilon \ge \mu'(A) \text{ for all } A \in \scr B(Y)
\right\}\,\, \in [0,1].
\end{align*}

For a signed Borel measure $\nu$ on $Y$, the total variation norm of $\nu$ is defined as
\[
\| \nu \|_\mathrm{TV} := |\nu|(Y).
\]
Note that it is independent of the choice of $d_Y$.

We will use the following famous extension procedure of Lipschitz functions, called the McShane-Whitney extension:
\begin{proposition}
\label{prop:Lip_ext}
Let $Y$ be a metric space and $A$ a non-empty subset of $Y$.
Let $f : A \to \mathbb R$ be a $L$-Lipschitz function.
Then, the function defined by $\tilde f(x) = \inf_{y \in A} \left( f(y) + Ld_Y(x,y) \right)$ for $x \in Y$, is $L$-Lipschitz on $Y$, which is an extension of $f$.
\end{proposition}
Furthermore, if $f$ is bounded, then truncating $\tilde f$, we obtain a bounded Lipschitz extension of $f$.

\begin{definition} \label{def:entropy}
For $m \in \mathscr P(Y)$ and a nonnegative measurable function $f : Y \to \mathbb R$, we define the {\it entropy} of $f$ with respect to $m$ by
\[
\mathrm{Ent}_m(f) = \int_Y f \log f\,dm - \left(\int_Y f\,dm \right)\log \left(\int_Y f\,dm\right)
\]
whenever $f \log (1+f) \in L^1(m)$.
Otherwise, we set $\mathrm{Ent}_m(f) = \infty$.
Here, we use the convention that $0 \log 0 := 0$.
\end{definition}
Due to Jensen's inequality, the entropy is always nonnegative.
Remark that the entropy is $1$-homogeneous, that is,
\begin{equation} \label{eq:ent_is_homog}
\mathrm{Ent}_m(tf) = t\, \mathrm{Ent}_m(f)
\end{equation}
holds for every $t > 0$.

For a metric space $Y$, 
we define an equivalent relation $\sim$ on $\Lip_1(Y)$ by $\varphi \sim \psi$ provided that there exists $c \in \mathbb R$ such that $\varphi(x) = \psi(x)+c$ on $Y$.
The equivalent class of $\varphi$ is denoted by $[\varphi]$ and the set of all equivalent classes is denoted by $\cal L_1(Y)$.
For $m \in \mathscr P(Y)$, a metric on $\cal L_1(Y)$ is induced by the Ky Fan metric $d_\mathrm{KF}^m$. 

Recall that $(\cal L_1(Y), {d_\mathrm{KF}^m})$ is compact
(see \cite[Proposition 4.46]{S}).

\subsection{Metric measure spaces} \label{subsec:mm-space}
We call a triple $X=(X,d,m)$ a metric measure space (for short, an mm-space), if $(X,d)$ is a complete separable metric space and $m$ is a nonnegative Borel measure on $(X,d)$.
Following \cite{G:isop}, we say that an mm-space $(X,d,m)$ is a {\it pm-space} if $m(X)=1$.
For an mm-space $X=(X,d,m)$ and $t > 0$, we set $tX = (X,td,m)$ the metric scale change of $X$.
From now on, for an mm-space $X=(X,d,m)$, we often write $d=d_X$ and $m=m_X$.

For an mm-space $X$, we denote by $\diam\, X$ the diameter of $\supp\,m_X$ and by $\# X$ the number of elements of $\supp\, m_X$.
Here,
\[
\supp\, m_X := \left\{ x \in X \mid m_X(B_r(x)) > 0 \text{ for all } r > 0\right\}
\]
is the support of $m_X$.
Note that $\supp\, m_X$ is closed in $X$ and that $m_X(X - \supp\, m_X) = 0$ because $X$ is Polish.

We say that two mm-spaces $X$ and $Y$ are {\it mm-isomorphic}, if there exists a Borel map $f : X \to Y$ which is an isometry between $\supp\, m_X$ to $\supp\,m_Y$ and the push-forward measure $f_\# m_X$ coincides with $m_Y$.
We call such an $f$ an {\it mm-isomorphism}.
It is clear that if $X$ and $Y$ are mm-isomorphic, then $m_X(X)=m_Y(Y)$ holds.
Furthermore, we denote by $\cal X$ the set of all mm-isomorphism classes of pm-spaces.
The class consisting of $X$ with $\# X=1$ is denoted by $\ast$.
We often call an mm-isomorphism class of pm-spaces itself a pm-space.

\begin{definition}
We say that a pm-space $X=(X,d,m)$ is {\it smooth} if $X$ is a smooth connected manifold possibly with boundary, with a complete Riemannian metric $g$, such that $d=d_g$ is the distance function induced by $g$ and that $m$ is a weighted measure which forms
\[
m(dx) = e^{-f(x)} \mathrm{vol}_g(dx)
\]
for some $f \in C^\infty(X)$.
\end{definition}

\subsection{Box distance and observable distance} \label{subsec:box_and_conc}
To study several convergence phenomena of pm-spaces, Gromov defined two metrics on $\cal X$.
One of them is the box distance denoted by $\square$, and another one is the observable distance denoted by $d_\conc$.
We recall their definitions.

For a pm-space $X$, a {\it parameter} of $X$ is a Borel map $f : [0,1] \to X$ satisfying $f_\# \cal L^1 = m_X$.
Here, the interval $[0,1]$ has the usual topology and $\cal L^1$ denotes the Lebesgue measure restricted to $[0,1]$.
For two pm-spaces $X, Y$, the {\it box distance} is defined by
\[
\square(X,Y) := \inf_{a,b}
\inf
\left\{ \epsilon > 0 \,\middle| \,
\begin{aligned}
& \text{there exists a Borel set $I_0 \subset {[0,1]}$ with } \\
& \text{$\mathcal L^1(I_0) \ge 1 - \epsilon$ such that} \\
& \text{$|a^\ast d_X - b^\ast d_Y| \le \epsilon$ on $I_0$}
\end{aligned}
\right\}.
\]
Here, $a$ and $b$ run over all parameters of $X$ and $Y$, respectively, and
$a^\ast d_X$ stands for the function
\[
a^\ast d_X(s,t) := d_X(a(s), a(t))
\]
for $s,t \in [0,1]$.
A sequence of pm-spaces is said to {\it box-converge} if it converges with respect to $\square$.
The box distance and the box-convergence are known to have several equivalent formulations (\cite{S}, \cite{K:thesis}, \cite{GPW}, \cite{N}).
In particular, we mainly use the following characterization of box-convergence, due to the second author:
\begin{lemma}[\cite{K:thesis}] \label{lem:mG_convergece}
Let $X_n, X \in \cal X$, $n \in \mathbb N$.
If $X_n$ box-converges to $X$ as $n \to \infty$, then there exists a complete separable metric space $Y$ having isometric embeddings $\iota_n : X_n \to Y$ and $\iota : X\to Y$ such that
the push-forward measure $(\iota_n)_\# m_{X_n}$ converges to $\iota_\#m_X$ weakly in $\mathscr P(Y)$.
\end{lemma}
Note that the converse statement is clearly true.

For two pm-spaces $X,Y$, we set
\[
d_\conc(X,Y) := \inf_{a,b} d_\Hau^{\,d_{\mathrm{KF}}^{([0,1], \cal L^1)}} (a^\ast \Lip_1(X), b^\ast \Lip_1(Y)),
\]
where $a,b$ stand for parameters of $X,Y$, respectively, and
\[
a^\ast \Lip_1(X) = \{ f \circ a \mid f \in \Lip_1(X)\}.
\]
The topology induced by $d_\conc$ is called the {\it concentration topology} and the convergence with respect to $d_\conc$ is said to {\it concentrate}.

From the definition, we have
\[
\square(X,Y) \ge d_\conc(X,Y).
\]
In particular, the concentration topology is coarser than the box topology.
Furthermore, two topologies are known to be actually different (\cite{F} see also \cite{S}).

\subsection{Lipschitz order and pyramids} \label{subsec:Lip_order_and_pyramid}

For two pm-spaces $X,Y$, we say that $X$ {\it dominates} $Y$ if there exists a $1$-Lipschitz map $f : X \to Y$ such that $m_Y = f_\# m_X$.
Then we denote $X \succ Y$ and $Y \prec X$.
This relation $\prec$ is called the {\it Lipschitz order}.
It is known that $X \prec Y$ and $Y \prec X$ implies $X$ and $Y$ are mm-isomorphic.
So, $\prec$ becomes a partial order relation on $\cal X$, which is denoted by the same symbol $\prec$.

\begin{definition} \label{def:pyramid}
We say that a family $P$ of isomorphism classes of pm-spaces is a {\it pyramid} if
\begin{enumerate}
\item $P$ is directed (in the Lipschitz order), that is, for every $X,Y \in P$, there exists $Z \in P$ such that $X \prec Z$ and $Y \prec Z$;
\item $P$ is downward-closed (in the Lipschitz order), that is, if $X \in P$ and $Y \prec X$, then $Y \in P$ holds;
\item $P$ is non-empty and closed in the box topology.
\end{enumerate}
The set of all pyramids is denoted by $\Pi$.
\end{definition}

Note that Gromov's original definition of pyramids {imposes} the conditions except $P$ being box-closed. Shioya added this property for establishing Theorem \ref{thm:rho} stated later.

\begin{example}
For a pm-space $X$, we set
\[
\cal PX := \left\{ Y \in \cal X \,\middle|\, Y \prec X \right\}.
\]
It is known that $\cal PX$ is a pyramid (\cite{G:green}, see also \cite{S}), called the {\it pyramid with apex} $X$.
It is clear that
\[
X \prec Y \iff \cal P X \subset \cal PY
\]
for $X, Y \in \cal X$.
In particular, this construction gives a map
\[
\iota : \cal X \ni X \mapsto \cal PX \in \Pi.
\]
From the definition, this map is injective.
Due to the map, we are going to regard $\cal X$ as a subset of $\Pi$.
In particular, $\{\ast \} = \cal P\ast$ is the smallest pyramid in the inclusion order.
It is obvious that $\cal X$ is a pyramid, which is called the
{\it greatest pyramid.}

As mentioned in Question \ref{question:main}, we say that a pyramid $P$ is {\it non-trivial} if $P$ is not contained in the image of $\iota$ and is not the greatest pyramid $\cal X$.
\end{example}

For $t > 0$ and two pm-spaces $X$ and $Y$, we have
\begin{equation} \label{eq:tX_box}
\square(tX, tY) \le c(t) \square(X,Y)
\end{equation}
due to \cite{S}, where $c(t)$ is a universal constant depending on $t$.
For a pyramid $P$, we set
\begin{equation} \label{eq:tP}
tP := \{t X \in \cal X \mid X \in P \}
\end{equation}
which is a pyramid, by \eqref{eq:tX_box}.

A convergence notion of pyramids is defined by a more general setting:
\begin{definition}[\cite{G:green}]
Let $F_n, F$ be closed subsets of $\cal X$ with respect to the box topology.
We say that the sequence $\{F_n\}_n$ {\it converges to} $F$ {\it in the weak Hausdorff sense} if the following (1) and (2) hold:
\begin{enumerate}
\item for every $X \in F$, there exists $X_n \in F_n$ such that $\square(X_n,X) \to 0$ as $n \to \infty$;
\item for every $X \not\in F$, $\liminf_{n \to \infty} \inf_{Y\in F_n} \square(X, Y) > 0$.
\end{enumerate}
In this case, we also say that $\{F_n\}$ {\it weakly converges to} $F$.
\end{definition}
Remark that this notion is also said to be {\it Kuratowski convergence}.
{It is known that the set $\Pi$ of all pyramids is closed under the weak convergence (see \cite[Theorem 6.11]{S}).}

We will say that a sequence $\{X_n\}_n$ of pm-spaces {\it weakly converges} (to a pyramid $P$) if $\{\cal PX_n\}_n$ weakly converges to $P$.
This condition is also said that $\{X_n\}$ is asymptotic, in \cite{G:green} and \cite{S}.

Gromov showed
\begin{theorem}[{\cite{G:green}}]
$\Pi$ is sequentially compact in the weak convergence.
Via the natural injective map $\iota : \cal X \ni X \mapsto \cal PX \in \Pi$, the concentration coincides with the weak convergence.
Furthermore, the image of $\iota$ is dense in $\Pi$ in the weak convergence.
\end{theorem}

Furthermore, Shioya defined a metric on $\Pi$ which is compatible with the weak convergence. More precisely, he showed:
\begin{theorem}[\cite{S}] \label{thm:rho}
There exists a distance function $\rho$ on $\Pi$ such that the following (1)--(4) hold.
\begin{enumerate}
\item For a sequence $\{P_n\}_n \subset \Pi$ and $P \in \Pi$, $\rho(P_n,P) \to 0$ as $n \to \infty$ if and only if $\{P_n\}$ weakly converges to $P$ as $n \to \infty$.
\item In particular, $(\Pi, \rho)$ is a compact metric space.
\item The map $\iota : \cal X \ni X \mapsto \cal PX \in \Pi$ is $1$-Lipschitz with respect to $d_\conc$.
\item Moreover, the map $\iota : \cal X \to \Pi$ is homeomorphic to the image with respect to $d_\conc$ and $\rho$.
\end{enumerate}
\end{theorem}
From now on, we often call the topology induced from $\rho$ on $\Pi$ the {\it weak  (Hausdorff) topology}.

\begin{remark}
By the statements (2)--(4) of Theorem \ref{thm:rho}, $(\Pi, \rho)$ is a compactification of $(\cal X, d_\conc)$.
The statement (2) follows from (1) (\cite{S}).
\end{remark}

Let us recall an explicit definition of $\rho$ appeared in Theorem \ref{thm:rho}, which will be used in the proof of Lemma \ref{lem:continuity02}.
For $R>0, N \in \mathbb N$, we set
\[
\cal X(N,R) := \left\{ X \in \cal X \,\middle|\, \begin{aligned}
&(\supp\,m_X,d_X) \text{ is isometrically embedded in}\\
& \left( \{x \in \mathbb R^N \mid \|x\|_\infty \le R \}, \|\cdot\|_\infty \right)
\end{aligned}  \right\}.
\]
\begin{definition}[\cite{S}, \cite{S:sphere}] \label{def:rho}
For two pyramids $P, Q$, we set
\[
\rho (P,Q) = \sum_{k=1}^\infty \frac{1}{4k \cdot 2^k} d_\Hau^{\,\square}(P \cap \cal X(k,k), Q \cap \cal X(k,k)).
\]
\end{definition}

Monotone sequences automatically have the limits:
\begin{proposition}[\cite{S}] \label{prop:monotone_convergence}
Let $\{P_n\}_n$ be a sequence of pyramids. Then, the following {(1) and (2) hold}.
\begin{enumerate}
\item If $\{P_n\}_n$ is monotone nondecreasing in the sense that
$P_n \subset P_{n+1}$ for every $n$, then $\{P_n\}_n$ weakly converges to $\overline{\bigcup_{n} P_n}$.
Here, $\overline{(\,\cdot\,)}$ means the closure in the box topology.
\item If $\{P_n\}_n$ is monotone nonincreasing in the sense that $P_n \supset P_{n+1}$ for every $n$, then $\{P_n\}_n$ weakly converges to $\bigcap_{n} P_n$.
\end{enumerate}
\end{proposition}

In particular, if a sequence $\{X_n\}$ of pm-spaces is monotone nondecreasing in the Lipschitz order: $X_1 \prec X_2 \prec \cdots$, then it weakly converges to some pyramid.
It is known that every pyramid $P$ admits a monotone nondecreasing sequence of pm-spaces which converges to $P$ weakly.
Such a sequence is called an {\it approximation} of $P$.

\begin{example}[Infinite product] \label{ex:product}
For pm-spaces $X, Y$ and $p \in [1, \infty]$, their $\ell_p$-product is defined in the usual way:
\[
X \times_p Y
:= (X \times Y, d_X \times_{p} d_Y, m_X \otimes m_Y).
\]
Here, $m_X \otimes m_Y$ is the product measure of $m_X$ and $m_Y$, and $d_X \times_p d_Y$ is the standard $\ell_p$-product metric:
\[
(d_X \times_p d_Y) ((x,y), (x',y')) := \left\{
\begin{aligned}
&\left( d_X(x,x')^p + d_Y(y,y')^p \right)^{1/p} && \text{ if } 1 \le p< \infty, \\
& \max\{ d_X(x,x'), d_Y(y,y')\} &&\text{ if } p= \infty,
\end{aligned}
\right.
\]
for $x,x' \in X$ and $y,y' \in Y$.

For $n \in \mathbb N$, we set
\[
X_p^1 := X \text{ and } X_p^n := X_p^{n-1} \times_p X \text{ for $n \ge 2$}.
\]
Then, the standard projection $X_p^{n} \to X_{p}^{n-1}$ is a domination.
So, we have the pyramid:
\[
X_p^\infty := \overline{\bigcup_{n=1}^\infty \cal PX_p^{n}}.
\]
This is called the {\it infinite $\ell_p$-product} of $X$.
Note that {$X_p^\infty$ may be trivial}.
We later show that $X_p^\infty$ is $\cal X$
for $p <2$ and for every $X$ unless $X= \ast$.
Note that
\[
t(X_p^\infty) = (tX)_p^\infty
\]
holds for $t > 0$.
\end{example}

\begin{example} \label{ex:infinite_product_2}
Let $\{X_n\}_{n=1}^\infty$ be an arbitrary sequence of pm-spaces and $p \in [1, \infty]$.
Setting $Y_n = X_1 \times_p X_2 \times_p \cdots \times_p X_n$, we have a monotone nondecreasing sequence $\{Y_n\}_{n=1}^\infty$ in the Lipschitz order.
We denote the weak limit of $\{Y_n\}$ by
\[
\bigotimes_{n=1}^\infty (\{X_n\}; \ell_p)
\]
and call it the {\it infinite $\ell_p$-product of $\{X_n\}$}.

More generally, for sequences $\{X_n\}_n \subset \cal X$ and $\{p_n\}_n \subset [1, \infty]$, we can consider the limit
\[
\left( \left( \left( X_1 \times_{p_1} X_2 \right) \times_{p_2} \cdots \right) \times_{p_{n-1}} X_n \right) \times_{p_n} \cdots
\]
as a pyramid.
Although, in this paper, we do not treat with such a complicated product.
\end{example}

In this paper, we focus on
the following concrete pyramids obtained by the infinite $\ell_2$-product.  
\begin{example} \label{ex:Gaussian}
Recall that, for $\sigma \in (0, \infty)$, the one-dimensional Gaussian measure of variance $\sigma^2$ (and of mean $0$) is defined by
\[
\gamma^1_{\sigma^2} (dt) := \frac{1}{\sqrt{2\pi \sigma^2}} e^{-t^2/2\sigma^2}\,dt.
\]
Then, the triple
\[
\Gamma^1_{\sigma^2} := (\mathbb R, |\cdot|, \gamma^1_{\sigma^2})
\]
is called the one-dimensional Gaussian space of variance $\sigma^2$.
Furthermore, its infinite $\ell_2$-product
\[
\Gamma^\infty_{\sigma^2} := (\Gamma_{\sigma^2}^1)^\infty_2 = \overline{\bigcup_{n=1}^\infty \cal P \Gamma_{\sigma^2}^n}
\]
is called {\it the virtual infinite dimensional Gaussian space of variance $\sigma^2$} in \cite{S}.
Recall that $\overline{(\,\cdot\,)}$ denotes the closure in the box topology.
Here,
\[
\Gamma_{\sigma^2}^n := (\Gamma_{\sigma^2}^1)^n_2 = (\mathbb R^n, \|\cdot\|_2, \gamma^n_{\sigma^2})
\]
is called the $n$-dimensional Gaussian space of variance $\sigma^2$, where
\[
\gamma_{\sigma^2}^n(dx) = (\gamma_{\sigma^2}^1)^{\otimes n}(dx) = \frac{1}{\sqrt{2\pi \sigma^2}^{\,n}} e^{-\|x\|^2_2/2 \sigma^2}\,dx
\]
is the $n$-dimensional Gaussian measure of variance $\sigma^2$ (with mean $0$).
\end{example}

From now on, $\Gamma^n$ denotes $\Gamma^n_1$ for $n=1,2,\dots, \infty$.
By the definition, we have
\[
\sigma \Gamma^n = \Gamma^n_{\sigma^2}
\]
for every $n \in \mathbb N \cup \{\infty\}$.

\begin{example} \label{ex:cube}
Let us consider the unit interval $I = ([0,1], |\cdot|, \cal L^1)$ with standard metric and measure.
Then, we obtain the infinite $\ell_2$-product of $I$, denoted by
\[
I^\infty := I^\infty_2 =  \overline{  \bigcup_{n=1}^\infty \cal P I^n_2}.
\]
\end{example}

\begin{example}[Blow-up limit] \label{ex:blow-up}
Let $X$ be a pm-space.
Then, $\{t X\}_{t > 0}$ is clearly a monotone nondecreasing family.
So, the limit
\[
\infty X := \lim_{t \to \infty} \cal P(t X) = \overline{ \bigcup_{ t> 0} \cal P(tX) }
\]
exists. We call it the {\it metric blow-up} of $X$.

If $M$ is a smooth pm-space with $\dim M \ge 1$,
then $\infty M = \cal X$, due to \cite{S}.
\end{example}

From now on, we use the following convention:
\begin{equation} \label{eq:infty_Gamma=cal_X}
\infty \Gamma^n = \Gamma_{\infty}^n = \cal X \text{ and } \Gamma_0^n = \{\ast \}
\end{equation}
for every $n \in \mathbb N \cup \{\infty\}$.

As mentioned in the introduction, Shioya proved:
\begin{theorem}[\cite{S:sphere}] \label{thm:Shioya_sphere}
Let $r_n \in (0, \infty)$ and $\mathbb S^n(r_n) = \{x \in \mathbb R^{n+1} \mid \|x\|_2 = r_n \}$. Suppose that $\mathbb S^n(r_n)$ has uniform measure and distance function restricted to the Euclidean distance.
Then, $r_n/\sqrt{n}$ converges to $\sigma \in [0, \infty]$ if and only if $\mathcal P \mathbb S^n(r_n)$ weakly converges to $\Gamma_{\sigma^2}^\infty$.
Here, $\Gamma_0^{\infty} = \{\ast\}$ and $\Gamma_\infty^\infty = \cal X$.
\end{theorem}

\begin{remark}
Due to the second author, the statement of Theorem \ref{thm:Shioya_sphere} is also true even if the distance function on $\mathbb S^n(r_n)$ is replaced by the intrinsic distance (\cite{K:mtf}).
Furthermore, by the second author and Shioya, it is known that the Euclidean ball $\mathbb B_{r_n}(0_n) = \{ x \in \mathbb R^n \mid \|x\|_2 \le r_n \}$ of radius $r_n$ with the flat distance and the uniform measure also converges to $\Gamma^\infty_{\sigma^2}$ if and only if $r_n/\sqrt{n} \to \sigma$ (\cite{KS}).

So, $\Gamma_{\sigma^2}^\infty$ is regarded as the infinite dimensional sphere (at the same time, infinite dimensional ball).
\end{remark}

\subsection{Fundamental	construction} 
\label{sec:intrinsic_construction}
Let $\cal D \subset \cal X$ be a directed set, that is, for every $X, X'  \in \cal D$, there exists $X'' \in \cal D$ with $X \prec X''$ and $X' \prec X''$.
If $\cal D$ is non-empty, the set
\begin{equation} \label{eq:intrinsic_construction}
\cal P \cal D := \overline{\bigcup_{X \in \cal D} \cal PX}
\end{equation}
is known to be a pyramid (see \cite[Lemma 6.10]{S}).
Such a $\cal D$ is called a {\it generating set} of $\cal P \cal D$.
Furthermore, if $\cal D$ is downward-closed {in addition}, then it is clear that $\cal P \cal D=\overline{\cal D}$.

Using such a way, we can define the {\it product} of pyramids.
For two pyramids $P,Q \in \Pi$ and $1 \le p \le \infty$, we set
\begin{equation} \label{eq:gen_of_product}
\cal D_{P \otimes_p Q} := \{ X \times_p Y \mid  X \in P, Y \in Q \}.
\end{equation}
It is easily to checked that $\cal D_{P \otimes_p Q}$ is a non-empty
directed set.
So, we obtain a pyramid
\[
P \otimes_p Q:= \cal P \cal D_{P \otimes_p Q}
\]
which is called the {\it $\ell_p$-product} of $P$ and $Q$.
For $n \ge 1$, we denote the $n$-times $\ell_p$-product of $P$ by $P_p^n$ and by $P^{\otimes_p n}$.

In particular, we obtain 
\[
P_p^\infty := \lim_{n \to \infty} P_p^n
\]
and
\[
\infty P := \lim_{ t \to \infty} t P = \overline{\bigcup_{t > 0} tP}
\]
by using Proposition \ref{prop:monotone_convergence}.
We also note that $\infty (P_p^\infty) = (\infty P)_p^\infty$ holds.

\begin{remark}
In stochastic analysis, Wiener measures (e.g., the law of Brownian motion on $\mathbb R^d$) are basic important objects.
We note that any Wiener process give a directed system of pm-spaces in a canonical way.
Hence, we obtain the corresponding pyramid.
Furthermore, this construction seems to be different from the infinite product construction.
A study of pyramids associated with Wiener process will be appeared in our future paper.
\end{remark}

\begin{lemma} \label{lem:approx_product}
Let $\{X_n\}_{n=1}^\infty$ and $\{Y_n\}_{n=1}^\infty$ be approximations of pyramids $P$ and $Q$, respectively.
Then $\{X_n \times_p Y_n\}_{n=1}^\infty$ is an approximation of $P \otimes_p Q$ for each $p \in [1,\infty]$.
\end{lemma}
\begin{proof}
Let us consider the generating set $\mathcal D_{P \otimes_p Q}$ of the product $P \otimes_p Q$ as in \eqref{eq:gen_of_product}.
Since $\{X_n \times_p Y_n\}_n$ is monotone nondecreasing, we obtain the weak limit
\[
R = \lim_{n \to \infty} \cal P(X_n \times_p Y_n).
\]
By $X_n \times_p Y_n \in \cal D_{P \otimes_p Q}$, we have $R \subset P \otimes_p Q$.

Let us take $X \in P$ and $Y \in Q$.
Then, there exist $X_n' \prec X_n$ and $Y_n' \prec Y_n$ such that $\square(X_n',X) + \square(Y_n',Y) \to 0$ as $n \to \infty$.
By \cite{K:product}, we obtain
\[
\square(X_n' \times_p Y_n', X \times_p Y) \to 0.
\]
Since $X_n \times_p Y_n \succ X_n' \times_p Y_n'$, we have {$X \times_p Y \in R$}.
Therefore, $\cal D_{P \otimes_p Q} \subset R$, 
and hence,
$P \otimes_p Q \subset R$.
This completes the proof.
\end{proof}

\begin{remark}
We shall give a small remark.
In the intrinsic construction \eqref{eq:intrinsic_construction}, the closure operation can not removed, in general.
For instance, $I^\infty$ is strictly bigger than $\bigcup_{n=1}^\infty \mathcal PI^n$, because the later set only consists of pm-spaces of finite diameter nevertheless the former set contains $\Gamma^1_{1/12}$ by Proposition \ref{prop:cube_vs_Gaussian}.
\end{remark}

\subsection{Observable diameter and separation distance}
Gromov introduced several invariants for pm-spaces \cite{G:green}.
We recall two of them, called the observable diameter and separation distance.
In \cite{OS:lim}, they are extended to the space of pyramids.

\begin{definition} \label{def:monotone_and_homogeneous_on_X}
Let $F : \cal X \to [0,\infty]$ be a functional.
We say that
\begin{itemize}
\item $F$ is {\it monotone} if $F(X) \le F(Y)$ for $X,Y \in \cal X$ with $X \prec Y$,
\item $F$ is {\it homogeneous} if there exists $d > 0$ such that $F(tX) = t^d F(X)$ for $X \in \cal X$ and $t > 0$.
\end{itemize}
More precisely, the second condition is said to be $d$-homogeneous.

These terminologies
are used for invariants of pyramids, that is,
a functional $F : \Pi \to [0,\infty]$ is said to be
\begin{itemize}
\item {\it monotone} if $F(P) \le F(Q)$ for $P, Q \in \Pi$ with $P \subset Q$,
\item {\it homogeneous} if there exists $d > 0$ such that $F(tP) = t^d F(P)$ for $P \in \Pi$ and $t > 0$.
\end{itemize}
The last condition is also said to be $d$-homogeneous.
\end{definition}

\begin{definition}[\cite{G:green}]
For $\kappa \in [0,1]$ and $\mu \in \mathscr P(\mathbb R)$, we set
\[
\diam_{\mathbb R} (\mu;1-\kappa) := \inf \left\{ \diam\, A\,\middle|\, A \subset \supp\,\mu \text{ with } \mu(A) \ge 1 -\kappa \right\}
\]
which is called the {\it partial diameter} of $\mu$.
Then, we set
\[
\ObsDiam(X;-\kappa) := \sup \left\{ \diam_{\mathbb R}(f_\# m_X;1-\kappa) \,\middle|\, f \in \Lip_1(X) \right\}
\]
which is called the {\it observable diameter} of $X$.
Note that $\ObsDiam(X;0)= \diam \, X = \diam\, \supp\, m_X$.
\end{definition}

\begin{definition}[\cite{OS:lim}]
The {\it observable diameter} of a pyramid $P$ is defined by
\begin{equation} \label{eq:obsdiam_of_pyramid}
\ObsDiam(P;-\kappa) := \lim_{\epsilon \to 0+} \sup_{X \in P} \ObsDiam(X;-\kappa-\epsilon).
\end{equation}
\end{definition}

\begin{remark} \label{rem:limit_formula}
Due to Ozawa and Shioya (\cite{OS:lim}), it is known that the right hand side of \eqref{eq:obsdiam_of_pyramid} is well-defined.
Furthermore, by the limit formula in \cite{OS:lim} and the monotonicity, it is reformulated as
\[
\ObsDiam(P;-\kappa) = \sup_{X\in P} \ObsDiam(X;-\kappa).
\]
\end{remark}

\begin{example} \label{ex:obsdiam_of_Gauss}
The observable diameter of the Gaussian space has an explicit form.
For $t \in [0, \infty)$, we set
\[
\Psi(t) := \gamma^{1}([0,t]) = \frac{1}{\sqrt{2 \pi}} \int_0^t e^{-s^2/2}\, ds.
\]
It is clearly smooth and invertible.
Then, we have
\[
\ObsDiam(\Gamma_{\sigma^2}^n;-\kappa) = 2 \sigma \Psi^{-1}\left(
\frac{1-\kappa}{2}
\right)
\]
for every $n \in \mathbb N$ and $\sigma > 0$ (see \cite{OS:lim}).

Furthermore, Ozawa and Shioya (\cite{OS:lim}) showed
\begin{equation} \label{eq:obsdiam_Gauss}
\ObsDiam(\Gamma_{\sigma^2}^\infty; - \kappa) = 2 \sigma \Psi^{-1} \left(
\frac{1-\kappa}{2}
\right)
\end{equation}
which is the same as the finite dimensional case.
In particular,
$\Gamma^\infty_{\sigma^2}$ and $\Gamma^\infty_{\tau^2}$ are distinct for each $\sigma \ne \tau > 0$.
\end{example}

\begin{definition}[{{\cite{G:green}}}] \label{def:separation}
Let $\kappa_1, \dots, \kappa_N > 0$. For $X \in \cal X$, the {\it separation distance} $\Sep(X;\kappa)$ of $X$ with respect to $\kappa = (\kappa_1, \dots, \kappa_N)$ is defined by the supremum of $\min_{i \ne j} d_X(A_i, A_j)$, where $\{A_i\}_{i=1}^N$ is a family of closed subsets of $X$ with $m_X(A_i) \ge \kappa_i$ for each $i$.
\end{definition}

The definition of the separation distance will be extended in Subsection \ref{subsec:dissip}.

\begin{definition}[{{\cite{OS:lim}}}] \label{def:sep_for_pyramid}
Let $\kappa = (\kappa_1, \dots, \kappa_N)$ be the same as in Definition \ref{def:separation}.
For $P \in \Pi$, the {\it separation distance} of $P$ with respect to $\kappa$ is defined by
\[
\Sep(P;\kappa) := \lim_{\delta \to 0+} \sup_{X\in P} \Sep(X; \kappa_1-\delta, \dots, \kappa_N-\delta).
\]
\end{definition}

Clearly, the observable diameter and the separation distance are monotone and $1$-homogeneous on $\cal X$ and on $\Pi$.

\section{Idempotent pyramids and scale invariant pyramids}
\label{sec:product_and_scale_change}

As in Subsection \ref{sec:intrinsic_construction}, we can define the product and the scale-change of pyramids.
In this section, we show several properties for these constructions.

\subsection{Idempotent pyramids}

\begin{lemma} \label{lem:product01}
Let $p \in [1,\infty]$.
Let $P$ and $Q$ denote arbitrary pyramids.
Then, the following {(1)--(5) hold}.
\begin{enumerate}
\item $\otimes_p$ is a commutative operator, that is, $P \otimes_p Q= Q \otimes_p P$ holds.
\item If $P \subset Q$, then $P \otimes_p R \subset Q \otimes_p R$ for each $R \in \Pi$.
\item $P \otimes_p Q$ contains both $P$ and $Q$.
\item The smallest pyramid $\{\ast \}$ is the identity element for $\otimes_p$.
\item The greatest pyramid $\cal X$ satisfies that $\cal X \otimes_p P = \cal X$.
\end{enumerate}
\end{lemma}
\begin{proof}
Recall that $X,Y \in \cal X$ mean that $X$ and $Y$ are mm-isomorphism classes of pm-spaces. Hence, $X \times_p Y = Y \times_p X$ as classes.
So, we have $P \otimes_p Q = Q \otimes_p P$, that is, (1) holds.
(2) is clearly true.
(3) follows from (2), because $\{\ast\} \subset P$.
(4) and (5) are trivial.
\end{proof}

Due to this lemma, $\Pi$ is a commutative monoid with respect to the $\ell_p$-product, where the identity is $\{\ast\}$.
Its Grothendieck group is trivial, because $P \otimes_p \cal X = \cal X$ for every $P \in \Pi$.

Let us consider a property for pyramids which is regarded as the infinite product of itself.
\begin{definition}\label{def:idempotent}
Let $1 \le p \le \infty$.
We say that a pyramid $P$ is $\ell_p$-{\it idempotent} if $P \otimes_p P = P$ holds.
\end{definition}

\begin{proposition} \label{prop:idempotent}
For a pyramid $P$ and $p \in [1,\infty]$, the following conditions are mutually equivalent.
\begin{enumerate}
\item $P$ is $\ell_p$-idempotent, i.e., $P \otimes_p P = P$.
\item $P \otimes_p P \subset P$.
\item For every $X,Y \in P$, $X \times_p Y \in P$.
\item $P_p^\infty =P$.
\end{enumerate}
\end{proposition}
\begin{proof}
It is clear that (4) $\Rightarrow$ (1) $\Rightarrow$ (2) $\Rightarrow$ (3).
Let us show (3) $\Rightarrow$ (4).
We take $X \in P_p^{\otimes \infty}$.
Then, there exists $X_k \in \bigcup_{n=1}^\infty P_p^{\otimes n}$ such that $\square(X, X_k) \to 0$ as $k \to \infty$.
Let $n_k \in \mathbb N$ with $X_k \in P_p^{n_k}$.
Then, there exist $A_1^{(k)}, A_2^{(k)}, \dots, A_{n_k}^{(k)} \in P$ such that
\[
\square(X_k, A_{1}^{(k)} \times_p A_{2}^{(k)} \times_p \cdots \times_p A_{n_k}^{(k)}) \to 0 \text{ as $k \to \infty$.}
\]
By (3), $A_1^{(k)} \times_p \cdots \times_p A_{n_k}^{(k)} \in P$.
Therefore, we have
\begin{align*}
\square(X,P) &\le \square (X, A_1^{(k)} \times_p \cdots \times_p A_{n_k}^{(k)}) \\
&\le \square(X_k, A_1^{(k)} \times_p \cdots \times_p A_{n_k}^{(k)}) + \square(X_k,X)
\end{align*}
which converges to zero.
Hence, $X \in P$.
This implies (4).
\end{proof}

\begin{remark} \label{rem:delta-discrete}
The infinite $\ell_p$-product of any discrete spaces is trivial, due to works of Ozawa and Shioya, if $p< \infty$.
Let $\delta>0$ and $X$ a $\delta$-discrete pm-space, that is, for every two distinct points $x,x' \in X$, we have $d_X(x,x') \ge \delta$.
Suppose that $\# X \ge 2$.
Due to \cite[Theorem 1.2]{OS:est} and \cite[Theorem 1.1]{OS:lim}, we have
\[
\ObsDiam(X_p^\infty;-\kappa) = \infty
\]
for every $\kappa \in (0,1)$ and $1 \le p < \infty$.
Therefore, by {\cite[Lemmas 6.3 and 6.6]{OS:lim}},
we have $X_p^\infty = \cal X$.

We say that a pm-space $X$ is {\it mm-disconnected} if
$X$ admits a decomposition $X = A \cup A'$ into closed subsets with positive measures such that $d_X(A, A') > 0$.
It is clear that $X$ is {mm-disconnected} if and only if $X$ dominates a two-points pm-space.
So, if $X$ is {mm-disconnected}, then $X_p^\infty = \cal X$ for $1 \le p < \infty$.
\end{remark}

\begin{remark}
Note that the infinite $\ell_\infty$-product is not trivial, in general.
For instance, if $X \ne \ast$ has finite diameter, then $\diam\, X_\infty^n = \diam\, X$.
Hence, $\{\ast \} \ne X_\infty^\infty \ne \cal X$.

Furthermore, it is clear that $(\Gamma^1)_{{\infty}}^n$ is dominated by $\Gamma^n = (\Gamma^1)_{{2}}^n$.
So, we have $(\Gamma^1)_{{\infty}}^\infty \subset \Gamma^\infty$.
\end{remark}

\begin{example}
We give a non-trivial construction of $\ell_p$-idempotent pyramid.
Let us fix a pm-space $X$.
Due to \cite{NS:group}, the group $I(X)$ of isometries which preserves $m_X$ is compact with respect to the Ky Fan distance.
Let us fix a closed subgroup $G$ of $I(X)$.
Then, $G$ acts $X^n_p$ diagonally.
It is clear that the sequence $\{X^n_p/G\}_n$ of quotient pm-spaces is monotone nondecreasing in the Lipschitz order.
Hence, we have the limit
\[
X^\infty_p/G := \overline{\bigcup_{n=1}^\infty \cal P(X^n_p/G)}.
\]
Furthermore, we note that for $m,n \ge 1$, a canonical map
\[
X^{n+m}_p/G \to (X^n_p/G) \times_p (X^m_p/G)
\]
is a domination.
Therefore, we know that $X^\infty_p/G$ is $\ell_p$-idempotent.

In particular, $\Gamma^\infty/\mathbb Z_2$ and $\Gamma^\infty/S^1$ obtained in \cite{S:sphere} and \cite{ST} are $\ell_2$-idempotent pyramids.
\end{example}

Here are fundamental properties of idempotent pyramids:
\begin{proposition} \label{prop:idempotent02}
Let $P$ and $Q$ denote pyramids.
\begin{enumerate}
\item If both $P$ and $Q$ are $\ell_p$-idempotent, then so is $P \otimes_p Q$.
\item If both $P$ and $Q$ are $\ell_p$-idempotent, then $P \cap Q$ is an $\ell_p$-idempotent pyramid.
\item If $P$ is $\ell_p$-idempotent and $Q \subset P$, then $P \otimes_p Q=P$.
\item The set of all $\ell_p$-idempotent pyramids is closed in $\Pi$.
\end{enumerate}
\end{proposition}
\begin{proof}
(1) is clear.
(2) follows from Proposition \ref{prop:idempotent}.
Let us prove (3).
Since $Q \subset P$, we have $P \subset P \otimes_p Q \subset P \otimes_p P = P$.

We prove (4).
Let us take $\ell_p$-idempotent pyramids $P_n$ converging to $P$ in the weak Hausdorff sense.
For $X, Y \in P$, there exist $X_n, Y_n \in P_n$ such that $\square(X_n,X) + \square(Y_n,Y) \to 0$ as $n \to \infty$.
By \cite{K:product}, we have $\square(X_n \times_p Y_n, X \times_p Y) \to 0$.
Therefore, $X \times_p Y \in P$.
So, by Proposition \ref{prop:idempotent}, $P$ is $\ell_p$-idempotent.
\end{proof}

Remark that the intersection of two pyramids is not a pyramid in general.

\begin{example} \label{ex:diam}
Let $\delta \in [0, \infty]$.
Since $\diam : \cal X \to [0,\infty]$ is lower semicontinuous in the box topology, the set
\[
\cal X^{\delta} = \left\{ X \in \cal X \mid \diam\, X \le \delta \right\}
\]
is box-closed.
Since the property (3) in Proposition \ref{prop:idempotent} holds for $\cal X^\delta$, we know that $\cal X^{\delta}$ is an $\ell_\infty$-idempotent pyramid.
Note that $\cal X^{0} = \{\ast \}$ and $\cal X^{\infty} = \cal X$.
Furthermore, $\delta \cal X^{1}= \cal X^{\delta}$ holds and $\{ \cal X^{\delta}\}_{0 \le \delta \le \infty}$ is an injective continuous curve in $\Pi$ by Remark \ref{rem:curve} later.
\end{example}

\subsection{Scale invariant pyramids}
For a pyramid $P$, as {\eqref{eq:tP}}, we can define $t P$ for $t > 0$.
\begin{definition}
We say that $P$ is {\it scale invariant} if $t P = P$ holds for every $t > 0$.
\end{definition}

By using \eqref{eq:tX_box}, we immediately have:
\begin{proposition}
For example, $\infty P$ is scale invariant for each $P \in \Pi$.
Moreover, the following conditions are equivalent:
\begin{itemize}
\item $P$ is scale invariant.
\item There exists $t > 0$ with $t \ne 1$ such that $tP=P$.
\item $P=\infty P$.
\item For every $X \in P$ and $t > 0$, we have $t X \in P$.
\end{itemize}

If $F : \Pi \to [0,\infty]$ is monotone and homogeneous, then clearly $F(P) = \infty$ or $F(P) = 0$ for a scale invariant pyramid $P$.

Clearly, $\{\ast\}$ and $\cal X$ are scale invariant.
In Section \ref{sec:atom}, we will 
construct a family of scale invariant pyramids with infinite topological dimension.
\end{proposition}

\begin{remark} \label{rem:curve}
Due to \cite{K:mtf}, for each $P \in \Pi$, the map $(0,\infty) \ni t \mapsto t P \in \Pi$ is continuous.
Note that if $P$ is not scale invariant, then $\{t P\}_{0<t \le \infty}$ is an injective continuous curve.
\end{remark}

\begin{remark} \label{rem:inftyI=cal_X}
Recall that $\infty \Gamma^\infty = \cal X$ by the convention \eqref{eq:infty_Gamma=cal_X}.
Moreover, due to \cite{OS:lim}, we obtain that
\[
t \Gamma^\infty = \Gamma_{t^2}^\infty \to \{\ast \} \text{ as } t \to 0+.
\]
So, the injective continuous curve $\{ \Gamma_{\sigma^2}^\infty \}_{0 \le \sigma \le \infty}$ has endpoints $\{\ast \}$ and $\cal X$.
Here, we regard $\Gamma_0^\infty$ and $\Gamma_\infty^\infty$ as $\{\ast\}$ and $\cal X$, respectively.
\end{remark}

\begin{proposition} \label{prop:set_of_scale_invariant_pyramids_is_closed}
If $P$ and $Q$ are scale invariant pyramids, then so is $P \otimes_p Q$ for every $p \in [1, \infty]$.
The set of all scale invariant pyramids is closed in $\Pi$.
\end{proposition}
\begin{proof}
The first statement is clear.
The second one follows from \eqref{eq:tX_box}.
\end{proof}

\section{General theory of invariants} \label{sec:invariant}

This section is devoted to express general theory of numerical invariants for pm-spaces and pyramids.
\subsection{General construction}
We consider invariants for pyramids given by the following general way.
If $F : \cal X \to [0,\infty]$ is monotone, then its standard extension
\[
\Pi(F) : \Pi \to [0,\infty]
\]
is defined by
\begin{equation} \label{eq:extension}
\Pi(F)(P) := \sup_{X \in P} F(X).
\end{equation}
Furthermore, we often write $\Pi(F) = F$, when there is no confusion.
Note that $\Pi(F)$ is also monotone on $\Pi$.
Furthermore, we have:
\begin{lemma} \label{lem:general_lemma}
Let $F : \cal X \to [0,\infty]$ be a monotone functional.
Then, the following three conditions are equivalent:
\begin{enumerate}
\item $F$ is lower semicontinuous in the box topology on $\cal X$.
\item $F$ is lower semicontinuous in the concentration topology on $\cal X$.
\item $\Pi(F)$ is lower semicontinuous in the weak Hausdorff topology on $\Pi$.
\end{enumerate}
\end{lemma}
\begin{proof}
The implications (3) $\Rightarrow$ (2) $\Rightarrow$ (1) are clear.
We prove that (1) implies (3).
Let us take a sequence $\{P_n\}_n \subset \Pi$ weakly converging to $P$ as $n \to \infty$.
For each $X \in P$, there exists $X_n \in P_n$ with $\square(X_n, X) \to 0$.
By (1), we have
\[
\liminf_{n \to \infty} \Pi(F)(P_n) \ge \liminf_{n \to \infty} F(X_n) \ge F(X).
\]
Since $X$ is an arbitrary element of $P$, we obtain (3).
\end{proof}
From now on, we say that a monotone functional $F : \cal X \to [0,\infty]$ is {\it lower semicontinuous} if $F$ satisfies one of three equivalent conditions in Lemma \ref{lem:general_lemma}.

We are going to see that our extension procedure \eqref{eq:extension} is canonical (Corollary \ref{thm:unique_extension}).
\begin{lemma} \label{lem:lower_approximation}
Let $G : \Pi \to [0,\infty]$ be a monotone functional which is lower semicontinuous with respect to the weak Hausdorff topology.
Let $P \in \Pi$ and $\{X_n\}_n \subset \cal X$ satisfy
that every $X_n$ is contained in $P$ and $\cal P X_n$ weakly converges to $P$ as $n \to \infty$.
Then, we have
\[
\lim_{n \to \infty} G(\cal PX_n) = G(P).
\]
\end{lemma}
\begin{proof}
Since $G$ is monotone, we have $G(\cal PX_n) \le G(P)$.
So, we obtain $\limsup_{n \to \infty} G(\cal PX_n) \le G(P)$.
Furthermore, since $G$ is lower semicontinuous, $\liminf_{n \to \infty} G(\cal PX_n) \ge G(P)$.
\end{proof}

\begin{corollary} \label{thm:unique_extension}
Let $F : \cal X \to [0,\infty]$ be a monotone functional which is lower semicontinuous.
Let $G : \Pi \to [0,\infty]$ be a monotone functional which is lower semicontinuous in the weak Hausdorff topology.
Suppose that $G$ is an extension of $F$.
Then, $G$ coincides with the standard extension $\Pi(F)$ of $F$ to $\Pi$ as in \eqref{eq:extension}.
\end{corollary}
\begin{proof}
Let us fix $P \in \Pi$ and take an approximation $\{X_n\}_{n=1}^\infty$ of $P$.
Then, by Lemma \ref{lem:lower_approximation}, we have
\[
G(P) = \lim_{n \to \infty} F(X_n).
\]
Clearly, $\Pi(F)$ satisfies the same assumption as $G$, we also have
\[
\Pi(F)(P) = \lim_{n \to \infty} F(X_n).
\]
This completes the proof.
\end{proof}

\begin{corollary} \label{cor:limit}
Let $F$ be monotone and lower semicontinuous.
Then, we have
\[
F(X_p^\infty) = \lim_{n \to \infty} F(X_p^n) \text{ and } F(\infty X) = \lim_{t \to \infty} F(tX)
\]
for each $X \in \cal X$ and $p \in [1, \infty]$.
\end{corollary}

\begin{remark}
In \cite{OS:lim}, Ozawa and Shioya proved that $\ObsDiam(\,\cdot\,;-\kappa)$ is lower semicontinuous on $\Pi$ for each $\kappa \in (0,1)$.

On the other hand, for each $\kappa_1, \dots, \kappa_N > 0$ with $\sum_i \kappa_i < 1$, they showed that $\Sep(\,\cdot\,;\kappa_1, \dots, \kappa_N)$ is upper semicontinuous on $(\cal X, \square)$, on $(\cal X, d_\conc)$ and on $(\Pi, \rho)$.
Furthermore, we do not care whether $\Sep(P;\kappa_1, \dots, \kappa_N)$ coincides with $\sup_{X \in P} \Sep(X;\kappa_1, \dots, \kappa_N)$ or not (compare Remark \ref{rem:limit_formula}).
\end{remark}

We now express how to define monotone lower semicontinuous invariants on $\Pi$, in a general framework.

\begin{proposition} \label{prop:general_prop}
Let $\{ A_\lambda \}_{\lambda \in [0, \infty]}$ denote a family of downward-closed and box-closed subsets of $\cal X$.
Suppose that $A_\lambda \subset A_{\lambda'}$ holds for every $\lambda \le \lambda'$.
Then,
\[
F(X) := \inf \left\{ \lambda \in [0, \infty] \,\middle|\, X \in A_\lambda \right\}
\]
for $X \in \cal X$, is a monotone lower semicontinuous invariant on $\cal X$.
Furthermore, the canonical extension of $F$ is represented as
\[
F(P) = \inf \left\{ \lambda \mid P \subset A_{\lambda} \right\}
\]
for $P \in \Pi$.
\end{proposition}
\begin{proof}
The monotonicity of $F$ is trivial.
To prove the lower semicontinuity of $F$, it suffices to show that $O:= \{X \in \cal X \mid F(X) > t \}$ is open in the box topology for each $t \in [0,\infty]$.
Let us take $X \in O$.
Then, there exists $\epsilon > 0$ such that $X \not\in A_{\lambda}$ for each $\lambda < t + \epsilon$.
Let us take $\lambda \in (t+\epsilon/2, t + \epsilon)$ and set $U_\lambda := \{Y \mid Y \not\in A_\lambda\}$.
For each $Y \in U_\lambda$, we have $F(Y) \ge t +\epsilon/2$.
Hence, $U_\lambda \subset O$.
Since $U_\lambda$ is box-open and $X \in U_\lambda$, we know that $O$ is open.
This completes the proof of the first statement.
The last statement follows from Corollary \ref{thm:unique_extension}.
\end{proof}

\subsection{Homogeneous invariants}
We also have a complete characterization of monotone lower semicontinuous homogeneous invariants.
Let us consider the following two sets:
\begin{itemize}
\item the set $\cal I$ of all $1$-homogeneous monotone lower semicontinuous invariants on $\cal X$, and
\item the set $\cal A$ of all downward-closed box-closed subsets in $\cal X$.
\end{itemize}
For $A \subset \cal X$, we are going to denote $t A = \{t X \in \cal X \mid X \in A \}$.
\begin{lemma} \label{lem:tA}
If a box-closed subset $A$ of $\cal X$ satisfies that $t A \subset A$ for every $0<t<1$, then $A = \bigcap_{\lambda > 1} \lambda A$.
\end{lemma}
\begin{proof}
Let $A' := \bigcap_{\lambda > 1} \lambda A$.
From the assumption, we have $A' \supset A$.
For $X \in A'$, we have $\lambda^{-1}X \in A$ for every $\lambda > 1$.
Then, $\lambda^{-1}X$ box-converges to $X$ as $\lambda \to 1$ (see \cite{K:mtf}).
Therefore, we have $X \in A$.
This completes the proof.
\end{proof}
\begin{theorem} \label{thm:invariant}
For each $A \in \cal A$, we set
\[
F_A(X) = \inf \left\{ t > 0 \,\middle|\, X \in t A \right\}
\]
for $X \in \cal X$.
Then $F_A \in \cal I$.
Moreover, this correspondence is bijective.
\end{theorem}
\begin{proof}
From the construction, we know that $F_A$ is $1$-homogeneous. 
Hence $F_A \in \cal I$ by Proposition \ref{prop:general_prop}.
We note that, by Lemma \ref{lem:tA}, 
\begin{equation*} 
	F_A(X) \le 1 \iff X \in A
\end{equation*}
holds for all $X \in \cal X$.
This implies that the map $\cal A \ni A \mapsto F_A \in \cal I$ is injective.

Let us prove that the correspondence is surjective.
For $F \in \cal I$, we set $A = \left\{ X \in \cal X \,\middle|\, F(X) \le 1 \right\}$.
Since $F$ is lower semicontinuous, $A$ is box-closed.
By the monotonicity of $F$, $A$ is downward-closed.
Hence, $A \in \cal A$.
Finally, we prove $F = F_A$.
Note that
\begin{equation*}
\left\{ X \in \cal X\,\middle|\, F(X) \le t \right\} = t A
\end{equation*}
for every $t > 0$.
By Lemma \ref{lem:tA}, we have $F_A(X) \le t$ if and only if $F(X) \le t$.
Hence, we obtain $F_A(X) = F(X)$.
This completes the proof.
\end{proof}

\begin{remark}
\begin{enumerate}
\item 
From the proof of Theorem \ref{thm:invariant}, we know that $A$ is box-closed if and only if $A$ is closed in $d_\conc$, for any downward-closed subset $A$ of $\cal X$.

\item Since $\Pi \subset \cal A$ is a proper inclusion, we know that the set $\cal I$ of invariants is ``strictly'' bigger than the set $\Pi$ of pyramids.
\end{enumerate}
\end{remark}

\begin{example}
Let $P \in \Pi$.
If $P$ is scale invariant, then we have
\[
F_P(Q) = \left\{
\begin{aligned}
& 0 && \text{ if } Q \subset P, \\
& \infty && \text{ if } Q \not\subset P.
\end{aligned}
\right.
\]
If $P$ is not scale invariant, then $F_P(P)=1$.
\end{example}

\begin{example}
We can define the {\it diameter} of a pyramid $P$ by
\[
\mathrm{diam}\, P := \sup_{X \in P} \mathrm{diam}\, X.
\]
This is nothing but $F_{\cal X^1}(P)$.
Here, $\cal X^1$ is defined in Example \ref{ex:diam}.

Note that $\mathrm{diam}\, P = \lim_{\kappa \to 0+} \ObsDiam(P;-\kappa)$ holds.
\end{example}

\begin{remark}
If a pyramid $P$ is scale-invariant, then $F(P) \in \{0, \infty\}$
for every monotone homogeneous invariant $F$ on $\Pi$.
By Kazukawa, Nakajima and Shioya (\cite{KNS}), the converse statement is known to be true, that is, if $P$ satisfies $F(P) \in \{0, \infty\}$ for suitable monotone homogeneous invariant $F$, then $P$ is scale-invariant.
\end{remark}

There exist inhomogeneous invariants. For example, so are $\min\{F,1\}$ and $F+1$ for $F \in \mathcal I$.
The following is a more non-trivial one:

\begin{example} \label{ex:L_1(X)}
Let us consider the functional given by
\[
\mathcal X \ni X \mapsto d_\conc(X,\ast) \in [0,1].
\]
It is clearly inhomogeneous.
This is monotone and coincides with $\diam(\mathcal L_1(X), d_\mathrm{KF}^{m_X})$, where the definition of $\mathcal L_1(X)$ was given in Subsection \ref{subsec:basic}.
\end{example}

\begin{example}
For two functionals $F, F' : \cal X \to [0, \infty]$, we write $F \ge F'$ if $F(X) \ge F'(X)$ holds for every $X \in \cal X$.
Suppose that $F$ is monotone $1$-homogeneous.
Then, the lower semicontinuous envelope of $F$ defined by
\[
\underline{F}(X) := \sup \left\{ F'(X) \mid F' \in \cal I \text{ with } F \ge F' \right\}
\]
is a lower semicontinuous monotone $1$-homogeneous functional, which is maximal one among all such $F'$ with $F' \le F$.
Note that the supremum is actually maximum.

We express another formulation.
Let $A = \left\{ X \in \cal X \mid F(X) \le 1\right\}$.
It is downward-closed, since $F$ is monotone.
So, the box-closure $\overline{A}$ is also downward-closed.
Then, we have
\[
\underline{F} = F_{\overline{A}}.
\]
Indeed, from the definition, $\{ \underline{F} \le 1 \}$ coincides with $\overline{A}$.
\end{example}

\section{Examples of invariants}
\label{sec:distinguish}

In this section, we introduce concrete examples of invariants which are monotone, homogeneous and lower semicontinuous in the sense of Section \ref{sec:invariant}.
More precisely, we extend some important and classical invariants for pm-spaces in geometry and measure theory to invariants of pyramids.

\subsection{Variance}
We recall the definition of the variance of pm-spaces following Nakajima-Shioya \cite{NS}.
Let $X$ be a pm-space.
For a measurable function $f : X \to \mathbb R$, the variance of $f$ (over $X$) is defined by
\[
V_X(f) := \frac{1}{2} \int_{X\times X} (f(x)-f(y))^2\,dm_X^{\otimes 2}(x,y).
\]
This is nothing but the usual variance of a random variable $f$.
For a pm-space $X$,
we set
\[
V(X) := \sup_{ f \in \Lip_1(X) } V_X(f)
\]
which is called the {\it variance} of $X$.
Note that $V(X)$ is called the {\it observable variance} in \cite{NS}.

It is clear that the variance $V$ is a $2$-homogeneous monotone invariant on $\cal X$. 

We give a fundamental and optimal estimate between the variance and the diameter:
\begin{proposition} \label{prop:variance_vs_diam}
For a pm-space $X$, we have
\begin{equation*}
V(X) \le \frac{(\diam\, X)^2}{4}.
\end{equation*}
\end{proposition}
\begin{proof}
We may assume $\delta := \diam\, X < \infty$.
Let us take $f \in \Lip_1(X)$ and set
\[
\tilde f = f - \inf_{X} f.
\]
Then, $0 \le \tilde f \le \delta$.
Setting $\mu = E_X(\tilde f)$, we have
\begin{align*}
V_X(f) &= V_X(\tilde f) = E_X(\tilde f^2) - E_X(\tilde f)^2 \\
&\le \delta E_X(\tilde f) - \mu^2 = \frac{\delta^2}{4} - \left(
\frac{\delta}{2} - \mu
\right)^2 \le \frac{\delta^2}{4}.
\end{align*}
This implies $V(X) \le \delta^2/4$.
\end{proof}
By considering a two-points space, we know that
the inequality of the proposition is optimal.

We will prove the following later.
\begin{proposition} \label{prop:variance}
$V$ is monotone, lower semicontinuous and $2$-homogeneous on $\cal X$.
\end{proposition}

We also consider the {\it variance} of pyramids defined by
\begin{equation} \label{eq:variance_for_pyramid}
V(P) := \sup_{X\in P} V(X)
\end{equation}
for $P \in \Pi$.

\subsubsection{Fundamental properties of the variance}

\begin{proposition} \label{prop:finite_variance}
For $X \in \cal X$, we have
\[
V(X) = \sup \left\{V_X(f) \,\middle|\, f \in \Lip_b(X) \cap \Lip_1(X) \right\}.
\]
In particular,
\[
V(X) = \sup \left\{ V_X(f) \,\middle|\, f\in \Lip_1(X) \text{ with } V_X(f) < \infty \right\}.
\]
\end{proposition}
\begin{proof}
For each $f \in \Lip_1(X)$, we consider the truncation $f_r$ of $f$ defined by
\begin{equation} \label{eq:truncation}
f_r(x) := \max \{ \min \{f(x), r\},-r\}
\end{equation}
for $r > 0$.
Then, $f_r \in \Lip_b(X) \cap \Lip_1(X)$ and
$(f_r)_\# m_X$ weakly converges to $f_\# m_X$ as $r \to \infty$.
Therefore we have
\[
\liminf_{r \to \infty} V_X(f_r) \ge V_X(f).
\]
Hence we obtain the conclusion.
\end{proof}

\begin{corollary}
For a pyramid $P$, we have
\[
V(P) = \sup \left\{ V(X) \mid X \in P \text{ with } V(X) < \infty \right\}.
\]
\end{corollary}
\begin{proof}
When $V(P) < \infty$, there is nothing to prove.
Suppose $V(P) = \infty$.
Let us take $R > 0$ and $X \in P$ with $V(X) > R$.
Then, there exists $f \in \Lip_b(X) \cap \Lip_1(X)$ with $V(X) > V_X(f) > R$.
Let us take $r > \sup_{X} |f|$
and define a new metric $d_r$ on $X$ as
\[
d_r(x,x') = \min \{ 2r, d(x,x')\}
\]
for $x,x' \in X$.
Then, $f$ is also $1$-Lipschitz with respect to $d_r$.
Clearly, $X_r =(X,d_r, m_X) \prec X$.
Since by {Proposition \ref{prop:variance_vs_diam}},
$V(X_r) \ge V_{X_r}(f) = V_X(f)$ and $V(X_r) \le {r^2} < \infty$, we obtain the conclusion.
\end{proof}

For the $\ell_p$-product of pm-spaces, there is a trivial bound:
\begin{proposition}
Let $1 \le p \le \infty$.
For $X, Y \in \cal X$, we have
\[
V(X \times_p Y) \le V(X) + V(Y).
\]
In particular, $V(X_p^n) \le n V(X)$ for every $n \ge 1$.
\end{proposition}
\begin{proof}
Let us take $f \in \Lip_b(X \times_p Y) \cap \Lip_1(X \times_p Y)$.
For each $x \in X$, we set
\[
g(x) = E_Y(f(x,\cdot)) = \int_Yf(x,y)\,dm_Y(y).
\]
It is clear that $g \in \Lip_b(X) \cap \Lip_1(X)$.
Then, we have
\begin{align*}
V_{X \times_p Y}(f) &= E_{X \times_p Y}(f^2) - E_{X \times_p Y}(f)^2 \\
&= E_X \left\{ E_Y(f(x,\cdot)^2) - \left\{ E_Y(f(x,\cdot)) \right\}^2 \right\} \\
& \hspace{1.4em} + E_X(g^2) - E_X (g)^2 \\
&\le E_X(V(Y)) + V(X).
\end{align*}
This implies the conclusion.
\end{proof}

\subsubsection{Variants of variance invariants}

For $p \in [1,\infty)$ and a pm-space $X$, the {\it $p$-variance} of a function $f \in L^1(m_X)$ is defined by
\[
V^{(p)}(f) := V_X^{(p)}(f) := \int_{X} \left| f - E_X(f)\right|^p\,dm_X = E_X \left( \left|f- E_X(f) \right|^p \right).
\]
Using this, we define the {\it $p$-variance} of $X$ by
\[
V^{(p)}(X) := \sup_{f \in \Lip_1(X) \cap \Lip_b(X)} V_X^{(p)}(f).
\]
Note that $V(X) = V^{(2)}(X)$ by Proposition \ref{prop:finite_variance}.
\begin{proposition} \label{prop:p-variance}
$V^{(p)}$ is lower semicontinuous in the box topology.
\end{proposition}
\begin{proof}
Let us take $f \in \Lip_b(X) \cap \Lip_1(X)$.
Let $X_n \in \cal X$ box-converge to $X$ as $n \to \infty$.
Due to Lemma \ref{lem:mG_convergece}, there exists a complete separable metric space $Y$ admitting isometric embeddings $X_n, X \hookrightarrow Y$ such that $m_{X_n}$ converges to $m_X$ weakly in $\mathscr P(Y)$, via the embeddings.
Let $f$ denote a bounded Lipschitz extension of $f$ to $Y$.
Then, we have $E_{X_n}(f) \to E_X(f)$, so that $V_{X_n}^{(p)}(f) \to V_X^{(p)}(f)$.
Hence, we have
\[
\liminf_{n \to \infty} V^{(p)}(X_n) \ge \liminf_{n \to \infty} V_{X_n}^{(p)}(f) = V_X^{(p)}(f).
\]
Since $f$ is arbitrary, we obtain the conclusion.
\end{proof}
\begin{proof}[Proof of Proposition \ref{prop:variance}]
This follows from Propositions {\ref{prop:finite_variance} and \ref{prop:p-variance}}.
\end{proof}

\begin{corollary} \label{cor:infinite_variance}
Let $p,q \in [1, \infty)$.
If a pm-space $X$ is {mm-disconnected}, then we have
\[
\lim_{n \to \infty} V^{(q)}(X_p^n) = \infty.
\]
\end{corollary}
\begin{proof}
By {the last paragraph of Remark \ref{rem:delta-discrete}},
we have $\cal P X_p^n \to \cal X$ weakly as $n \to \infty$.
Then by Corollary \ref{cor:limit}, we have $\lim_{n \to \infty} V^{(q)}(X_p^n) = V^{(q)}(\cal X) = \infty$.
\end{proof}

\begin{example}
There is a pm-space $X$ satisfying that $V(X_p^n) < \infty$ for every $n \ge 1$ and $\lim_{n \to \infty} V(X_p^n) = \infty$.
For instance, {by Corollary \ref{cor:infinite_variance}}, if $\# X = 2$, then $X$ satisfies this property.
\end{example}

Let us set $v^{(p)}(X) := \left( V^{(p)}(X) \right)^{1/p}$, which is called the {\it $p$-deviation} of $X$.
Furthermore, we set
\[
v_X^{(\infty)}(f) := \left\| f - E_X(f) \right\|_{L^\infty(m_X)}
\]
for $f \in L^1(m_X)$, and
\[
v^{(\infty)}(X) := \sup_{f \in \Lip_1(X) \cap \Lip_b(X)} v_X^{(\infty)}(f).
\]
Then, $v^{(p)}$ is lower semicontinuous in the box topology, for every $p\in [1,\infty]$.
\begin{corollary}
For $p \in [1, \infty]$, $v^{(p)}$ is a monotone $1$-homogeneous lower semicontinuous invariant.
\end{corollary}

\subsection{Poincar\'e constants}

In this subsection, we give an invariant of pm-spaces, called Poincar\'e constant.
A main result states that the Poincar\'e constant is lower semicontinuous in the box topology (Theorem \ref{thm:(p,q)-P_is_lsc}).

\begin{definition} \label{def:2,2-PI}
Let $C \in [0,\infty]$ and $X$ a pm-space.
We say that $X$ satisfies {\it $(2,2)$-Poincar\'e inequality} for $C$ if
\begin{equation} \label{eq:(2,2)-PI}
V_X(f) \le C^2 \int_X \lip_a(f)^2\,dm_X
\end{equation}
holds for every $f \in \Lip_b(X)$, where $\lip_a(f)$ is defined in \eqref{def:asymp-Lip}.
The infimum of such a $C$ is denoted by $C_{2,2}(X)$ and {is called} the {\it $(2,2)$-Poincar\'e constant} of $X$.
Here, we use the convention that $\inf \emptyset = \infty$.
\end{definition}

	From the definition, we have
	\begin{equation} \label{eq:V_vs_PI}
		\sqrt{V(X)} \le C_{2,2}(X)
	\end{equation}
	for every pm-space $X$.

Note that, from the definition, $C_{2,2}$ is $1$-homogeneous and monotone.

\subsubsection{Variants of Poincar\'e inequality}

Let us give a slight generalization of Definition \ref{def:2,2-PI}.

\begin{definition} \label{def:(p,q)-PI}
Let $p,q \in [1,\infty)$.
For $C \in [0,\infty]$ and a pm-space $X$, we say that $X$ satisfies {\it $(p,q)$-Poincar\'e inequality} if
\begin{equation} \label{eq:(p,q)-PI}
V^{(p)}_X(f)^{1/p} \le C \left( \int_X \lip_a(f)^q\,dm_X \right)^{1/q}
\end{equation}
holds for every $f \in \Lip_b(X)$.
Furthermore, we denote by $C_{p,q}(X)$ the infimum of $C$ satisfying \eqref{eq:(p,q)-PI}, which is called the $(p,q)$-{\it Poincar\'e constant} of $X$.
\end{definition}

Note that $C_{p,q}$ is monotone.
By the definition, we have
\[
C_{p,q}(X) \ge C_{r,s}(X)
\]
for every $r \le p$ and $q \le s$.

One of main results of our paper is:
\begin{theorem} \label{thm:(p,q)-P_is_lsc}
When {$1 < p \le q$ or $p=1<q$}, $C_{p,q} : \cal X \to [0,\infty]$ is a monotone $1$-homogeneous lower semicontinuous invariant.
\end{theorem}
To prove this theorem, we give another formulation of Poincar\'e inequality, {using Sobolev functions on singular spaces}.

\subsubsection{Sobolev spaces}
Let $X$ be a pm-space and {$1 < q < \infty$}.
Following \cite{G:on_the} and \cite{GP}, we {define} the Sobolev space $W^{1,q}_{\ast,a}(X,d_X,m_X)$ on $X$ {as follows}.

\begin{definition}[\cite{G:on_the}, \cite{GP}]
For $f \in L^q(m_X)$, we define the $q$-{\it energy} of $f$ by
\[
\mathsf{E}_{\ast,a}^{(q)}(f) := \frac{1}{q} \inf \left\{
\liminf_{n \to \infty} \int_X \lip_a(f_n)^q\,dm_X \,\middle|\,
\begin{aligned}
&\{f_n\}_n \subset \Lip(X) \text{ such that } \\
& f_n \to f \text{ strongly in } L^q(m_X)
\end{aligned}
\right\}.
\]
Then we set
\[
W^{1,q}_{\ast,a}(X) := \left\{ f \in L^q(m_X) \mid \mathsf E_{\ast,a}^{(q)}(f) < \infty \right\}
\]
which is called the {\it $(1,q)$-Sobolev space} of $X$.
\end{definition}
It is known that $W_{\ast,a}^{1,q}(X)$ is a Banach space in the norm \[
\|f\|_{1,q}= \left( \|f\|_q^q + q \,\mathsf E_{\ast,a}^{(q)}(f) \right)^{1/q}.
\]
Furthermore, a density function of the energy exists:
\begin{proposition}[\cite{G:on_the}, \cite{GP}]
For each $f \in W^{1,q}_{\ast,a}(X)$, there exists a function $|Df|_{\ast,a} \in L^q(m_X)$ such that
\[
\mathsf E_{\ast,a}^{(q)}(f) = \frac{1}{q} \int_X |Df|_{\ast,a}^q\,dm_X
\]
holds.
\end{proposition}
More precisely, such a $|Df|_{\ast,a}$ is unique, in a certain class of functions called asymptotic relaxed slopes of $f$ (see \cite{G:on_the}, \cite{GP}).
So, the density function $|Df|_{\ast,a}$ is called the {\it minimal asymptotic relaxed slope} of $f$.
In particular,
\begin{equation} \label{eq:|Df|_le_lip_a(f)}
|Df|_{\ast,a} \le \lip_a(f) \hspace{1em} m_X\text{-a.e. on } X.
\end{equation}
It is known that $|Df|_{\ast,a}$ has a nice approximation:
\begin{theorem}[\cite{G:on_the}, \cite{GP}]
For each $f \in W^{1,q}_{\ast,a}(X)$, there exists a sequence $\{f_n\}_n \subset \Lip(X) \cap L^q(m_X)$ such that
\begin{equation} \label{eq:approximation}
\|f_n-f\|_q + \|\lip_a(f_n) - |Df|_{\ast,a} \|_{q} \to 0
\end{equation}
as $n \to \infty$.
\end{theorem}

Using this theorem, we can give another formulation of Poincar\'e inequalities:
\begin{lemma} \label{prop:equiv_PI}
Let $C\in [0,\infty]$.
Then the following are equivalent.
\begin{enumerate}
\item $X$ satisfies $(p,q)$-Poincar\'e inequality for $C$.
\item For every $f \in \Lip_b(X)$, we have
\[
V_X^{(p)}(f)^{1/p} \le C (q \, \mathsf E_{\ast,a}^{(q)}(f))^{1/q}.
\]
\item For every $f \in W_{\ast,a}^{1,q}(X)$, we have
\[
V_X^{(p)}(f)^{1/p} \le C (q \, \mathsf E_{\ast,a}^{(q)}(f))^{1/q}.
\]
\end{enumerate}
\end{lemma}
\begin{proof}
By \eqref{eq:|Df|_le_lip_a(f)}, the implications (3) $\Rightarrow$ (2) $\Rightarrow$ (1) are clear.
We are going to prove (1) $\Rightarrow$ (3).
Suppose that (1) holds.
For each $f \in W_{\ast,a}^{1,q}(X)$, we obtain a sequence $\{f_n\}_n \subset \Lip(X) \cap L^q(m_X)$ satisfying \eqref{eq:approximation}.
We may assume that $f$ is bounded.
By considering the truncation 
\[
\max \left\{ \min \left\{f_n, \sup_X f \right\}, \inf_X f\right\}
\]
and using Mazur's lemma, we may assume that $f_n \in \Lip_b(X)$.
From the assumption, we have
\[
V_X^{(p)}(f_n)^{1/p} \le C \left( \int_X \lip_a (f_n)^q\,dm_X\right)^{1/q}.
\]
Taking the limit as $n \to \infty$ and using \eqref{eq:approximation}, we obtain
\[
V_X^{(p)}(f)^{1/p} \le C (q \, \mathsf E_{\ast,a}^{(q)}(f))^{1/q}.
\]
So, $X$ satisfies (3).
This completes the proof.
\end{proof}

\begin{proof}[Proof of Theorem \ref{thm:(p,q)-P_is_lsc}]
The monotonicity and the homogeneity of $C_{p,q}$ are clear from the definition.
We prove that $C_{p,q}$ is lower semicontinuous in the box topology.
To prove this, it is enough to show that the set of all $X \in \cal X$ with $C_{p,q}(X) \le C$ is closed in the box-topology for each $C \in[0, \infty)$.
Let us take a sequence $\{X_n\} \subset \cal X$ satisfying $(p,q)$-Poincar\'e inequality for a common constant $C$.
Suppose $\{X_n\}$ box converges to $X$.
By Lemma \ref{lem:mG_convergece},
we may assume that $X_n$ and $X$ are isometrically embedded in a complete separable metric space $Y$ such that $m_{X_n}$ weakly converges to $m_X$ in $\mathscr P(Y)$.
Let us fix $f \in \Lip_b(X)$.
By Proposition \ref{prop:Lip_ext}, there exists a bounded Lipschitz extension $\tilde f : Y \to \mathbb R$ of $f$.
Then, there exists a sequence $f_k \in \Lip(Y) \cap L^q(m_X)$, $k \in \mathbb N$, such that
\[
\| f_k -\tilde f \|_{L^q(m_X)} + \left\| \lip_a(f_k) - | D \tilde f |_{\ast,a} \right\|_{L^q(m_X)} \to 0
\]
as $k \to \infty$.
Let $g_k \in \Lip_b(Y)$ be given by the truncation
\[
g_k =
\max \left\{ \min \left\{ f_k, \sup_Y \tilde f \right\} , \inf_Y \tilde f\right\}.
\]
Then, we have
\[
|g_k - \tilde f| \le |f_k - \tilde f| \text{ on } Y.
\]
Hence,
\[
\| g_k - \tilde f\|_{L^r(m_X)} \to 0
\]
holds, for every $1 \le r \le q$.
In particular, we have
\[
\lim_{k \to \infty} V_{X}^{(p)}(g_k) = V_X^{(p)}(\tilde f).
\]
Since $g_k$ is the truncation of $f_k$, we have
\[
\lip_a (g_k) \le \lip_a (f_k) \text{ on } Y.
\]
Since each $X_n$ satisfies $(p,q)$-Poincar\'e inequality,
\[
V_{X_n}^{(p)}(g_k)^{1/p} \le C \left( \int_Y \lip_a(g_k)^q\,dm_{X_n} \right)^{1/q}
\]
holds.
Since $\lip_a(g_k)$ is upper semicontinuous and has an upper bound, {as $n \to \infty$},
we obtain
\[
V_X^{(p)}(g_k)^{1/p} \le C \left( \int_Y\lip_a(g_k)^q\,dm_X \right)^{1/q} \le C \left( \int_Y \lip_a(f_k)^q\,dm_X \right)^{1/q}.
\]
Taking $k \to \infty$, we have
\begin{align*}
V_X^{(p)}(\tilde f)^{1/p} &\le C \left( \int_Y |D \tilde f|_{\ast,a}^q\,dm_X \right)^{1/q} \\
&= C \left( \int_X|D \tilde f|_{\ast,a}^q\,dm_X \right)^{1/q} = C \left( q \,\mathsf E_{\ast,a}( \tilde f ) \right)^{1/q}.
\end{align*}
Since $\tilde f = f$ on $X$, we have
\[
\mathsf E_{\ast,a}(f) = \mathsf E_{\ast,a}( \tilde f ).
\]
Therefore, we know that $X$ satisfies $(p,q)$-Poincar\'e inequality for $C$ by Lemma \ref{prop:equiv_PI}.
This completes the proof.
\end{proof}

\begin{proposition} \label{prop:smooth_PI}
Let $X$ be a smooth pm-space. Then we have $C_{2,2}(X_2^\infty) = C_{2,2}(X)$.
\end{proposition}
\begin{proof}
The tensorization property {of $C_{2,2}$} for smooth spaces (\cite{BGL})
implies $C_{2,2}(X_2^n) = C_{2,2}(X)$ for all $n \ge 1$.
So, by the lower semicontinuity of $C_{2,2}$, we have
\[
C_{2,2}(X) \le \sup_n C_{2,2}(X_2^n) \le C_{2,2}(X_2^\infty) \le \liminf_{n \to \infty} C_{2,2}(X_2^n) = C_{2,2}(X).
\]
This completes the proof.
\end{proof}

\begin{proposition}
\label{prop:Poincare=infty}
If a pm-space $Y$ is {mm-disconnected} (see Remark \ref{rem:delta-discrete}), then
$C_{p,q}(Y) = \infty$ for every $p,q \in (1, \infty)$.
\end{proposition}
\begin{proof}
{We may assume that $Y$ is a two-points pm-space $Y=(\{0,1\}, d, m)$.}
Let us consider the distance function $d_0 : \{0,1\} \to \mathbb R$ from $0$.
Since
$Y$
is discrete, we have
\[
\lip_a (d_0) \equiv 0.
\]
On the other hand, we obtain
\begin{align*}
V_Y(d_0) &= E_Y(d_0^2) - E_Y(d_0)^2 \\
&= m_X(\{1\}) (\diam\, Y)^2 - \left( m_Y(\{1\}) \diam\,Y \right)^2 \\
&= (\diam\, Y)^2 m_Y(\{1\}) m_Y(\{0\}) > 0.
\end{align*}
Therefore, $C_{p,q}(Y) = \infty$.
This completes the proof.
\end{proof}

\subsection{Logarithmic Sobolev inequality}
Let us introduce one of the most famous invariants on $\cal X$ and extend it to $\Pi$.

\begin{definition} \label{def:LSI}
Let $X$ be a pm-space and $C \in [0,\infty]$.
We say that $X$ satisfies {\it logarithmic Sobolev inequality for} $C$ if
\[
\mathrm{Ent}_{m_X}(f^2) \le 2 C^2 \int_X \lip_a(f)^2\,dm_X
\]
for all $f \in \Lip_b(X)$.
Here, we use the convention $0 \log 0 = 0$.
The infimum of all such $C$'s is denoted by $C_\mathrm{LS}(X)$ and it is called the {\it logarithmic Sobolev constant} of $X$.
\end{definition}

By a direct calculation, $C_\mathrm{LS}$ is monotone and $1$-homogeneous.
Results similar to Lemma \ref{prop:equiv_PI} and Theorem \ref{thm:(p,q)-P_is_lsc} also hold for $C_\mathrm{LS}$.
\begin{proposition}
Let $X$ be a pm-space and $C \in [0,\infty]$.
The following properties are equivalent:
\begin{enumerate}
\item $X$ satisfies logarithmic Sobolev inequality for $C$;
\item for every $f \in W_{\ast,a}^{1,2}(X)$,
we have
\[
\mathrm{Ent}_{m_X}(f^2)
\le 2C^2 \int_X |Df|_{\ast,a}^2\,dm_X.
\]
\end{enumerate}
\end{proposition}

\begin{theorem} \label{thm:LS_is_lsc}
$C_\mathrm{LS}$ is lower semicontinuous on $\cal X$.
\end{theorem}

We omit proofs of the above two properties.
Furthermore, by considering a function $1+\epsilon f$ and taking $\epsilon \to 0$, we have
\begin{equation} \label{eq:Poincare_vs_LS}
C_{2,2} \le C_\mathrm{LS}.
\end{equation}
Then, we obtain:

\begin{proposition} \label{prop:LS_tensor}
If $X$ is a smooth pm-space, then $C_{\mathrm{LS}}(X_2^\infty) = C_{\mathrm{LS}}(X)$.
\end{proposition}
\begin{proof}
Since $X$ is smooth, we have $C_\mathrm{LS}(X^n_2) = C_\mathrm{LS}(X)$ by \cite[Proposition 5.2.7]{BGL}.
So, by Theorem \ref{thm:LS_is_lsc},
we obtain the conclusion.
\end{proof}

\subsection{Relation between invariants}

In this subsection, we discuss relations between our invariants and prove Theorem \ref{thm:invariant_of_smooth_space} and \ref{thm:CD_implies_FI}. 

\begin{remark} \label{rem:FI_relation}
We summarize relations among functional inequalities.
Let $X$ be a pm-space.
Then, we have
\[
V^{(p)}(X)^{1/p} \le C_{p,q}(X)
\]
for every $p,q \in [1, \infty)$, 
and
\[
C_{p,q}(X) \ge C_{r,s}(X)
\]
for every $r \le p $ and $q \le s$.
Furthermore, we recall the relation:
\[
C_{2,2}(X) \le C_\mathrm{LS}(X).
\]

Since the above invariants are extended to the space of pyramids, the corresponding relations also hold for pyramids.
\end{remark}

\begin{proof}[Proof of Theorem \ref{thm:invariant_of_smooth_space}]
This is a direct consequence of Propositions \ref{prop:smooth_PI} and \ref{prop:LS_tensor}, Remark \ref{rem:FI_relation} and the monotonicity of $V$.
\end{proof}

\begin{proof}[Proof of Theorem \ref{thm:CD_implies_FI}]
Let $\{X_n\}$ be a sequence of pm-spaces and $X_n$ is dominated by a CD$(K,\infty)$-space $Y_n$ for each $n$.
Suppose that $\{X_n\}$ weakly converges to a pyramid $P$.
Then, due to \cite[Equation (6.10)]{AGS}, 
$Y_n$ satisfies $C_\mathrm{LS}(Y_n) \le 1/K$ for every $n$.
Hence, we obtain $C_\mathrm{LS}(X_n) \le 1/K$.
By Theorem \ref{thm:LS_is_lsc}, we have $C_\mathrm{LS}(P) \le 1/K$.
Therefore, by Remark \ref{rem:FI_relation} for pyramids, we obtain
\[
\sqrt{V(P)} \le C_{2,2}(P)
\le C_\mathrm{LS}(P) \le 1/K.
\]
This completes the proof.
\end{proof}

\section{Applications}
\label{sec:appl}

In this section, we calculate the values of our invariants introduced in Section \ref{sec:distinguish} for concrete pyramids.
As an application, we know that $\Gamma^\infty$ and $I^\infty$ are distinct, 
even if these are scaled (Theorem \ref{thm:main}).
For other applications, see Section \ref{subsec:regular_space_is_not_disjoint}.

\subsection{Idempotent pyramids and Gaussian}
The infinite product of spaces is intuitively related to Gaussian measures due to the central limit theorem.
We discuss this rigorously in our framework.

\begin{proposition}[Central limit theorem]
\label{prop:CLT}
Let $P$ be an $\ell_p$-idempotent pyramid for $p \in [1,2]$.
If $p=2$, then $P$ contains $\Gamma_{V(P)}^\infty$.
If $1 \le p<2$, then $P$ is $\cal X$ unless $P = \{\ast\}$.
\end{proposition}
\begin{proof}
Let us assume $p = 2$ and $\sigma^2 = V(P)$.
If $\sigma=0$, then $P= \{\ast \}$.
So, there is nothing to prove.
We assume that $\sigma > 0$.
We first consider the case of $\sigma<\infty$.
Let us take arbitrary $\sigma' < \sigma$.
Then, there exist
a pm-space $X \in P$ with $V(X) > \sigma'$ and a function $f \in \Lip_1(X)$ with $V_X(f) > \sigma'$ and $E_X(f) = 0$.
Then, the function $f_n : X_2^n \to \mathbb R$ defined by
\[
f_n(x) = \frac{1}{\sqrt{n}} \sum_{i=1}^n f(x_i)
\]
for $x=(x_i)_{i=1}^n \in X_2^n$, is $1$-Lipschitz.
By the central limit theorem, $(f_n)_\# m_X^{\otimes n}$ converges to the Gaussian measure $\gamma_{V_X(f)}^1$ weakly, as $n \to \infty$.
Therefore,
\[
\Gamma_{(\sigma')^2}^1 \prec \Gamma_{V_X(f)}^1 \in X_2^\infty \subset P.
\]
Letting $\sigma' \to \sigma-0$, we obtain $\Gamma_{\sigma^2}^1 \in P$.
Since $P$ is $\ell_2$-idempotent, we have
$P \supset (\Gamma_{\sigma^2}^1)_{{2}}^{\infty} = \Gamma_{V(P)}^\infty$.
This is the desired conclusion.

We assume $\sigma = \infty$ and $p=2$.
By the above argument, $P$ contains $\Gamma_{\lambda^2}^1$ for all $\lambda > 0$.
Since $\lim_{\lambda \to \infty} \cal P \Gamma_{\lambda^2}^1 = \cal X$, we have $P=\cal X$.

We consider the case of $1 \le p<2$.
We may assume that $P \ne \{\ast \}$.
Let $\lambda > 0$ be an arbitrary number.
Since $P \ne \{\ast\}$, there exists a pm-space $X \in P$ with $0< V(X)< \infty$.
Let us take $f \in \Lip_1(X)$ such that $V_X(f)$ is close to $V(X)$.
Then, $f_n^\lambda : X_p^n \to \mathbb R$ defined by
\[
f_n^\lambda(x) = \frac{\lambda}{\sqrt{n}} \sum_{i=1}^n f(x_i)
\]
is $1$-Lipschitz, if $n^{\frac{2-p}{2p}} > \lambda$.
By the central limit theorem and an argument as in the first paragraph, we have $
P \ni \Gamma_{\lambda^2 V(X)}^1$.
Since $\lambda$ is arbitrary, $P=\cal X$.
\end{proof}
Due to this proposition, we know that the infinite $\ell_p$-product of pm-spaces and pyramids is not continuous with respect to $p$.

\begin{corollary} \label{cor:ell_2-product_of_discrete_space}
Let $P$ be an $\ell_p$-idempotent pyramid with $2 \le p < \infty$.
Suppose that $P \ne \cal X$.
Then, $P$ never contain {mm-disconnected} pm-spaces.
\end{corollary}

From the definition of $F_{\Gamma^\infty}$ obtained in Theorem \ref{thm:invariant}, we have:
\begin{corollary} \label{cor:Gauss}
Let $P$ be an $\ell_2$-idempotent pyramid.
Then, $\sqrt{V(P)} = F_{\Gamma^\infty}(P)$ if and only if $P$ coincides with $\Gamma_{V(P)}^\infty$.
\end{corollary}

\subsection{Calculating invariants of $I^{\infty}$ and $\Gamma^{\infty}$}

\begin{proposition} \label{prop:V(Gamma)=1}
	$V(\Gamma^\infty)=1$.
\end{proposition}
\begin{proof}
	Due to Gaussian Poincar\'e inequality (see e.g. \cite[Proposition 4.1.1]{BGL}) and \eqref{eq:V_vs_PI}, we have
	\[
	\sqrt{V(\Gamma^n)} \le C_{2,2}(\Gamma^n) = 1.
	\]
	Since $V$ is monotone, we have $V(\Gamma^1) \le V(\Gamma^n)$.
	By considering a concrete function $f(t) = t$ on $\Gamma^1$, we have $V(\Gamma^1) = 1$.
	Therefore, by Corollary \ref{cor:limit}, we obtain the conclusion.
\end{proof}

\begin{remark}
	Due to Proposition \ref{prop:V(Gamma)=1}, we have $V(\Gamma_{\sigma^2}^\infty)=\sigma^2$.
	Recall that the name of the pyramid $\Gamma_{\sigma^2}^\infty$ is the virtual infinite dimension Gaussian space of variance $\sigma^2$,
	because it is the limit of finite dimensional Gaussian spaces of variance $\sigma^2$.
	Our result says that the variance of the virtual infinite dimensional Gaussian space of variance $\sigma^2$ is actually $\sigma^2$.
\end{remark}

\begin{corollary}\label{cor:Poincare_constant}
$C_{2,2}(\Gamma^\infty) = 1$ and $C_{2,2}(I^\infty) = 1/\pi$.
\end{corollary}
\begin{proof}
We can easily check $C_{2,2}(\Gamma^1) = 1$ and $C_{2,2}(I)=1/\pi$ (see for instance \cite{BGL}).
Therefore, the conclusion follows from Proposition \ref{prop:smooth_PI}.
\end{proof}

For $I^\infty$, we only have an estimate of its variance: 
\begin{corollary}
	\label{cor:V(I^infty)_estimate}
	For $r > 0$, we have
	\[
	\frac{r}{2 \sqrt{3}} \le \sqrt {V(rI^\infty)} \le \frac{r}{\pi}.
	\]
\end{corollary}
\begin{proof}
	By Corollary \ref{cor:Poincare_constant}, 
	we have
	\[
	\frac{1}{12} = V(I) \le V(I^n) \le C_{2,2}(I^n)^2 \le \frac{1}{\pi^2}.
	\]
	Letting $n \to \infty$, we obtain the conclusion.
\end{proof}

\begin{corollary} \label{cor:LS_constant}
$C_\mathrm{LS}(\Gamma^\infty)=1$ and $C_\mathrm{LS}(I^\infty) = 1/\pi$.
\end{corollary}
\begin{proof}
Due to the Gaussian logarithmic Sobolev inequality (see for instance \cite[Proposition 5.5.1]{BGL}), we have $C_\mathrm{LS}(\Gamma^n) \le 1$ for all $n \in \mathbb N$.
By \eqref{eq:Poincare_vs_LS} and Corollary \ref{cor:Poincare_constant}, we have $C_\mathrm{LS}(\Gamma^n)=1$.
Therefore, by Proposition \ref{prop:LS_tensor},
we obtain $C_\mathrm{LS}(\Gamma^\infty)=1$.

By \cite[Proposition 5.7.5]{BGL}, \eqref{eq:Poincare_vs_LS} and Corollary \ref{cor:Poincare_constant}, we know $C_\mathrm{LS}(I)=1/\pi$.
Therefore, by Proposition \ref{prop:LS_tensor}, we obtain $C_\mathrm{LS}(I^\infty)=1/\pi$.
\end{proof}

Using the above quantities we have the following proposition. 

\begin{proposition} \label{prop:cube_vs_Gaussian}
We have
\[
\Gamma_{\frac{r^2}{12}}^\infty \subset rI^\infty \subset \Gamma_{\frac{r^2}{2\pi}}^\infty.
\]
\end{proposition}

\begin{proof}
	Due to a simple calculation, we have $V(I)=1/12$.
	Hence, $V(I^\infty) \ge 1/12$.
	By Proposition \ref{prop:CLT}, we have
	\[
	I^\infty \supset \Gamma^\infty_{1/12}.
	\]
	Let $\sigma > 0$ be an arbitrary number.
	We consider the function defined by
	\[
	f_\sigma(t) := \gamma_{\sigma^2}^1((-\infty ,t]) = \frac{1}{\sqrt{2\pi \sigma^2}} \int_{-\infty}^t \exp\left( {\frac{-u^2}{2\sigma^2}} \right)\, du.
	\]
	Then, we have $(f_{\sigma})_\# \gamma_{\sigma^2}^1 = \cal L^1 \lfloor_{[0,1]}$.
	Clearly, $f_\sigma$ is $1$-Lipschitz if and only if $\sigma \ge \frac{1}{\sqrt{2 \pi}}$.
	Therefore, we have
	\[
	I \prec \Gamma_{1/{2\pi}}^1.
	\]
	By taking the infinite $\ell_2$-product, we obtain
	\[
	I^\infty \subset \Gamma_{1/{2\pi}}^\infty.
	\]
	This completes the proof.
\end{proof}

\subsection{Distinguishing $I^\infty$ with $\Gamma^\infty$}

To prove Theorem \ref{thm:main}, we show 
the following lemma. 

\begin{lemma} \label{lem:obsdiam_comparison}
	For every $\sigma<1/\sqrt{2 \pi}$, there exists $\kappa \in (0,1)$ such that
	\begin{equation*}
		\ObsDiam(\Gamma_{\sigma^2}^\infty;-\kappa) < \ObsDiam(I^\infty;-\kappa).
	\end{equation*}
	In particular, we have
	\[
	I^\infty \not\subset \Gamma_{\sigma^2}^\infty
	\]
	for every $\sigma < 1/\sqrt{2\pi}$.
\end{lemma}
\begin{proof}
	Let $\Psi(s) := \gamma^1([0,s]) = \int_0^s \frac{1}{\sqrt{2\pi}} e^{-t^2/2}\,dt$.
	As in Remark \ref{rem:limit_formula},
	we have
	\[
	\ObsDiam(\Gamma^{\infty}_{\sigma^2}; - \kappa) = 2\sigma \Psi^{-1} \left( \frac{1-\kappa}{2} \right).
		\]
	Since the observable diameter is monotone, we obtain
	\[
	\ObsDiam(I^\infty;-\kappa) \ge \ObsDiam(I;-\kappa) = 1 - \kappa.
	\]

	Let us define a smooth function $f : [0,1] \to \mathbb R$ by
	\[
	f(s) = \Psi \left(
	\frac{1-s}{2\sigma} \right)- \frac{1-s}{2}.
	\]
	Since
	\begin{align*}
		f'(1)
		&= \frac{1}{2\sigma} \left(
		\sigma- \frac{1}{\sqrt{2 \pi}}
		\right),
	\end{align*}
	we have $f'(1) < 0$ if $\sigma < 1/\sqrt{2\pi}$.
	Hence, if $\sigma < 1/\sqrt{2 \pi}$, then
	there exists $\kappa \in (0,1)$ close to $1$ such that $f(\kappa) > 0$.
	Therefore, we obtain the conclusion.
\end{proof}

We say that two pyramids $P$ and $Q$ are {\it similar} if there exists $t > 0$ such that $t P=Q$.
In this terminology, Theorem \ref{thm:main} is shortly stated that $\Gamma^\infty$ and $I^\infty$ are not similar.

\begin{proof}[Proof of Theorem \ref{thm:main}]
	By Proposition \ref{prop:cube_vs_Gaussian} and Lemma \ref{lem:obsdiam_comparison}, we have
	\[
	F_{\Gamma^\infty}(I^\infty) = \frac{1}{\sqrt{2\pi}}.
	\]
	On the other hand, by \eqref{eq:V_vs_PI}, and by Proposition \ref{prop:variance}, we have
	\[
	\sqrt{V(I^\infty)} \le C_{2,2}(I) = \frac{1}{\pi}.
	\]
	Therefore, $I^\infty$ is not similar to $\Gamma^\infty$, because
	$\sqrt{V(\Gamma^\infty)} ={1}= F_{\Gamma^\infty}(\Gamma^\infty)$.
\end{proof}

\begin{proof}[Proof of Theorem \ref{thm:F_Gamma}]
	This follows from Corollary \ref{cor:Gauss} and Lemma \ref{lem:obsdiam_comparison}.
\end{proof}

\subsection{Other applications} \label{subsec:regular_space_is_not_disjoint}

Due to Proposition \ref{prop:cube_vs_Gaussian}, we know that, for any $N \ge 1$,
there exists $\{X_n\}_n$ box-convergent to $I^N$ with $X_n \prec \Gamma_{1/2\pi}^n$.
However, using Lemma \ref{lem:obsdiam_comparison}, we immediately obtain:
\begin{proposition}
	For every $\epsilon > 0$, there exists $N \ge 1$ such that, for every $N' \ge N$, if a sequence $\{X_n\}_n$ of pm-spaces satisfying one of $X_n \prec \Gamma_{(1-\epsilon)/2\pi}^n$, $X_n \prec \mathbb S^n(\sqrt{(1-\epsilon)n/{2\pi}})$ and $X_n \prec \mathbb B^{n+1}(\sqrt{(1-\epsilon)n/{2\pi}})$ for each $n$, then $\{X_n\}_n$ 
does not concentrate to $I^{N'}$ as $n \to \infty$.
\end{proposition}

For each $t\in [0,\infty)$, let us consider a subset of $\cal X$ defined by
\[
A(t) = \left\{ X \in \cal X \,\middle|\, C_{2,2}(X) > t \right\}.
\]
Furthermore, we set
\[
A(\infty) = \bigcap_{t >0}A(t) = \left\{X \in \cal X \,\middle|\, C_{2,2}(X) = \infty \right\}.
\]
Clearly, if a pyramid $P$ satisfies $C_{2,2}(P) < \infty$, then $P \cap A(\infty) = \emptyset$.
For more precisely, we have
\begin{proposition} \label{prop:Poincare_vs_discrete}
If a pyramid $P$ is the weak limit of a sequence $\{X_n\}_n$ of pm-spaces satisfying $t:=\sup_{n} C_{2,2}(X_n) < \infty$, then we have $P \cap A(t) = \emptyset$.
In particular, such a $P$ never contain {mm-disconnected} pm-spaces.

If a pyramid $Q$ is the weak limit of pm-spaces $\{Y_n\}$ so that $Y_n$ is dominated by a CD($K,\infty$)-space for each $n$, where $K \in (0,\infty)$ is a constant, then we have $Q \cap A(1/K) = {\emptyset}$.
\end{proposition}
\begin{proof}
The first statement clearly holds from the definition.
The second statement follows from Proposition \ref{prop:Poincare=infty}.
The last statement follows from Theorem \ref{thm:CD_implies_FI}.
\end{proof}

\subsubsection{Concentrated pyramids} \label{subsec:conc_pyramid}
Since the embedding $\iota : \cal X \to \Pi$ is $1$-Lipschitz in $d_\conc$,
this map is canonically extended to a $1$-Lipschitz map $\overline{\cal X} \to \Pi$, where $\overline{\cal X}$ is the completion of $(\cal X, d_\conc)$.
It is known that the map $\overline{\cal X} \to \Pi$ is also topological embedding (\cite[Theorem 7.27]{S}).

We say that a pyramid $P$ is {\it concentrated} if $\{\cal L_1(X)\}_{X \in P}$ is precompact in the Gromov-Hausdorff topology. 
Here, $\cal L_1(X)$ was defined in Subsection \ref{subsec:basic}.
From the definition, for two pyramids $P$ and $Q$, if $P$ is concentrated and $P \supset Q$, then $Q$ is concentrated.
Due to \cite[Theorem 7.25]{S}, $P$ is concentrated if and only if $P$ is in the image of $\overline{\cal X} \to \Pi$.

For example, due to \cite[Corollary 7.35]{S}, if a sequence $\{X_n\}$ consisting of smooth compact Riemannian manifold with usual distance function and with normalized volume measure {satisfies} $\lim_{n \to \infty} C_{2,2}(X_{n}) = 0$, then $\{Y_n\}$ is $d_\conc$-Cauchy, where
\[
Y_n := X_1 \times_2 X_2 \times_2 \cdots \times_2 X_n.
\]
Hence, the weak limit
\[
\bigotimes_{n=1}^\infty (\{X_n\}; \ell_2) = \lim_{n \to \infty} \cal PY_n
\]
is a concentrated pyramid.

From now on, we use the following symbol
\[
\bigotimes_{n=1}^\infty Z_n := \bigotimes_{n=1}^\infty (\{Z_n\}; \ell_2)
\]
for pm-spaces $Z_n$.
We focus on the following particular concentrated pyramid
\[
\bigotimes_{n=1}^\infty \mathbb S^n
= \lim_{n \to \infty} \cal P(\mathbb S^1 \times_2 \mathbb S^2 \times_2 \mathbb S^3 \times_2 \cdots \times_2 \mathbb S^n).
\]
Here, $\mathbb S^n$ denotes the standard unit $n$-sphere with intrinsic metric and uniform measure.
Due to \cite{F}, we know that the family $\{\mathbb S^n\}_{n=1}^\infty$ of unit spheres has no convergent subsequence in the box topology.
Since $\bigotimes_{n = 1}^\infty \mathbb S^n$ contains $\{\mathbb S^n\}_n$, it can not represented by pm-spaces {(see \cite{KY})}.

\begin{theorem} \label{thm:non-idempotent}
For each distinct $k, \ell \in \mathbb N$, the pyramids 
$\left( \bigotimes_{n=1}^\infty \mathbb S^n \right)^{\otimes_2 k}$ and $\left( \bigotimes_{n=1}^\infty \mathbb S^n \right)^{\otimes_2 \ell}$
are different.
Furthermore, 
$\left( \bigotimes_{n=1}^\infty \mathbb S^n \right)^{\otimes_2 k}$ is concentrated and {non-trivial.}
\end{theorem}

Note that due to Lemma \ref{lem:approx_product}, we have
\begin{equation} \label{eq:product_of_spheres}
\left( \bigotimes_{n=1}^\infty \mathbb S^n \right)^{\otimes_2 k}
= \bigotimes_{n=1}^\infty (\mathbb S^n)^{\otimes_2 k}.
\end{equation}

To prove Theorem \ref{thm:non-idempotent}, we prepare the following lemma.
\begin{lemma} \label{lem:idempotent}
Let $P$ be a pyramid. Suppose that
\[
P^{\otimes_p k} = P^{\otimes_p \ell}
\]
holds for some $k, \ell \in \mathbb N$ with $k < \ell$.
Then, $P^{\otimes_p k}$ is $\ell_p$-idempotent.
\end{lemma}
\begin{proof}
From the assumption, we have
\[
P^{\otimes_p k} \subset P^{\otimes_p (k+1)} \subset P^{\otimes_p \ell} = P^{\otimes_p k}.
\]
Hence, $P^{\otimes_p k} = P^{\otimes_p (k+1)}$.
Therefore,
\[
P^{\otimes_p k} = P^{\otimes_p (k+1)} = P^{\otimes_p k} \otimes_p P = P^{\otimes_p (k+1)} \otimes_p P = P^{\otimes_p (k+2)} = \cdots = P^{\otimes_p 2k}.
\]
This completes the proof.
\end{proof}

\begin{proof}[Proof of Theorem \ref{thm:non-idempotent}]
Let $P := \bigotimes_{n=1}^{\infty} \mathbb S^n$.
The equation \eqref{eq:product_of_spheres} and 
\begin{equation*} 
	C_{2,2}((\mathbb S^n)^{\otimes_2 k}) = C_{2,2}(\mathbb S^n) \to 0 \text{ as } n \to \infty
\end{equation*}
imply that $P^{\otimes_2 k}$ is concentrated. 
Moreover, since $P^{\otimes_2 k}$ contains $\{ \mathbb S^n\}_n$, $P^{\otimes_2 k}$ is not represented by pm-spaces.
These prove the second statement. 

Let us prove the first statement. 
Suppose that $P^{\otimes_2 k} = P^{\otimes_2 \ell}$ for some $k<\ell$.  
Then $P^{\otimes_2 k}$ is $\ell_2$-idempotent by Lemma \ref{lem:idempotent}, and hence, $P^{\otimes_2 k}$ contains $\Gamma_{V(\mathbb S^1)}^\infty$ by Proposition \ref{prop:CLT}.
This implies that $\Gamma^\infty_{V(\mathbb S^1)}$ is concentrated. 
However, this contradicts to that $\Gamma_{V(\mathbb S^1)}^\infty$ is not concentrated (\cite[Proposition 7.37]{S}).
Therefore, we obtain the first statement. 
\end{proof}

The following is obtained immediately.
\begin{proposition}
	\label{lem:concentrated}
	For a concentrated pyramid $Q$ with $V(Q) > 0$, we have
	\[
	Q \ne 
	Q^{\otimes_2 2}.
	\]
\end{proposition}

\section{The space of pyramids {generated by} atoms and dissipation property} \label{sec:atom}

In this section, we introduce generalized notion of dissipation and separation distance. 
After that, we prove Theorems \ref{thm:mugen} and \ref{thm:algebra}.

\subsection{Moduli space of atoms}

Let us define
\[
\widetilde V := \left\{
\alpha = (\alpha_i)_{i=1}^\infty \in [0, \infty)^{\mathbb N}\,\middle|\, \|\alpha\|_1 \le 1
\right\}.
\]
Here, $\|x\|_1 := \sum_{i=1}^\infty |x_i|$ for arbitrary $x =(x_i) \in \mathbb R^{\mathbb N}$.
Let us consider an equivalent relation on $\widetilde V$ generated by
\[
\alpha \sim \beta \iff
\left\{
\begin{aligned}
& \text{there exists an injective map $\sigma : \mathbb N \to \mathbb N$} \\
& \text{such that $\alpha_i = \beta_{\sigma(i)}$ for all $i \ge 1$}.
\end{aligned}
\right.
\]
Then, the quotient set $\widetilde V/\!\!\sim$ is identified with
\[
V := \left\{ \alpha \in \widetilde V \,\middle|\, \alpha_i \ge \alpha_{i+1} \text{ for all } i \ge 1
\right\}.
\]
Let us give a topology $V$ induced by the $\ell_\infty$-norm.
Note that the $\ell_\infty$-topology is different from the $\ell_1$-topology.

\begin{lemma} \label{lem:uniform_vs_pointwise}
For a sequence $\alpha_n = (\alpha_{ni})_{i=1}^\infty \in V$ and $\alpha = (\alpha_i)_{i=1}^\infty \in V$, the following statements are equivalent:
\begin{enumerate}
\item $\alpha_n$ converges to $\alpha$ in the $\ell_\infty$-norm as $n \to \infty$.
\item $\alpha_n$ converges to $\alpha$ pointwisely, that is, $\lim_{n \to \infty} \alpha_{ni} = \alpha_i$ holds for each $i \ge 1$.
\end{enumerate}
In particular, $(V, \|\cdot\|_\infty)$ is compact.
\end{lemma}
\begin{proof}
(1) $\implies$ (2) is trivial.
So, we only prove the converse.
Let us assume (2).
First, we are going to show:
\begin{equation} \label{eq:unform_convergence001}
\lim_{i \to \infty} \sup_{n} \alpha_{ni}  = 0.
\end{equation}
Note that \eqref{eq:unform_convergence001} is equivalent to the statement ``for any $\epsilon > 0$, there exists $N \in \mathbb N$ such that if $n,i \ge N$ then $\alpha_{ni} \le \epsilon$''.
So, if \eqref{eq:unform_convergence001} does not hold, then there exist $n_j \to \infty$ and $i_j \to \infty$ such that
\[
c:= \inf_{j \ge 1} \alpha_{n_j i_j} > 0.
\]
Hence, we have
\[
\sum_{i=1}^\infty \alpha_{n_j\,i} \ge \sum_{i=1}^{i_j} \alpha_{n_j\,i} \ge i_j c
\]
since $\alpha_n$ is monotone nonincreasing.
This implies
\[
1 \ge \sup_n \|\alpha_n \|_1 \ge i_j c.
\]
This is a contradiction.
So, \eqref{eq:unform_convergence001} holds.

We next show that $(1)$ holds.
If (1) does not hold, then there 
exist
$\epsilon > 0$, $n_j \to \infty$ and $i_j \to \infty$ such that
\begin{equation} \label{eq:uniform_convergence001-1}
|\alpha_{n_j i_j} - \alpha_{i_j} | \ge \epsilon
\end{equation}
for every $j \ge 1$.
By $\alpha_{n_j i_j} \le \sup_n \alpha_{ni_j}$ and \eqref{eq:unform_convergence001}, we have $\lim_{j \to \infty} \alpha_{n_j i_j} = 0$.
Then we obtain $\lim_{j \to \infty} \alpha_{i_j} \ge \epsilon$ by \eqref{eq:uniform_convergence001-1}.
On the other hand, by $\alpha \in V$, we have $\lim_{j \to \infty} \alpha_{i_j} = 0$.
This is a contradiction.
So, (1) holds.
\end{proof}

For each $\alpha \in {\mathbb R^{\mathbb N}}$, let us set
\[
N(\alpha) = \# \{ i \ge 1 \mid \alpha_i > 0\} \in \mathbb Z_{\ge 0} \cup \{\infty\}.
\]

\subsection{Pyramids {generated by} atoms}

For each $\alpha \in \widetilde V$,
we consider $\cal X_\alpha \subset \cal X$ defined by
\[
\cal X_\alpha = \left\{ X \in \cal X \,\middle|\,
\begin{aligned}
& \text{there exists a map $f : \mathbb N \to X$ such that} \\
& \text{$m_X - \sum_{i=1}^\infty \alpha_i \delta_{f(i)}$ is a nonnegative measure}
\end{aligned}
\right\}.
\]
We will often write $f_\# \alpha = \sum_{i=1}^\infty \alpha_i \delta_{f(i)}$.
Note that
\[
\cal X_0 = \cal X \text{ and } \cal X_1 = \{\ast\}.
\]
Here, $0=(0,0,0,\dots)$ and $1=(1,0,0,\dots)$ as elements of $\widetilde V$.
Furthermore, if $\alpha \sim \beta$, then $\cal X_\alpha = \cal X_\beta$.

\begin{theorem} \label{thm:X_alpha_is_pyramid}
For each $\alpha \in \widetilde V$, $\cal X_\alpha$ is a pyramid.
\end{theorem}

After proving this, we call $\cal X_\alpha$ the {\it pyramid {generated by} atoms} $\alpha$.
To prove Theorem \ref{thm:X_alpha_is_pyramid}, we show the following statements.

\begin{lemma} \label{lem:X_alpha_is_directed}
$\cal X_\alpha$ is directed and downward-closed.
\end{lemma}
\begin{proof}
Since it is clear that $\cal X_\alpha$ is downward-closed, we only show that $\cal X_\alpha$ is directed.
Let us take $X, Y \in \cal X_\alpha$.
Then, there exist maps $f : \mathbb N \to X$ and $g : \mathbb N \to Y$ such that
\[
\nu_X :=m_X- \sum_{i=1}^\infty \alpha_i \delta_{f(i)} \text{ and } \nu_Y:= m_Y - \sum_{i=1}^\infty \alpha_i \delta_{g(i)}
\]
are nonnegative Borel measures on $X$ and $Y$, respectively.
Let us consider the measure $m$ on $X \times Y$ defined by
\[
m =
\left\{
\begin{aligned}
& \sum_{i=1}^\infty \alpha_i \delta_{(f(i), g(i))} + (1 - \|\alpha \|_1) \frac{\nu_X}{1-\| \alpha \|_1} \otimes \frac{\nu_Y}{1-\|\alpha\|_1} &&\text{ if } \|\alpha \|_1 < 1, \\
& \sum_{i=1}^\infty \alpha_i \delta_{(f(i), g(i))} &&\text{ if } \| \alpha \|_1 = 1.
\end{aligned}
\right.
\]
Then, $Z:=(X \times Y, d_X \times_\infty d_Y, m) \in \cal X_\alpha$.
Clearly, $Z$ dominates $X$ and $Y$.
Therefore, $\cal X_\alpha$ is directed.
\end{proof}

\begin{proposition} \label{prop:N<infty_then_X_alpha_is_pyramid}
If $N(\alpha) < \infty$, $\cal X_\alpha$ is a pyramid.
\end{proposition}
\begin{proof}
Let us set $N = N(\alpha) < \infty$ and assume $N > 0$.
By Lemma \ref{lem:X_alpha_is_directed}, it is sufficient to show that $\cal X_\alpha$ is box-closed.
Let us take a sequence $\{X_n\}_n \subset \cal X_\alpha$ box-convergent to $X \in \cal X$, as $n \to \infty$.
From the definition, there exists a map $f_n : \{1, \dots, N \} \to X_n$ such that
\[
\nu_n := m_{X_n} - (f_n)_\# \alpha
\]
is a nonnegative measure on $X_n$.
By Lemma \ref{lem:mG_convergece}, we may assume that $X_n$ and $X$ are isometrically contained in a common complete separable metric space $Y$ such that $m_{X_n}$ weakly converges to $m_X$.
Let us take $\epsilon \in (0, \min_i \alpha_i)$.
Due to Prokhorov's theorem, there exists a compact set $K \subset Y$ such that
$m_{X_n}(K) > 1 - \epsilon$.
Since $\epsilon < \min_i \alpha_i$, we have
\[
\supp\, (f_n)_\# \alpha \subset K
\]
for every $n$.
Therefore, there is a nonnegative measure $\xi$ on $K$ such that a subsequence of $\{ (f_n)_\# \alpha \}_n$ converges to $\xi$.
Then, $\xi(K) = \|\alpha \|_1$ and $\# \supp\, \xi \le N$.
Moreover, $\xi$ has the form of
\[
\xi = \sum_{i=1}^N \alpha_i \delta_{x_i}
\]
for some $x_i \in K$ ($1 \le i \le N$).
Therefore, $X \in \cal X_\alpha$.
This completes the proof.
\end{proof}

For $\alpha \in \widetilde V$ and $N \ge 1$, we set
\[
\alpha[N]=(\alpha_1, \alpha_2, \dots, \alpha_N, 0,0,\dots).
\]
Remark that $\cal X_{\alpha[N]} \supset \cal X_{\alpha[N+1]} \supset \cal X_\alpha$ holds.

\begin{proposition} \label{prop:intersection}
$\cal X_\alpha = \bigcap_{N =1}^\infty \cal X_{\alpha[N]}$ holds.
\end{proposition}
\begin{proof}
$\cal X_\alpha \subset \bigcap_{N=1}^\infty \cal X_{\alpha[N]}$ is obvious.
We check the converse inclusion. Take any $X \in \bigcap_{N=1}^\infty \cal X_{\alpha[N]}$.
For each $N \in \mathbb N$, there exists map $f_N : \{1,\ldots,N\} \to X$ such that
\[
\nu_N := m_X - \sum_{i=1}^N \alpha_i\delta_{f_N(i)}
\]
is a nonnegative measure.
Then, for any $i=1,2,\ldots$, the sequence $\{f_N(i)\}_{N\ge i}$ has a convergent subsequence. Indeed, if not, then there exists a subsequence $\{N_k\}_{k}$ such that $\{f_{N_k}(i)\}_{k}$ is uniformly discrete, {that is, $\inf_{k \ne \ell} d_X(f_{N_k}(i), f_{N_\ell}(i)) > 0$.}
Thus
\[
1\ge m_X(\{f_{N_k}(i) \, | \, k=1,2,\ldots\}) = \sum_{k=1}^\infty m_X(\{f_{N_k}(i)\}) \ge \sum_{k=1}^\infty \alpha_i =\infty,
\]
which is a contradiction.
Let $x_i$ be the limit of a convergent subsequence of $\{f_N(i)\}_{N\ge i}$ and define a map $f : \mathbb N \to X$ by $f(i) := x_i$ for each $i$.
Then the measure
\[
m_X - \sum_{i=1}^\infty \alpha_i\delta_{f(i)}
\]
is the weak limit of $\{\nu_N\}_{N}$ and then it is nonnegative. Therefore we obtain $X \in \cal X_\alpha$.
\end{proof}

\begin{proof}[Proof of Theorem \ref{thm:X_alpha_is_pyramid}]
Let us take $\alpha \in \widetilde V$.
By Propositions \ref{prop:N<infty_then_X_alpha_is_pyramid} and \ref{prop:intersection}, $\cal X_\alpha$ is closed.
Combining this and Lemma \ref{lem:X_alpha_is_directed} implies that $\cal X_\alpha$ is a pyramid.
\end{proof}

\begin{remark} \label{rem:X_alpha_is_not_I_and_Gamma}
Clearly, $\cal X_\alpha$ is scale-invariant.
Furthermore, $F(\cal X_\alpha) = \infty$ if $\alpha \ne 1$,
for $F=V$, $C_{p,q}$ and $C_\mathrm{LS}$.
Thus, both $\Gamma^\infty$ and $I^\infty$ does not contain $\cal X_\alpha$ with $\alpha \ne 1$ (compare Proposition \ref{prop:Poincare_vs_atoms}).
In particular, $\cal X_\alpha$ coincides with neither $\Gamma_{\sigma^2}^\infty$ nor $rI^\infty$ for $\sigma, r > 0$ and $\alpha \in V$ with $\alpha \ne 1$.

A recent work by Kazukawa-Nakajima-Shioya (\cite{KNS}) shows that every scale-invariant pyramid is $\cal X_\alpha$ for some $\alpha$.
\end{remark}

\begin{remark} \label{rem:X^delta_alpha}
For $\delta \in [0, \infty]$, recall that
\[
\cal X^\delta = \left\{ X \in \cal X \mid \diam \, X \le \delta \right\}
\]
is a pyramid.
Furthermore, we set
\[
\cal X_\alpha^\delta := \cal X_\alpha \cap \cal X^\delta,
\]
which is clearly box-closed.
By the proof of Lemma \ref{lem:X_alpha_is_directed}, $\cal X_\alpha^\delta$ is known to be directed and downward-closed.
Thus $\cal X_\alpha^\delta$ is a pyramid.
By the definition, $\cal X_\alpha^\infty = \cal X_\alpha$ and $\cal X_\alpha^0 = \{\ast\}$.
\end{remark}

\subsection{Injectivity}

Let $\alpha, \beta \in \widetilde V$.
We say that $\beta$ is a {\it contraction} of $\alpha$ if there exists a map $\pi : \mathbb N \to \mathbb N$ such that
\[
\beta_j = \sum_{i \in \pi^{-1}(j)} \alpha_i
\]
holds for every $j \ge 1$.

Note that if $\beta$ is a contraction of $\alpha$, then
\[
\cal X_\beta \subset \cal X_\alpha.
\]

\begin{proposition} \label{prop:bi-contraction}
For $\alpha, \beta \in V$, if $\alpha$ is a contraction of $\beta$ and if $\beta$ is also a contraction of $\alpha$, then $\alpha=\beta$.
\end{proposition}

\begin{proof}
Assume that $\alpha$ is a contraction of $\beta$ and $\beta$ is also a contraction of $\alpha$. There exist maps $\pi_1, \pi_2 : \mathbb N \to \mathbb N$ such that for any $j \in \mathbb N$,
\[
\beta_j = \sum_{i \in \pi_1^{-1}(j)}\alpha_i \quad\text{ and }\quad \alpha_j = \sum_{i \in \pi_2^{-1}(j)} \beta_i.
\]
Letting $\pi := \pi_2 \circ \pi_1$, for any $j \in \mathbb N$,
\[
\alpha_j = \sum_{i \in \pi_2^{-1}(j)} \beta_i = \sum_{i \in \pi_2^{-1}(j)} \sum_{k \in \pi_1^{-1}(i)}\alpha_k = \sum_{k \in \pi^{-1}(j)} \alpha_k
\]
holds. Since
\[
\alpha_{\pi(1)} = \sum_{i \in \pi^{-1}(\pi(1))}\alpha_i \geq \alpha_1 \geq \alpha_{\pi(1)},
\]
we have $\pi^{-1}(\pi(1)) = \{1\}$.
Then $\pi(1) = 1$ can be assumed by relabeling.
By iteration, we obtain $\pi = \pi_2\circ\pi_1 = \mathrm{id}_\mathbb N$.
Similarly, $\pi_1 \circ \pi_2 = \mathrm{id}_\mathbb N$ is also obtained, so that $\pi_1, \pi_2$ are bijections.
Moreover, since
\[
\alpha_1 = \beta_{\pi_1(1)} \leq \beta_1 = \alpha_{\pi_2(1)} \leq \alpha_1,
\]
we can assume that $\pi_1(1) = \pi_2(1) = 1$ by relabeling.
By iteration, we obtain $\pi_1 = \pi_2 = \mathrm{id}_\mathbb N$, so that $\alpha=\beta$.
\end{proof}

\begin{proposition} \label{prop:injective}
The correspondence
$V \ni \alpha \mapsto \cal X_\alpha \in \Pi$ is injective.
\end{proposition}
\begin{proof}
For each $\alpha \in V$, we define $X_\alpha \in \cal X$ by
\begin{equation} \label{eq:X_alpha}
X_\alpha := ([-1, 0]\cup \mathbb N, |\cdot|, m_{X_\alpha}),
\end{equation}
where
\[
m_{X_\alpha} := (1-\|\alpha\|_1)\cal L \lfloor_{[-1,0]} + \sum_{i=1}^\infty \alpha_i \delta_i.
\]
Let $\alpha, \beta \in V$ with $\cal X_\alpha = \cal X_\beta$.
Since $X_\alpha \in \cal X_\alpha = \cal X_\beta$, there exists a map $f : \mathbb N\to X_\alpha$ such that $m_{X_\alpha}-f_\# \beta$ is a nonnegative measure.
Then it must hold that
\[
\sum_{i=1}^\infty\beta_i\delta_{f(i)} \le \sum_{i=1}^\infty\alpha_i\delta_{i},
\]
which implies $\|\alpha\|_1 \ge \|\beta\|_1$ particularly.
By symmetry, since $X_\beta \in \cal X_\alpha$, we obtain $\|\alpha\|_1 = \|\beta\|_1$.
This implies
\[
\sum_{i=1}^\infty\beta_i\delta_{f(i)} = \sum_{i=1}^\infty\alpha_i\delta_{i}
\]
and then $\alpha$ is a contraction of $\beta$.
By a parallel discussion, $\beta$ is a contraction of $\alpha$.
Therefore Proposition \ref{prop:bi-contraction} implies $\alpha=\beta$.
\end{proof}

For $\alpha \in \widetilde V$ and $N \ge 1$, we set
\[
\alpha_c[N] = \left( \alpha_1, \alpha_2, \dots, \alpha_{N-1}, \sum_{{i=N}}^\infty \alpha_i, 0,0,\dots \right)
\]
which is a contraction of $\alpha$.
Remark that $\cal X_{\alpha_c[N]} \subset \cal X_{\alpha_c[N+1]} \subset \cal X_\alpha$ holds.
\begin{proposition} \label{prop:lower_approximation_of_pyramid_with_atoms}
$\cal X_\alpha = \overline{\bigcup_{{N=1}}^\infty \cal X_{\alpha_c[N]}}$ holds.
\end{proposition}
\begin{proof}
$\cal X_\alpha \supset \overline{\bigcup_{{N=1}}^\infty \cal X_{\alpha_c[N]}}$ is obvious.
We check the converse inclusion.
Take any $X \in \cal X_\alpha$.
There exists a map $f : \mathbb N \to X$ such that $m_X - f_\# \alpha$ is nonnegative.
For every ${N \ge 1}$, we define a map $f_N : \mathbb N \to X$ by
\[
f_N(i) := \begin{cases}
f(i) & \text{if } i < N, \\
f({N}) & \text{if } i \ge N,
\end{cases}
\]
and define a Borel probability measure $m_{X_N}$ on $X$ by
\[
m_{X_N} := m_X - f_\# \alpha + {f_N}_\# \alpha.
\]
Note that $X_N := (X, d_X, m_{X_N}) \in \cal X_{\alpha_c[N]}$.
Then we have
\begin{align*}
\square(X, X_N) &\le 2d_{\Prkh}(m_X, m_{X_N}) \\
&\le 2\|m_X-m_{X_N}\|_{\mathrm{TV}} \le 4\sum_{{i=N}}^\infty \alpha_i \to 0
\end{align*}
as $N \to \infty$, which implies $X \in \overline{\bigcup_{{N=1}}^\infty \cal X_{\alpha_c[N]}}$.
\end{proof}

\subsection{Dissipation with atoms} \label{subsec:dissip}

In this subsection, we extend the separation distance 
	in {Definitions \ref{def:separation} and \ref{def:sep_for_pyramid}} 
	and introduce a generalized notion of dissipation.
Results appeared in this subsection are generalizations of \cite{OS:lim}.

\begin{definition} \label{def:Sep_new}
Let $\kappa = (\kappa_i)_{i=1}^\infty \in \mathbb R^\mathbb N$.
For $X \in \cal X$, the {\it separation distance} $\Sep(X;\kappa)$ of $X$ with respect to $\kappa $ is defined by the supremum of $\inf_{i\neq j} d_X(A_i,A_j)$,
where $\{A_i\}_{i=1}^\infty$ is a family of (possibly empty) closed subsets of $X$ with $m_X(A_i) \ge \kappa_i$ for each $i$.
\end{definition}

We note several properties of our separation distance.
If $\sigma : \mathbb N \to \mathbb N$ is a bijection, then
\begin{equation} \label{eq:shuffle}
\Sep(X; \kappa) = \Sep(X; \kappa\circ \sigma)
\end{equation}
holds.
Under the convention
\[
d(A, \emptyset) = \infty,
\]
we have
\begin{equation} \label{eq:negative_part_is_omitted}
\Sep(X;\kappa) = \Sep(X;\kappa_+)
\end{equation}
for any $\kappa\in \mathbb R^\mathbb N$, where $\kappa_+ := (\max\{\kappa_i,0\})_{i=1}^\infty$.
In particular, Definition \ref{def:Sep_new} is an extension of Definition \ref{def:separation} in the sense that
\[
\Sep(X;(\kappa_1, \ldots, \kappa_N, 0, 0, \ldots)) = \Sep(X;\kappa_1, \ldots, \kappa_N)
\]
holds, for any $\kappa_1, \ldots, \kappa_N > 0$.
Note that
\begin{equation} \label{eq:sep_monotone_in_kappa}
\Sep(X;\kappa') \le \Sep(X;\kappa)
\end{equation}
for $\kappa, \kappa' \in \mathbb R^\mathbb N$ with $\kappa_i \le \kappa'_i$ for every $i$ and that
\begin{equation} \label{eq:sep_sum_formula}
\Sep(X;\kappa, \lambda) \le \Sep(X;\kappa+\lambda)
\end{equation}
for $\kappa, \lambda \in \mathbb R^\mathbb N$, 
where $(\kappa, \lambda)$ is the sequence of real numbers given by
\[
(\kappa, \lambda)(i) = \left\{
\begin{aligned}
& \kappa((i+1)/2) && \text{ if $i$ is odd}, \\
& \lambda(i/2) && \text{ if $i$ is even}.
\end{aligned}
\right.
\]
Note that the above definition of $(\kappa, \lambda)$ is not important due to \eqref{eq:shuffle}.

To give the following definition, we set
\[
\overrightarrow{\theta} := (\theta, \theta, \ldots)
\]
for a real number $\theta$.

\begin{definition}
For $P \in \Pi$, the {\it separation distance} $\Sep(P;\kappa)$ of $P$ with respect to $\kappa $ is defined to be
\[
\Sep(P; \kappa) := \lim_{\delta\to 0+} \sup_{X \in P} \Sep(X; \kappa -\overrightarrow{\delta}).
\]
\end{definition}

The above formula is an analogy of \cite{OS:lim}.
By the definition, $\Sep(P;\kappa)$ also satisfies \eqref{eq:shuffle}, \eqref{eq:sep_monotone_in_kappa} and \eqref{eq:sep_sum_formula}, and it satisfies
\begin{equation}\label{eq:sep_left_conti_in_kappa}
\lim_{\delta\to 0+}\Sep(P; \kappa-\overrightarrow{\delta}) = \Sep(P; \kappa).
\end{equation}

\begin{lemma}
For any $X \in \cal X$ and $\kappa \in \mathbb R^{\mathbb N}$, we have
\[
\lim_{\delta\to 0+}\Sep(X; \kappa-\overrightarrow{\delta}) = \Sep(X; \kappa).
\]
In particular,
\[
\Sep(\cal PX; \kappa) = \Sep(X; \kappa).
\]
\end{lemma}
\begin{proof}
Take any sequence $\{\delta_n\}_{n=1}^\infty$ of positive numbers with $\delta_n \to 0$ and put
\[
r:= \lim_{n\to \infty} \Sep(X; \kappa - \overrightarrow{\delta_n}).
\]
By \eqref{eq:sep_monotone_in_kappa}, we have $\Sep(X; \kappa) \le r$. Let us prove the converse inequality. By the definition of $r$, there {exist} sequences $\{A_i^n\}_{i=1}^\infty$, $n \in \mathbb N$, of closed sets in $X$ such that
\[
m_X(A_i^n) \ge \kappa_i-\delta_n \text{ for any $i$ and $n$} \quad \text{and} \quad \lim_{n \to \infty} \inf_{i\neq j} d_X(A_i^n, A_j^n) = r.
\]
By the diagonal argument, there exists a subsequence $\{n_k\}$ of $n$ such that $\{A_i^{n_k}\}_{k=1}^\infty$ converges weakly (in the weak Hausdorff sense) for every $i$. Let $A_i$ be the (weak Hausdorff) limit of $\{A_i^{n_k}\}_{k=1}^\infty$. Then we have
\[
m_X(A_i) \ge {\limsup_{k\to\infty}}\, m_X(A_i^{n_k}) \ge \kappa_i
\]
by \cite[Lemma 6.4]{NS:group}. Moreover, we have
\[
\inf_{i\neq j} d_X(A_i, A_j) \ge r.
\]
Indeed, if $x \in A_i$ and $y \in A_j$, then there exist sequences $\{x_k\}_{k=1}^\infty$ and $\{y_k\}_{k=1}^\infty$ in $X$ convergent to $x$ and $y$ as $k\to\infty$, respectively, such that $x_k \in A_i^{n_k}$ and $y_k \in A_j^{n_k}$ for every $k$. Thus we have
\[
d_X(x, y) = \lim_{k\to\infty}d_X(x_k,y_k) \ge \lim_{n \to \infty} \inf_{i\neq j} d_X(A_i^n, A_j^n) = r,
\]
which implies $d_X(A_i, A_j) \ge r$. If $A_i = \emptyset$ or  $A_j = \emptyset$, then $d_X(A_i, A_j) \ge r$ holds trivially.

Combining these implies
\[
r \le \inf_{i\neq j} d_X(A_i, A_j) \le \Sep(X; \kappa).
\]
This completes the proof.
\end{proof}

\begin{proposition}[limit formula]\label{prop:limit_formula}
Let $\kappa \in \mathbb R^{\mathbb N}$.
If a sequence $\{P_n\}_{n=1}^\infty$ of pyramids converges weakly to a pyramid $P$, then we have
\begin{align*}
\Sep(P; \kappa) &= \lim_{\epsilon \to 0+} \liminf_{n\to\infty} \Sep(P_n; \kappa-\overrightarrow{\epsilon}) \\
&= \lim_{\epsilon \to 0+} \limsup_{n\to\infty} \Sep(P_n; \kappa-\overrightarrow{\epsilon})
\end{align*}
for any $\kappa \in {\mathbb R}^{\mathbb N}$.
\end{proposition}

\begin{remark} \label{rem:OS_limit_formula}
Proposition \ref{prop:limit_formula} is obtained by Ozawa-Shioya (\cite{OS:lim}) in the case of $2\le N(\kappa) <\infty$.
\end{remark}

\begin{proof}[Proof of Proposition \ref{prop:limit_formula}]
By \eqref{eq:negative_part_is_omitted}, we may assume that $\kappa= \kappa_+$.
When $\|\kappa\|_1 > 1$, the statement trivially holds.
Hence, we may assume that $\|\kappa\|_1 \le 1$.
If $N(\kappa) < 2$, then all sides of the desired equality are infinite.
By Remark \ref{rem:OS_limit_formula},
$\kappa$ can be assumed $N(\kappa) = \infty$.
Take any $\eta > 0$.
Taking $N(\kappa - \overrightarrow{\eta}) < \infty$ into account, we have
\begin{align*}
\Sep(P; \kappa- \overrightarrow{\eta}) &= \lim_{\epsilon \to 0+} \liminf_{n\to\infty} \Sep(P_n; \kappa- \overrightarrow{\eta}-\overrightarrow{\epsilon}) \\
&= \lim_{\epsilon \to 0+} \limsup_{n\to\infty} \Sep(P_n; \kappa-\overrightarrow{\eta}-\overrightarrow{\epsilon}),
\end{align*}
{by Remark \ref{rem:OS_limit_formula}.}
As $\eta \to 0$, we obtain the conclusion by \eqref{eq:sep_left_conti_in_kappa}.
\end{proof}

\begin{definition}\label{dfn:pyramid_dissip}
Let $\{P_n\}_{n=1}^\infty$ be a sequence of pyramids and let $\delta > 0$ and $\alpha\in V$ with $\alpha \neq 1$.
We say that $\{P_n\}_{n=1}^\infty$ {\it $\delta$-dissipates with atoms} $\alpha$ if for any $M \in \mathbb N$ and {any} sequence $\kappa = (\kappa_1,\ldots,\kappa_M)$ of nonnegative numbers with $\|\kappa\|_1 + \|\alpha\|_1 \le 1$, we have
\[
\lim_{\epsilon \to 0+}\liminf_{n\to\infty} \Sep(P_n; (\alpha, \kappa) - \overrightarrow{\epsilon}) \ge \delta.
\]
Moreover, we say that $\{P_n\}_{n=1}^\infty$ {\it infinite dissipates with atoms} $\alpha$ if it $\delta$-dissipates with atoms $\alpha$ for any $\delta > 0$.

For a sequence $\{X_n\}_{n=1}^\infty$ of pm-spaces, we say that $\{X_n\}$ {\it  $\delta$-dissipates with atoms} $\alpha$ if the sequence $\{\cal PX_n\}_{n=1}^\infty$ $\delta$-dissipates with atoms $\alpha$, for $\delta \in (0,\infty]$ and $\alpha \in V \setminus \{1\}$.
\end{definition}

\begin{remark}
In the case of $\alpha = 1$, then the separation distance
\[
\Sep(P; (1-\epsilon, -\epsilon, -\epsilon, \ldots)) = \infty, \quad P \in \Pi
\]
and then any sequence of pyramids always infinite dissipates with atom $\alpha = 1$. On the other hand, in the case of $\alpha = 0$, our definition is equivalent to the usual dissipation (see \cite[Section 8]{S} {and \cite[Definition 6.1]{OS:lim}}).
Indeed, we have
\begin{align*}
\Sep(X; \kappa_1, \ldots, \kappa_M) &\ge \Sep(X; \kappa_1-\epsilon, \ldots, \kappa_M-\epsilon, \epsilon, \ldots, \epsilon) \\
&\ge \Sep(X; \kappa_1-\epsilon, \ldots, \kappa_M-\epsilon, \lambda-\epsilon, \ldots, \lambda-\epsilon)
\end{align*}
for any $\kappa_1, \ldots, \kappa_M > 0$ with $\|\kappa\|_1 < 1$ and any $\lambda \ge 2\epsilon$, which implies that our definition leads to the usual dissipation. The converse implication is trivial.
\end{remark}

\begin{proposition}\label{prop:dissip_sep}
Let $\{P_n\}_{n=1}^\infty$ be a sequence of pyramids and let $\delta > 0$ and $\alpha \in V$ with $\alpha\neq 1$. Then, the sequence $\{P_n\}_{n=1}^\infty$ $\delta$-dissipates with atoms $\alpha$ if and only if the weak limit $P$ of any weakly convergent subsequence of $\{P_n\}_{n=1}^\infty$ satisfies
\[
\Sep(P; \alpha, \kappa) \ge \delta
\]
for any finite nonnegative sequence $\kappa = (\kappa_1, \ldots, \kappa_M)$ with $\|\kappa\|_1+\|\alpha\|_1 \le 1$.
\end{proposition}

\begin{proof}
Assume $\{P_n\}_{n=1}^\infty$ $\delta$-dissipates with atoms $\alpha$ and a subsequence $\{P_{n_k}\}_{k=1}^\infty$ of $\{P_n\}_{n=1}^\infty$ converges weakly to $P$. Then Proposition \ref{prop:limit_formula} implies
\begin{align*}
\Sep(P; \alpha, \kappa) &= \lim_{\epsilon\to 0+} \liminf_{k\to\infty} \Sep(P_{n_k}; (\alpha,\kappa)-\overrightarrow{\epsilon}) \\
&\ge \lim_{\epsilon\to 0+} \liminf_{n\to\infty} \Sep(P_n; (\alpha,\kappa)-\overrightarrow{\epsilon}) \ge \delta
\end{align*}
for any finite sequence $\kappa = (\kappa_1, \ldots, \kappa_M)$ with $\|\kappa\|_1+\|\alpha\|_1 \le 1$.

Conversely, we assume the latter condition and prove $\{P_n\}_{n=1}^\infty$ $\delta$-dissipates with atoms $\alpha$. Take any $\epsilon > 0$ and any sequence $\kappa = (\kappa_1, \ldots, \kappa_M)$ with $\|\kappa\|_1+\|\alpha\|_1 \le 1$.
Then there exists a subsequence $\{P_{n_k}\}_{k=1}^\infty$ of $\{P_n\}_{n=1}^\infty$ such that
\[
\lim_{k\to\infty} \Sep(P_{n_k}; (\alpha,\kappa)-\overrightarrow{\epsilon}) = \liminf_{n\to\infty} \Sep(P_n; (\alpha,\kappa)-\overrightarrow{\epsilon}).
\]
By the compactness of $\Pi$, we can assume that $\{P_{n_k}\}_{k=1}^\infty$ converges weakly to a pyramid $P$. Thus, by the assumption and Proposition \ref{prop:limit_formula} and \eqref{eq:sep_monotone_in_kappa}, we have
\begin{align*}
\delta & \le \Sep(P; \alpha, \kappa) \\
& = \lim_{\eta\to 0+} \liminf_{k\to\infty} \Sep(P_{n_k}; (\alpha,\kappa)-\overrightarrow{\eta}) \\
& \le \lim_{k\to\infty} \Sep(P_{n_k}; (\alpha,\kappa)-\overrightarrow{\epsilon}) \\
& = \liminf_{n\to\infty} \Sep(P_n; (\alpha,\kappa)-\overrightarrow{\epsilon}),
\end{align*}
which implies that $\{P_n\}_{n=1}^\infty$ $\delta$-dissipates with atoms $\alpha$.
\end{proof}

\begin{lemma}\label{lem:dissip_contraction}
Let $\{{P_n}\}_{n=1}^\infty$ be a sequence of pyramids and let $\delta > 0$ and $\alpha \in V$ with $\alpha\neq 1$. If the sequence $\{{P_n}\}_{n=1}^\infty$ $\delta$-dissipates with atoms $\alpha$, then it $\delta$-dissipates with atoms $\beta$ for any contraction $\beta$ of $\alpha$.
\end{lemma}

\begin{proof}
Assume that $\beta$ is a contraction of $\alpha$. There exists a map $\pi : \mathbb N \to \mathbb N$ such that
\[
\beta_i = \sum_{j \in \pi^{-1}(i)} \alpha_j.
\]
Let $\epsilon > 0$. Since $\alpha \in V$, there exists a number $N \in \mathbb N$ such that $\sum_{i=N}^\infty \alpha_i < \epsilon$ and $N \to \infty$ as $\epsilon \to 0$. Then, we have
\[
\beta_i -2\epsilon \le \sum_{j \in \pi^{-1}(i)\cap\{1, \ldots, N\}} \alpha_j -\epsilon \le \sum_{j \in \pi^{-1}(i)\cap\{1, \ldots, N\}} (\alpha_j - N^{-1}\epsilon)
\]
for any $i$.
By using \eqref{eq:sep_monotone_in_kappa} and \eqref{eq:sep_sum_formula}, we have
\begin{align*}
&\Sep({P_n}; (\beta, \kappa) - \overrightarrow{2\epsilon}) \\
&\ge \Sep({P_n}; \alpha_1- N^{-1}\epsilon, \ldots, \alpha_N- N^{-1}\epsilon, \kappa_1-2\epsilon, \ldots, \kappa_M-2\epsilon) \\
&\ge \Sep({P_n}; (\alpha, \kappa)- \overrightarrow{N^{-1}\epsilon})
\end{align*}
for any $\kappa_1, \ldots, \kappa_M \ge 0$ with $\sum_{i=1}^{M} \kappa_i + \|\alpha\|_1 \le 1$. Therefore, by the assumption, we obtain
\[
\lim_{\epsilon \to 0+}\liminf_{n\to\infty}\Sep({P_n}; (\beta, \kappa) - \overrightarrow{\epsilon}) \ge \delta.
\]
We finish the proof.
\end{proof}

The following proposition gives an intuitive comprehension. We omit the proof because we do not use it later.
\begin{proposition}\label{prop:dissip}
Let $\{X_n\}_{n=1}^\infty$ be a sequence of pm-spaces and let $\delta>0$ and $\alpha \in V$ with $\alpha \neq 1$.
The sequence $\{X_n\}_{n=1}^\infty$ $\delta$-dissipates with atoms $\alpha$ if and only if for each $n$ there exist finite families $\{A_{ni}\}^{k_n}_{i=1}$, $\{B_{ni}\}^{k_n}_{i=1}$ of Borel subsets of $X_n$ such that
\begin{enumerate}
\item $\lim_{n \to \infty} k_n = \infty$,
\item $\lim_{n\to\infty} m_{X_n}\left(\bigcup_{i=1}^{k_n} A_{ni} \cup B_{ni}\right) = 1$,
\item $\liminf_{n\to\infty} \min_{C\neq D} d_{X_n}(C,D) \ge \delta$, where both $C$ and $D$ run over $\left\{A_{ni}, B_{ni} \mid 1\le i\le k_n\right\}$,
\item $\lim_{n\to\infty} \sum_{i=1}^{k_n} |m_{X_n}(A_{ni}) - \alpha_i| = 0$, and
\item $\lim_{n\to\infty} \sup_i m_{X_n}(B_{ni}) = 0$.
\end{enumerate}
\end{proposition}

\begin{lemma}
\label{lem:fin_is_dense}
For each $\alpha \in V$ and $\delta \in [0, \infty]$, the set
\[
\left\{ X \in \cal X_\alpha^\delta \,\middle|\, \# X < \infty
\text{ and } \mathrm{diam}\, X < \delta
\right\}
\]
is dense in $\cal X_\alpha^\delta$ in the box topology.
\end{lemma}
\begin{proof}
Let us take $X \in \cal X_\alpha^\delta$ and $\epsilon > 0$.
Due to \cite[Lemmas 4.18 and 4.19]{S}, there exist a finite subset $Y$ of $X$ and a Borel map $f : X \to Y$ such that $d_{\mathrm{P}}^{m_X}(f_\# m_X, m_X) \le \epsilon$.
This implies
\[
\square(X, (Y,d_X, f_\# m_X)) \le 2\epsilon.
\]
Since $Y \subset X$, we have $\mathrm{diam}\,Y \le \delta$.
{From the definition, $Y \in \cal X_\alpha$.}
Since the metric scale-change $(0,\infty) \ni t \mapsto t Z \in \cal X$ is continuous in the box topology for each $Z$, we obtain the conclusion.
\end{proof}

The following is a key to the proof of Theorem \ref{thm:mugen}.
\begin{proposition} \label{prop:dissip_pyramid}
Let $\{X_n\}_{n=1}^\infty$ be a sequence of pm-spaces and let $\delta > 0$ and $\alpha \in V$ with $\alpha\neq 1$. Then, the sequence $\{X_n\}_{n=1}^\infty$ $\delta$-dissipates with atoms $\alpha$ if and only if the weak limit of any weakly convergent subsequence of $\{X_n\}_{n=1}^\infty$ contains the pyramid $\cal X_{\alpha}^\delta$.
\end{proposition}

\begin{proof}
Assume that $\{X_n\}_{n=1}^\infty$ $\delta$-dissipates with atoms $\alpha$ and assume that $\{X_n\}_{n=1}^\infty$ converges weakly to a pyramid $P$ by subtracting a subsequence.
Then it is sufficient to prove that any finite pm-space $Y \in \cal X^\delta_\alpha$ with $\mathrm{diam}\,Y < \delta$ is included in $P$, by Lemma \ref{lem:fin_is_dense}.
Take any such space $Y$ and let $\{y_1, \ldots, y_M\} = Y$.
There exists a contraction $\beta$ of $\alpha$ such that $N = N(\beta) \le M$ and
\[
m_Y(\{y_i\}) \ge \beta_i
\]
for any $i =1,\ldots,N$, by relabeling $y_i$ if necessary.
We define real numbers $\kappa_1, \ldots, \kappa_M \ge 0$ by
\[
\kappa_i := \begin{cases}
m_Y(\{y_i\})-\beta_i & \text{if } i \le N, \\
m_Y(\{y_i\}) & \text{if } i > N.
\end{cases}
\]
Note that $\sum_{i=1}^M \kappa_i + \|\beta\|_1 = 1$.
Then, by Lemma \ref{lem:dissip_contraction} and by \eqref{eq:sep_monotone_in_kappa}, we have
\[
\liminf_{n\to\infty} \Sep(X_n; \beta_1-\epsilon,\ldots,\beta_N-\epsilon, \kappa_1-\epsilon, \ldots, \kappa_M-\epsilon) \ge \delta > \mathrm{diam}\, Y.
\]
for any $\epsilon > 0$.
Hence, for any $\epsilon > 0$,
there exists a number $n_0$ such that
\[
\Sep(X_n; \beta_1-\epsilon,\ldots,\beta_N-\epsilon, \kappa_1-\epsilon, \ldots, \kappa_M-\epsilon) > \mathrm{diam}\, Y
\]
for any $n \ge n_0$.
We find (possibly empty) closed subsets $A_{n1}, \ldots, A_{nN}$ and
$B_{n1}, \ldots, B_{nM}$ of $X_n$ such that
\begin{align}
&m_{X_n}(A_{ni}) \ge \beta_i - \epsilon, \quad m_{X_n}(B_{n{j}}) \ge \kappa_{{j}}-\epsilon, \text{ and } \label{eq:dissip_pyramid1} \\
& \min_{C\neq D} d_{X_n}(C, D) > \diam{Y}, \label{eq:dissip_pyramid2}
\end{align}
where {$1 \le i \le N$, $1 \le j \le M$ and}
both $C$ and $D$ run over $\{A_{ni}, B_{nj} \mid 1\le i \le N, 1 \le {j} \le M\}$.
Here, we can assume that $Y$ is a subset of $\mathbb R^M$ since $Y$ embeds into $(\mathbb R^M, \|\cdot\|_\infty)$ isometrically by the Kuratowski embedding.
We define a map $f_n : \bigcup_{i=1}^N A_{ni} \cup \bigcup_{{j}=1}^M B_{n{j}} \to \mathbb R^M$ by
\[
f_n(x) := \begin{cases}
y_i & \text{if } x \in A_{ni} \cup B_{ni} \text{ and } i \le N, \\
y_j & \text{if } x \in B_{n{j}} \text{ and } {j} > N.
\end{cases}
\]
By \eqref{eq:dissip_pyramid2}, $f_n$ is a 1-Lipschitz map. Hence $f_n$ can be extended to a 1-Lipschitz map on the whole $X_n$, which is also written as $f_n$,
by Proposition \ref{prop:Lip_ext}.
Then we have
\[
\|{f_n}_\# m_{X_n} - m_Y\|_{\mathrm{TV}} \le (N+M)\epsilon.
\]
Indeed, by \eqref{eq:dissip_pyramid1}, we have
\[
{f_n}_\# m_{X_n}(\{y_i\}) \ge m_{X_n}(A_{ni} \cup B_{ni}) \ge m_Y(\{y_i\}) + 2\epsilon
\]
for any $i \le N$, which implies
\[
|{f_n}_\# m_{X_n}(\{y_i\}) - m_Y(\{y_i\})| \le {f_n}_\# m_{X_n}(\{y_i\}) - m_Y(\{y_i\}) + 4\epsilon.
\]
Similarly, for any $i>N$,
\[
|{f_n}_\# m_{X_n}(\{y_i\}) - m_Y(\{y_i\})| \le {f_n}_\# m_{X_n}(\{y_i\}) - m_Y(\{y_i\}) + 2\epsilon
\]
holds. Thus we have
\begin{align*}
&\|{f_n}_\# m_{X_n} - m_Y\|_{\mathrm{TV}} \\
&= \frac{1}{2} \sum_{i=1}^M |{f_n}_\# m_{X_n}(\{y_i\}) - m_Y(\{y_i\})| \\
&\le \frac{1}{2} \sum_{i=1}^M \left( {f_n}_\# m_{X_n}(\{y_i\}) - m_Y(\{y_i\})\right) + (N+M)\epsilon \\
&= (N+M)\epsilon.
\end{align*}
Therefore the pm-space $(\mathbb R^M, \|\cdot\|_\infty, {f_n}_\# m_{X_n})$ box-converges to $Y$ as ${n} \to \infty$.
Since $(\mathbb R^M, \|\cdot\|_\infty, {f_n}_\# m_{X_n})$ is dominated by $X_n$, we obtain $Y \in P$.

Conversely, assume that $\{X_n\}_{n=1}^\infty$ converges weakly to a pyramid $P$ containing $\cal X^\delta_\alpha$.
We take any $\kappa_1, \ldots, \kappa_M \ge 0$ with $\sum_{i=1}^M \kappa_i + \|\alpha\|_1 \le 1$ and let
\[
\epsilon_0 := \frac{1}{M}\left(1-\|\alpha\|_1-\sum_{i=1}^M \kappa_i \right) \ge 0.
\]
We define a pm-space $Y$ as
\[
{Y = \{y_1, \ldots, y_M\} \sqcup \{x_j\}_{j=1}^{N(\alpha)}}
\]
with metric
\[
d_Y(y,y') = \delta \text{ if } y\neq y'
\]
and measure
\[
m_Y = \sum_{j=1}^{N(\alpha)}\alpha_j\delta_{x_j} + \sum_{i=1}^{M}(\kappa_i+\epsilon_0)\delta_{y_i}.
\]
Since $Y \in \cal X^\delta_\alpha$ and $\cal X^\delta_\alpha \subset P$, we have $Y \in P$. There exists a sequence $\{Y_n\}_{n=1}^\infty$ of pm-spaces box-converging to $Y$ with $Y_n \prec X_n$ for every $n$.
By Lemma \ref{lem:mG_convergece}, we may assume that $Y_n$ and $Y$ are isometrically embedded in a complete separable metric space $Z$ such that $m_{Y_n}$ weakly converges to $m_Y$ over $Z$. Take any $0<\epsilon<\delta/2$.
Since $m_Y(B_\epsilon(y)) = m_Y(\{y\})$ for any $y \in Y$, we have
\[
\lim_{n \to \infty} m_{Y_n}(B_\epsilon(y)) = \begin{cases}
\alpha_j & \text{if } y=x_j, \\
\kappa_i + \epsilon_0 & \text{if } y=y_i.
\end{cases}
\]
Thus, putting $A_{n,y} := B_\epsilon(y) \cap \supp\,m_{Y_n}$ for $y \in Y$, we have
\[
m_{Y_n}(A_{n,y}) \ge \begin{cases}
\alpha_j-\epsilon & \text{if } y=x_j, \\
\kappa_i-\epsilon & \text{if } y=y_i,
\end{cases}
\quad \text{and} \quad \min_{y\neq y'}d_{Y_n}(A_{n,y},A_{n,y'}) \ge \delta-2\epsilon
\]
for any sufficiently large $n$. These imply
\[
\liminf_{n\to \infty}\Sep(Y_n; (\alpha, \kappa)-\overrightarrow{\epsilon}) \ge \delta-2\epsilon.
\]
Therefore, since $X_n \succ Y_n$, we have
\[
\liminf_{n \to \infty}\Sep(X_n; (\alpha, \kappa)-\overrightarrow{\epsilon}) \ge \delta-2\epsilon
\]
Taking $\epsilon \to 0$, the sequence $\{X_n\}_{n=1}^\infty$ $\delta$-dissipates with atoms $\alpha$.
The proof is completed.
\end{proof}

\begin{corollary}\label{cor:dissip_pyramid_01}
Let $\{X_n\}_{n=1}^\infty$ be a sequence of pm-spaces and let $\delta>0$ and $\alpha \in V$ with $\alpha \neq 1$.
The sequence $\{X_n\}_{n=1}^\infty$ $\delta$-dissipates with atoms $\alpha$ if and only if for any $N \in \mathbb N$, the sequence $\{X_n\}_{n=1}^\infty$ $\delta$-dissipates with atoms $\alpha_c[N]$.
\end{corollary}

\begin{proof}
Propositions \ref{prop:lower_approximation_of_pyramid_with_atoms} and \ref{prop:dissip_pyramid} together imply this corollary.
\end{proof}

\begin{corollary} \label{cor:dissp_pyramid_02}
Let $\{X_n\}_{n=1}^\infty$ be a sequence of pm-spaces and let $\alpha \in V$ with $\alpha\neq 1$.
The sequence $\{X_n\}_{n=1}^\infty$ infinite dissipates with atoms $\alpha$ if and only if the weak limit of any weakly convergent subsequence of $\{X_n\}_{n=1}^\infty$ contains the pyramid $\cal X_\alpha$.
\end{corollary}

\begin{remark} \label{rem:X_alpha_is_blow-up_limit}
For each $\alpha \in V$, let $X_\alpha$ be as \eqref{eq:X_alpha}.
Then we have
\[
\cal X_\alpha = \infty X_\alpha.
\]
Indeed, $\cal X_\alpha \supset \infty X_\alpha$ is clear.
The converse inclusion follows from
Corollary \ref{cor:dissp_pyramid_02}
({see also Proposition \ref{prop:dissip}}).
\end{remark}

\begin{corollary} \label{cor:dissip_characterization}
Let $P$ be a pyramid and let $\delta > 0$ and $\alpha \in V$. Then the following conditions {\rm (}1{\rm )}--{\rm (}4{\rm )} are equivalent to each other.
\begin{enumerate}
\item $P$ contains the pyramid ${\mathcal X}_\alpha^\delta$.
\item There exists a sequence of mm-spaces convergent weakly to $P$ $\delta$-dissipates with atoms $\alpha$.
\item Any sequence of pyramids convergent weakly to $P$ $\delta$-dissipates with atoms $\alpha$.
\item For any finite sequence $\kappa = (\kappa_1, \ldots, \kappa_M)$ with $\|\kappa\|_1+\|\alpha\|_1 \le 1$,
\[
\Sep(P; \alpha, \kappa) \ge \delta.
\]
\end{enumerate}
\end{corollary}

\begin{proof}
Proposition \ref{prop:dissip_pyramid} means (1) $\Leftrightarrow$ (2) and Proposition \ref{prop:dissip_sep} means (2) $\Leftrightarrow$ (4) $\Leftrightarrow$ (3).
\end{proof}

\begin{corollary}
Let $\alpha \in V$ with $\alpha \neq 1$. Then we have
\[
\Sep({\mathcal X}_\alpha^\delta; \alpha, \kappa) = \delta
\]
for any finite sequence $\kappa = (\kappa_1, \ldots, \kappa_M)$ with $\|\kappa\|_1+\|\alpha\|_1 \le 1$ and $N(\alpha, \kappa) \ge 2$.
\end{corollary}

\begin{proof}
We take any sequence $\kappa = (\kappa_1, \ldots, \kappa_M)$ with $\|\kappa\|_1+\|\alpha\|_1 \le 1$. Then
\[
\Sep({\mathcal X}_\alpha^\delta; \alpha, \kappa) \ge \delta
\]
by Corollary \ref{cor:dissip_characterization}. Moreover, if $N(\alpha, \kappa) \ge 2$, it must hold
\[
\Sep(X; (\alpha, \kappa)-\overrightarrow{\epsilon}) \le \diam{X} \le \delta
\]
for any $X \in {\mathcal X}_\alpha^\delta$ and $\epsilon > 0$. Thus we obtain the conclusion.
\end{proof}

\subsection{Continuity}

{We prove Theorem \ref{thm:mugen} in this subsection.}

\begin{lemma} \label{lem:continuity02}
For $\alpha, \beta \in \widetilde V$, we have
\[
\rho(\cal X_\alpha, \cal X_\beta) \le \|\alpha-\beta\|_1.
\]
In particular,
\[
\rho(\cal X_\alpha, \cal X_\beta) \le \max\{N(\alpha), N(\beta) \} \|\alpha-\beta\|_\infty.
\]
\end{lemma}
This lemma guarantees the Lipschitz continuity of the correspondence $(V, \|\cdot\|_\infty) \ni \alpha \mapsto \cal X_\alpha \in \Pi$ on $\{\alpha \in V \mid N(\alpha) \le N\}$ for each $N \in \mathbb N$.

\begin{proof}[{Proof of Lemma \ref{lem:continuity02}}]
Let $N \ge 1$ and $R > 0$.
Take any $X \in \cal X_\alpha \cap \cal X(N,R)$ and prove that there exists $Y \in \cal X_\beta \cap \cal X(N,R)$ such that $\square(X, Y) \le 4\|\alpha-\beta\|_1$.
If $\beta = 0$, then let $Y := X$.
We now assume $\beta \neq 0$.
There exists a map $f : \mathbb N \to X$ such that
\[
\nu := m_X - f_\# \alpha
\]
is a nonnegative measure.
Let us consider the probability measure $\tilde m$ on $X$ defined by
\[
\widetilde m = \left\{
\begin{aligned}
& \frac{1}{\|\beta\|_1}f_\# \beta &&\text{if } \|\alpha\|_1 = 1, \\
& \frac{1-\|\beta\|_1}{1-\|\alpha\|_1} \nu + f_\# \beta
&&\text{if } \|\alpha\|_1 < 1.
\end{aligned}
\right.
\]
Note that $Y := (X,d_X,\widetilde m) \in \cal X_\beta \cap \cal X(N,R)$.
Let $\tilde \nu := \widetilde m - f_\# \beta$.
Then we have
\begin{align*}
\| m_X - \widetilde m \|_\mathrm{TV} & \le \| f_\# \alpha - f_\# \beta\|_\mathrm{TV} + \| \nu - \tilde \nu\|_\mathrm{TV} \\
&\le \|\alpha-\beta\|_1 + \left|\|\alpha\|_1-\|\beta\|_1 \right| \\
&\le 2 \|\alpha-\beta\|_1.
\end{align*}
Since $\square((Y,d,\xi), (Y,d,\xi')) \le 2 \|\xi-\xi'\|_\mathrm{TV}$ holds for every complete metric space $(Y,d)$ and probability measures $\xi,\xi' \in \mathscr P(Y)$, we obtain
\[
\square(X,Y) \le 4 \|\alpha-\beta\|_1.
\]
By the definition of $\rho$ (Definition \ref{def:rho}),
we have
\[
\rho(\cal X_\alpha, \cal X_\beta) \le \|\alpha-\beta\|_1.
\]
So, the first statement holds.
The second statement follows from
\[
\| \alpha- \beta\|_1 \le \max\{N(\alpha), N(\beta)\} \| \alpha - \beta\|_\infty.
\]
This completes the proof.
\end{proof}

\begin{proof}[{Proof of Theorem \ref{thm:mugen}}]
Recall that the injectivity is already proved by Proposition \ref{prop:injective}.
Since $(V, \|\cdot\|_\infty)$ is compact (Lemma \ref{lem:uniform_vs_pointwise}) and $\Pi$ is Hausdorff, we only prove the continuity.
Let us take $\alpha_n, \alpha \in V$ with $\|\alpha_n - \alpha\|_\infty \to 0$ as $n \to \infty$.

By the compactness of $\Pi$, we can assume that $\cal X_{\alpha_n}$ converges weakly to a pyramid $P$.
Let us prove $P = \cal X_\alpha$.
For any $n$ and $N$, we have $\cal X_{\alpha_n} \subset \cal X_{\alpha_n[N]}$.
By Lemma \ref{lem:continuity02},
$\cal X_{\alpha_n[N]}$ converges weakly to $\cal X_{\alpha[N]}$, so that $P \subset \cal X_{\alpha[N]}$.
Thus we obtain
\[
P \subset \bigcap_{N=1}^\infty \cal X_{\alpha[N]} = \cal X_\alpha
\]
by Lemma \ref{prop:intersection}.
We prove the converse inclusion by constructing a sequence $\{X_n\}_{n=1}^\infty$ of pm-spaces with $X_n \in \cal X_{\alpha_n}$ such that $\{X_n\}_{n=1}^\infty$ infinite dissipates with atoms $\alpha_c[N]$ for every $N$.
Indeed, if such sequence $\{X_n\}_{n=1}^\infty$ exists, then a weak limit of weakly convergent subsequence of $\{X_n\}_{n=1}^\infty$ contains $\cal X_{\alpha_c[N]}$ by
Corollary \ref{cor:dissp_pyramid_02}
and is contained in $P$ since $X_n \in \cal X_{\alpha_n}$, so that we obtain
\[
\cal X_\alpha = \overline{\bigcup_{N=1}^\infty \cal X_{\alpha_c[N]}} \subset P.
\]
We define a pm-space $X_n$ as $X_n = \{y_1, \ldots, y_n, x_1, x_2, \ldots \}$ with metric
\[
d_{X_n}(x,x') = n \text{ if } x \neq x'
\]
and measure
\[
m_{X_n} = \sum_{i=1}^\infty \alpha_{ni}\delta_{x_i}+ \sum_{i=1}^n \frac{1-\|\alpha_n\|_1}{n} \delta_{y_i}.
\]
Note that $X_n \in \cal X_{\alpha_n}$. Let us prove that $\{X_n\}_{n=1}^\infty$ infinite dissipates with atoms $\alpha_c[N]$. Take any $\kappa_1, \ldots, \kappa_M \ge 0$ with $\sum_{i=1}^M \kappa_i + \|\alpha\|_1 \le 1$ and take any $\epsilon > 0$.
Then, by \eqref{eq:unform_convergence001}, there exists $i_0 > N$ such that $\sup_{n} \alpha_{ni_0} < \epsilon$.
Take any sufficiently large $n$ satisfying
\[
\sum_{i=1}^{i_0}\left|\alpha_{ni} - \alpha_i\right| < \epsilon, \quad  \text{and} \quad \frac{1-\|\alpha_n\|_1}{n} < \epsilon.
\]
We rename points of $X_n$ as
\[
z_i := \begin{cases}
y_i & \text{if } i=1,\ldots,n, \\
x_{i-n+i_0} & \text{if } i > n.
\end{cases}
\]
Then
\[
X_n=\{x_1,\ldots,x_{i_0},z_1,z_2,\ldots\}.
\]
We see that $m_{X_n}(\{z_i\}) < \epsilon$ for any $i$.
Moreover, we have
\begin{equation}\label{eq:homeo_thm}
m_{X_n}\left(\{z_1,z_2,\ldots\} \right) = 1-\sum_{i=1}^{i_0} \alpha_{ni} \ge 1-\sum_{i=1}^{i_0}\alpha_i -\epsilon \ge \sum_{i=1}^M \kappa_i -\epsilon.
\end{equation}
Borel subsets $A_1, \ldots, A_M$ of $X_n$ are defined as
\begin{align*}
A_i := \{z_{\ell_{i-1}+1}, \ldots, z_{\ell_i}\},
\end{align*}
where $\ell_0 := 0$ and
\[
\ell_i := \min\left\{\ell \ge \ell_{i-1} \, \middle| \, m_{X_n}(\{z_{\ell_{i-1}+1}, \ldots, z_\ell\})\ge \kappa_i -2\epsilon\right\}
\]
for $1\le i\le M$. Here, each $\ell_i$ is well-defined, because $m_{X_n}(A_i) \le \kappa_i-\epsilon$ whenever $\kappa_i - 2\epsilon > 0$ and \eqref{eq:homeo_thm} hold.
And let
\begin{align*}
A_0 &:= X_n \setminus \left(\bigcup_{i=1}^M A_i \cup \{x_1, \ldots, x_{N-1}\}\right)\\
&= \{x_N, \ldots, x_{i_0}, z_{\ell_M+1}, z_{\ell_M+2},\ldots \}.
\end{align*}
Then we have
\[
m_{X_n}(A_0) \ge 1- \sum_{i=1}^{N-1} \alpha_{ni} - \sum_{i=1}^{M} m_{X_n}(A_i) \ge 1 - \sum_{i=1}^{N-1} \alpha_i - \sum_{i=1}^{M} \kappa_i -\epsilon  \ge \sum_{i=N}^\infty \alpha_i -\epsilon.
\]
Therefore we obtain
\begin{align*}
&\Sep \left(X_n; \alpha_1-\epsilon, \ldots, \alpha_{N-1}-\epsilon, \sum_{i=N}^\infty \alpha_i-\epsilon, \kappa_1-2\epsilon, \ldots, \kappa_M-2\epsilon \right) \\
&\ge \min \left\{\min_{1\le i\neq j \le N-1} d_{X_n}(x_i, x_j), \min_{\substack{1\le i\le N-1,\\ 0\le j \le M}}d_{X_n}(x_i, A_j), \min_{0\le i\neq j \le M} d_{X_n}(A_i, A_j) \right\} \\
&= n,
\end{align*}
which implies that $\{X_n\}_{n=1}^\infty$ infinite dissipates with atoms $\alpha_c[N]$. The proof is completed.
\end{proof}

\begin{remark}
Our proof of Theorem {\ref{thm:mugen}} is due to an application of dissipation property (Proposition \ref{prop:dissip_pyramid} and Corollary \ref{cor:dissp_pyramid_02}).
Another direct proof appears in \cite{KNS}.
\end{remark}

\subsection{Pyramids {generated by atoms with bounded diameter}}
We are going to determine the topology of the space
\[
\cal S := \left\{ \cal X_\alpha^\delta \in \Pi \,\middle|\, \alpha \in V, \delta \in [0,\infty] \right\}.
\]
Here, $\cal X_\alpha^\infty := \cal X_\alpha$.
Note that
\[
\cal X_\alpha^\delta = \{\ast\} \iff \alpha = 1 \text{ or } \delta = 0.
\]
Let us consider the map
\[
\Phi : V \times [0,\infty] \ni (\alpha, \delta) \mapsto \cal X_\alpha^\delta \in \cal S.
\]
We introduce the equivalent relation on $V \times [0,\infty]$ generated by
\[
(\alpha, \delta) \sim (\beta, \delta') \iff
(\alpha=1 \text{ or } \delta = 0) \text{ and } (\beta = 1 \text{ or } \delta' = 0).
\]
Let us denote by $[\alpha, \delta]$ the equivalent class of $(\alpha, \delta)$.
Then the map
\[
\Psi([\alpha, \delta]) := \Phi(\alpha, \delta)
\]
is well-defined.
Furthermore, we have
\begin{theorem} \label{thm:homeo_with_delta}
The map $\Psi : V \times [0,\infty]/\!\!\sim\, \to \mathcal S$ is a homeomorphism.
Here, $V \times [0,\infty]/\!\!\sim$ is equipped with the quotient topology of the product topology of $V$ and $[0,\infty]$.
\end{theorem}
\begin{proof}
To prove the injectivity of $\Psi$, it suffices to show that $\Phi$ is injective on $(V\times [0,\infty]) \setminus \{ (\alpha ,\delta) \mid \alpha = 1 \text{ or } \delta = 0\}$.
This follows from an argument similar to the proof of Proposition \ref{prop:injective}.
We give a proof for reader's convenience.
Let $(\alpha, \delta), (\beta, \delta')$ with $\alpha \ne 1 \ne \beta$ and $\delta, \delta' > 0$.
Suppose $\cal X_\alpha^\delta = \cal X_{\beta}^{\delta'}$.
We may assume $\delta, \delta' < \infty$.
Let us consider a pm-space
\[
X_\alpha = ([0,\delta] \cup \left\{ x_i\right\}_{i=1}^{N(\alpha)}, d, \mu)
\]
where the metric and the measure are defined by
\[
d(t, x_i) = \delta = d(x_i,x_j) \text{ and } d(s,t) = |s-t|
\]
for $i \ne j$ and $s,t \in [0,\delta]$, and
\[
\mu = (1 - \|\alpha\|_1) 
\cal L^1 \lfloor_{[0,\delta]}
+ \sum_{i=1}^{N(\alpha)} \alpha_i \delta_{x_i}.
\]
Since $X_\alpha \in \cal X_\beta^{\delta'}$ and the Lebesgue measure $\cal L^1$ has no atom, $\|\beta\|_1 \le \|\alpha\|_1$ and $\delta \le \delta'$.
By symmetry, we have $\|\alpha\|_1 = \| \beta\|_1$ and $\delta = \delta'$.
Using $X_\alpha \in \cal X_\beta^{\delta'} \subset \cal X_\beta$ again and Proposition \ref{prop:bi-contraction}, we have $\alpha = \beta$.

Next, we prove the continuity of $\Psi$.
To do this, it suffices to show that $\Phi$ is continuous on $V \times [0, \infty]$.
By an argument similar to the proof of Theorem {\ref{thm:mugen}}, we know that $\Phi(\alpha, \delta)$ is continuous in $\alpha$.
Indeed, we obtain
\[
\rho(\cal X_\alpha^\delta, \cal X_\beta^\delta) \le \|\alpha-\beta\|_1
\]
and construct a sequence of pm-spaces $\delta$-dissipating with certain atoms.
We prove that $\Phi(\alpha, \delta)$ is continuous in $\delta$.
If $\delta' \to \delta + 0$, then
\[
\lim_{\delta' \to \delta+0} \cal X_\alpha^{\delta'} = \bigcap_{\delta' > \delta} \cal X_\alpha^{\delta'} = \cal X_\alpha^\delta
\]
holds.
If $\delta' \to \delta-0$, then
\[
\lim_{\delta' \to \delta-0} \cal X_\alpha^{\delta'} = \overline{ \bigcup_{\delta' < \delta} \cal X_\alpha^{\delta'}} \subset \cal X_\alpha^\delta
\]
holds.
Let us show the reverse inclusion.
We may assume $\delta > 0$.
Let us take $X =(X,d_X, m_X) \in \cal X_\alpha^\delta$.
For $\delta' < \delta$, we consider
\[
X^{\delta'} := (X, \min\{d_X, \delta'\}, m_X).
\]
Then, $X^{\delta'} \in \cal X_\alpha^{\delta'}$.
By \cite{K:mtf}, we have $\square(X^{\delta'}, X) \to 0$ as $\delta' \to \delta-0$.
So, the above inclusion is actually the equality.
This completes the proof.
\end{proof}

\begin{proposition} \label{prop:concentrated_X_alpha^delta}
\label{prop:delta_alpha_conc}
The pyramid $\cal X^\delta_\alpha$ is concentrated
if and only if $0 \le \delta < \infty$ and $\|\alpha\|_1 = 1$.
In particular, in the case of $\delta = \infty$ or $\|\alpha\|_1<1$, 
$\cal X_\alpha^\delta$ is non-trivial or is $\cal X$.
\end{proposition}

In order to prove this {proposition}, we first generalize \cite[Proposition 7.37]{S} to the following Proposition \ref{prop:7.37_alpha}.

\begin{proposition}\label{prop:7.37_alpha}
Let $\alpha \in V$ with $\|\alpha\|_1<1$ and let $p \in [1,\infty]$.
Assume that $X \in \cal X_\alpha$ {satisfies} $\#\supp\,(m_X-f_\#\alpha)\ge 2$, where $f$ is a map in the definition of $\cal X_\alpha$.
Put $\nu_X := m_X-f_\#\alpha$ and define a probability measure $m_n$ on $X^n$ by
\[
m_n :=  (1-\|\alpha\|_1)\left(\frac{\nu_X}{1-\|\alpha\|_1}\right)^{\otimes n} + \sum_{i=1}^\infty \alpha_i \delta_{(f(i), f(i), \ldots, f(i))}.
\]
Then $\{(X^n, d_{\ell_p}, m_n)\}$ is not $d_\conc$-Cauchy.
\end{proposition}

\begin{proof}
Since $\#\supp\,\nu_X \ge 2$, there exists a 1-Lipschitz function $\varphi \colon X \to \mathbb R$ that is not constant on $\supp\,\nu_X$.
We define functions $\varphi_{n,i} \colon X^n \to \mathbb R$, $i=1,2,\ldots,n$, by
\[
\varphi_{n,i}(x_1, \ldots, x_n) := \varphi(x_i) = \varphi \circ \mathrm{pr}_i (x_1, \ldots, x_n).
\]
Here, $\mathrm{pr}_i : X^n \to X$ denotes the projection into the $i$-th coordinate.
Clearly, $\varphi_{n,i}$ are 1-Lipschitz functions on $X^n$.
Then we have
\[
d_{\mathrm{KF}}^{m_n}([\varphi_{n,i}], [\varphi_{n,j}])
= d_{\mathrm{KF}}^{m_2}([\varphi_{2,1}], [\varphi_{2,2}])
= d_{\mathrm{KF}}^{m_2}(\varphi_{2,1}, \varphi_{2,2}+c)
\]
for any $i,j$ and for some $c\in \mathbb R$ since
${(\mathrm{pr}_{i}, \mathrm{pr}_{j})}_\# m_n = m_2$.
Moreover, we have $d_{\mathrm{KF}}^{m_2}({\varphi_{2,1}}, {\varphi_{2,2}}+c) > 0$.
Indeed, if not, then for any $x, x' \in \supp\,\nu_X$, $\varphi(x)=\varphi(x')+c$ holds. This contradicts that $\varphi$ is not constant.
Therefore
\[
\epsilon_0 := d_{\mathrm{KF}}^{m_n}([\varphi_{n,i}], [\varphi_{n,j}])
\]
does not depend on $n,i,j$ with $i\neq j$.
This implies $\{\cal L_1(X^n, d_{\ell_p},m_n)\}_{n}$ is not precompact in the Gromov-Hausdorff topology.
On the other hand, from the construction, $\{(X^n, d_{\ell_p},m_n)\}$ is monotone nondecreasing in the Lipschitz order.
So, it has the weak limit.
Therefore, $\{(X^n,d_{\ell_p}, m_n)\}_n$ is not $d_\conc$-Cauchy.
\end{proof}

\begin{remark}
In Proposition \ref{prop:7.37_alpha}, the assumption of $\#\supp\,\nu_X\ge 2$ is essential. Indeed, if the pm-space $X$ is $\mathbb N \cup \{0\}$ with measure
\[
m_X = \sum_{i=1}^\infty\alpha_i\delta_{i} + (1-\|\alpha\|_1)\delta_{0},
\]
then
\[
m_n = \sum_{i=1}^\infty\alpha_i\delta_{(i, \ldots, i)} + (1-\|\alpha\|_1)\delta_{(0,\ldots, 0)}
\]
and $(X^n, d_{\ell_p}, m_n)$ is mm-isomorphic to $(X, n^{1/p}d_X, m_X)$.
If $p=\infty$ in addition, the product space $(X^n, d_{\ell_\infty}, m_n)$ is mm-isomorphic to $X$ for any $n$.
\end{remark}

\begin{proof}[Proof of Proposition \ref{prop:delta_alpha_conc}]
We first prove that $\cal X^\delta_\alpha$ is concentrated if $0 \le \delta < \infty$ and $\|\alpha\|_1 = 1$.
Indeed, the pm-space $X$ given by
\[
X := \{x_1, x_2,\ldots\}, \quad d_X(x_i, x_j) := \delta \text{ if } i\neq j, \quad m_X := \sum_{i=1}^\infty \alpha_i \delta_{x_i}
\]
satisfies $\cal X_\alpha^\delta = \cal P X$. In particular, $\cal X_\alpha^\delta$ is concentrated.

We next prove that $\cal X^\delta_\alpha$ is not concentrated if $\|\alpha\|_1 < 1$.
Take a pm-space $X$ satisfying the assumption of Proposition \ref{prop:7.37_alpha} and $\diam{X} \le \delta$.
Let $P$ be the weak limit of $\{(X^n, d_{\ell_\infty}, m_n)\}_{n=1}^\infty$.
Proposition \ref{prop:7.37_alpha} implies that $P$ is not concentrated.
Since $(X^n, d_{\ell_\infty}, m_n) \in \cal X^\delta_\alpha$ for any $n$, we have $P \subset \cal X^\delta_\alpha$.
Combining them implies that $\cal X^\delta_\alpha$ is not concentrated.

We finally prove that $\cal X^\delta_\alpha$ is not concentrated if $\delta=\infty$. In this case, if $\kappa < \min\{\alpha_1, 1-\alpha_1\}$, then we have
\[
\ObsDiam(\cal X_\alpha;-\kappa) \ge \diam_{\mathbb R}(\alpha_1\delta_0+(1-\alpha_1)\delta_t; 1-\kappa) = t
\]
for any $t > 0$, so that
\[
\ObsDiam(\cal X_\alpha;-\kappa) = \infty,
\]
which implies $\cal X_\alpha$ is not concentrated.
We finish the proof.
\end{proof}

\begin{corollary}
Let $\overline{\cal X}$ denote the completion of $\cal X$ with respect to $d_\conc$.
Then we have $\dim (\Pi \setminus \iota(\overline{\cal X})) = \infty$.
\end{corollary}
\begin{proof}
By Proposition \ref{prop:concentrated_X_alpha^delta}, $\Pi \setminus \iota(\overline{\cal X})$ contains
\[
\{ \cal X_\alpha \in \Pi \mid \alpha \in V \text{ with } \|\alpha\|_1 < 1 \}
\]
By Theorem {\ref{thm:mugen}}, it has infinite dimension.
This completes the proof.
\end{proof}

\subsection{Nondissipation}

In this subsection, we define a generalized notion of non-dissipation and show some properties. 
The original one was introduced in \cite{G:green}.

\begin{definition} \label{def:is_not_dissipated}
We say that a pyramid $P$ {{\it is non-dissipated}} if $P$ does not contain $\cal X_\alpha^\delta$ for every $\delta \in (0,\infty]$ and $\alpha \in V$ with $\alpha \ne 1$.
Furthermore, a sequence $\{X_n\}$ of pm-spaces {{\it is non-dissipated}} if the weak limit of any weakly convergent subsequence of $\{X_n\}$ {is non-dissipated}.
\end{definition}

Note that if $P$ is non-dissipated in the sense of Definition \ref{def:is_not_dissipated}, then $P$ is not $\delta$-dissipated in the sense of
\cite[Definition 6.2]{OS:lim} for every $\delta > 0$.

\begin{proposition} \label{prop:Poincare_vs_atoms}
Let $P$ be a pyramid with $C_{2,2}(P) < \infty$.
Then $P$ does not contain $\cal X_\alpha^\delta$ for every $\delta \in (0,\infty]$ and $\alpha \in V$ with $\alpha \ne 1$.
\end{proposition}

\begin{proposition} \label{prop:nondissip_poincare}
Let $\{X_n\}_{n=1}^\infty$ be a sequence of pm-spaces.
If
\[
\sup_{n} C_{2,2}(X_n) < \infty,
\]
then $\{X_n\}_{n=1}^\infty$ is non-dissipated.
\end{proposition}

Propositions \ref{prop:Poincare_vs_atoms} and \ref{prop:nondissip_poincare} follow from Proposition \ref{prop:Poincare_vs_discrete}, because $\cal X_\alpha^\delta$ contains a two-points pm-space.

\begin{remark} \label{rem:nondissip_poincare_is_strong}
The statement of Proposition \ref{prop:nondissip_poincare}
is similar to but is stronger than \cite[Theorem 8.13]{S}.
Our proof is another to the proof in \cite{S}.
\end{remark}

\begin{corollary}
Let $(M,d_g,m)$ be a smooth pm-space, where $g$ is a complete Riemannian metric on $M$, $d_g$ stands for the induced distance function and $m = e^{-f}\,\mathrm{vol}_g$. Assume that it satisfies Bakry-\'Emery lower Ricci curvature bound
	\[
	\mathrm{Ric}_g + \mathrm{Hess}_g(f) \ge K g,
	\]
	where $K > 0$ is a constant. Then the sequence of $\ell_2$-product $M_2^n = (M,d_g,m)^n_2$ 
is non-dissipated.
\end{corollary}
\begin{proof}
This follows from the tensorization property (\cite{BGL}) and Proposition \ref{prop:nondissip_poincare}.
\end{proof}

\begin{lemma}
\label{lem:dissip_unif}
Let $\{X_n\}_{n=1}^\infty$ and $\{Y_n\}_{n=1}^\infty$ be two sequences of pm-spaces and let $f_n : X_n \to Y_n$ be uniformly equicontinuous maps with ${f_n}_\# m_{X_n}=m_{Y_n}$.
If $\{X_n\}_{n=1}^\infty$ is non-dissipated,
then neither is $\{Y_n\}_{n=1}^\infty$.
\end{lemma}

\begin{proof}
Suppose that $\{Y_n\}$ $\delta$-dissipates with atoms $\alpha$ for some $\delta > 0$ and $\alpha \in V$ with $\alpha \neq 1$.
Take any $\kappa_1, \ldots, \kappa_M \ge 0$ with $\sum_{i=1}^M \kappa_i + \|\alpha\|_1 \le 1$.
Then, for any $\epsilon > 0$, there exists a number $n_0$ such that
\[
\Sep(Y_n; (\alpha, \kappa) - \overrightarrow{\epsilon}) > \frac{\delta}{2}
\]
if $n \ge n_0$.
By the same discussion as in \cite[Lemma 8.17]{S}, there exists $\delta_0 > 0$ such that
\[
\Sep(X_n; (\alpha, \kappa) - \overrightarrow{\epsilon})  > \delta_0
\]
if $n \ge n_0$. Therefore $\{X_n\}$ also $\delta_0$-dissipates with atom $\alpha$. This completes the proof.
\end{proof}

\begin{theorem}[Nondissipation theorem]
Let $X$ be a compact, connected, and locally connected pm-space.
Then the sequence $\{X^n_\infty\}_{n=1}^\infty$ of $\ell_\infty$-product of 
$X$ is non-dissipated.
In particular, the infinite $\ell_\infty$-product $X^\infty_\infty$ does not contain $\cal X^\delta_\alpha$.
\end{theorem}

\begin{proof}
This follows from the same discussion {as in} \cite[Theorem 8.8]{S} since
{Proposition \ref{prop:nondissip_poincare}}
and Lemma \ref{lem:dissip_unif} are prepared.
The last statement follows from Proposition \ref{prop:dissip_pyramid}.
\end{proof}

\subsection{Algebraic property of pyramids generated by atoms} \label{subsec:X_alpha-algebra}

Recall that each $\gamma \in \widetilde V$ admits an injective map $\varphi : \mathbb N \to \mathbb N$ such that $\gamma \circ \varphi \in V$ and it is corresponding to the equivalent class of $\gamma$.
We call the map $\widetilde V \ni \gamma \mapsto \gamma \circ \varphi \in V$ the {\it sorting map}.

For $\alpha, \beta \in V$, we give the double sequence
\[
\mathbb N^2 \ni (i,j) \mapsto \alpha_i \beta_j \in [0,1].
\]
Let us fix some bijection $\mathbb N \to \mathbb N^2$, the double sequence $(\alpha_i \beta_j)_{(i,j) \in \mathbb N^2}$ is regarded as a single sequence.
Furthermore, composing to the sorting map, we obtain a single monotone sequence $\alpha \beta \in V$.
We call $\alpha \beta$ the {\it product} of $\alpha$ and $\beta$.
Note that
\[
\|\alpha \beta\|_1 = \|\alpha \|_1 \cdot \|\beta\|_1
\]
holds.
It is clear that $\alpha \beta = \beta \alpha$, $\alpha 1 = \alpha$ and $\alpha 0 = 0$ hold.

\begin{lemma} \label{lem:sorting_map_is_continuous}
The sorting map $\widetilde V \to V$ is continuous in the $\ell_\infty$-topology.
\end{lemma}
\begin{proof}
Let us denote by $\Phi$ the sorting map.
Suppose that $\alpha_n, \alpha \in \widetilde V$ 
satisfy
$\|\alpha_n - \alpha\|_\infty \to 0$ as $n \to \infty$.
We are going to show that
\begin{equation} \label{eq:sorting_is_conti01}
\text{$\Phi(\alpha_n)_i \to \Phi(\alpha)_i$ as $n \to \infty$}
\end{equation}
for every $i \in \mathbb N$.
Note that if this is true, then the lemma is proved, due to Lemma \ref{lem:uniform_vs_pointwise}.

By the triangle inequality, we have $\left| \|\alpha_n\|_\infty - \|\alpha\|_\infty \right| \to 0$ as $n \to \infty$.
Since $\|\gamma\|_\infty = \Phi(\gamma)_1$ for every $\gamma \in \widetilde V$, we obtain \eqref{eq:sorting_is_conti01} for $i=1$.

We next prove \eqref{eq:sorting_is_conti01} for $i=2$.
Let us take
$J \in \mathbb N$ and $J_n \in \mathbb N$, for $n \in \mathbb N$, such that
\[
\alpha_{nJ_n} = \sup_{j} \alpha_{nj} \text{ and } \alpha_{J} = \sup_j \alpha_j,
\]
and consider new sequences $\beta_{n} = (\beta_{nj})_{j \ne J}, \beta =(\beta_j)_{j \ne J} \in [0,1]^{\mathbb N \setminus \{J\}}$ defined as follows.
Let $\beta_j = \alpha_j$ for every $j \ne J$.
If $J_n \ne J$, then we set
\[
\beta_{nj} = \left\{
\begin{aligned}
& \alpha_{nJ} &&\text{if } j = J_n, \\
& \alpha_{nj} &&\text{if } j\in \mathbb N \setminus \{J, J_n\}.
\end{aligned}
\right.
\]
When $J_n = J$, then we set
\[
\beta_{nj} = \alpha_{nj}
\]
for every $j \ne J$.
Then, we note that
\[
\|\beta_n \|_\infty = \Phi(\alpha_n)_2 \text{ and } \|\beta\|_\infty = \Phi(\alpha)_2.
\]
Therefore, to prove \eqref{eq:sorting_is_conti01} for $j=2$, it suffices to show that $\beta_n$ converges to $\beta$ as $n \to \infty$ uniformly.

From the assumption, for every $\epsilon > 0$, there exists $N \in \mathbb N$ such that
\begin{equation*}
|\alpha_{nj} - \alpha_{j} | \le \epsilon
\end{equation*}
for every $j \in \mathbb N$ and $n \ge N$.
From this, we have
\begin{equation*}
|\alpha_{nJ_n} - \alpha_{J}| \le \epsilon
\end{equation*}
for all $n \ge N$.
By these, we have
\[
|\alpha_{nJ} - \alpha_{J_n}| \le 3 \epsilon.
\]
Furthermore, these imply
\[
\|\beta_n - \beta\|_\infty \le 3 \epsilon
\]
if $n \ge N$.
Therefore, we have $\|\beta_n - \beta\|_\infty \to 0$ as $n \to \infty$.
Hence, \eqref{eq:sorting_is_conti01} holds for $i=2$.

By repeating this procedure inductively, we obtain \eqref{eq:sorting_is_conti01} for any $i \in \mathbb N$. This completes the proof.
\end{proof}

\begin{theorem} 
Let $p \in [1, \infty]$.
For $\alpha, \beta \in V$, we have
\[
\cal X_\alpha \otimes_p \cal X_\beta = \cal X_{\alpha \beta}.
\]
Furthermore, the map
\[
V \ni \alpha \mapsto \cal X_\alpha \in \{ \cal X_\beta \mid \beta \in V \}
\]
is an isomorphism as topological monoids.
\end{theorem}

\begin{proof}
For any $\alpha \in V$, let $X_\alpha$ be the pm-space given by \eqref{eq:X_alpha}.
Then $\{nX_\alpha\}_{n=1}^\infty$ is an approximation of the pyramid $\cal X_\alpha$ (see Remark \ref{rem:X_alpha_is_blow-up_limit}).
Thus, for any $\alpha, \beta \in V$, the sequence $\{nX_\alpha \times_p nX_\beta\}_{n=1}^\infty$ is an approximation of $\cal X_\alpha \otimes_p \cal X_\beta$ by Lemma \ref{lem:approx_product}.
For any $n$, since $nX_\alpha \times_p nX_\beta \in \cal X_{\alpha\beta}$, we have ${\cal X_\alpha \otimes_p \cal X_\beta} \subset \cal X_{\alpha \beta}$.
On the other hand, the sequence $\{nX_\alpha \times_p nX_\beta\}_{n=1}^\infty$  infinite dissipates with atoms $\alpha \beta$.
Thus $\cal X_\alpha \otimes_p \cal X_\beta \supset \cal X_{\alpha \beta}$ by Corollary \ref{cor:dissp_pyramid_02}.
Therefore we have $\cal X_\alpha \otimes_p \cal X_\beta = \cal X_{\alpha \beta}$.

The last statement follows from Lemma \ref{lem:sorting_map_is_continuous} and Theorem \ref{thm:mugen}.
\end{proof}

\begin{proposition}
Let $p \in [1,\infty]$. For any $\alpha, \beta \in V$ and $\delta, \epsilon > 0$, we have
\[
\cal X_{\alpha\beta}^{\min\{\delta, \epsilon\}} \subset \cal X_\alpha^\delta \otimes_p \cal X_\beta^\epsilon \subset \cal X_{\alpha\beta}^{\|(\delta, \epsilon)\|_p},
\]
where $\|(\delta, \epsilon)\|_p = (\delta^p+\epsilon^p)^{1/p}$ if $p<\infty$ and $\|(\delta, \epsilon)\|_\infty = \max\{\delta, \epsilon\}$ . In particular,
\[
\cal X_\alpha^\delta \otimes_\infty \cal X_\beta^\delta = \cal X_{\alpha\beta}^{\delta}
\]
for any $\delta > 0$.
\end{proposition}

\begin{proof}
For any $\alpha \in V$ and $\delta > 0$, let $X_n^{\alpha,\delta}$ be the pm-space given by
\[
X_n^{\alpha,\delta} = \{y_0, \ldots, y_{2^n-1}\} \sqcup \{x_i\}_{i=1}^{N(\alpha)}
\]
with metric
\[
d_{X_n^{\alpha,\delta}}(x, x') = \delta \text{ if } x\neq x'
\]
and with measure
\[
m_{X_n^{\alpha,\delta}} = \sum_{i=1}^{N(\alpha)} \alpha_i \delta_{x_i} + \sum_{i=0}^{2^n-1}\frac{1-\|\alpha\|_1}{2^n}\delta_{y_i}.
\]
Note that $X_n^{\alpha,\delta} \in \cal X_\alpha^\delta$.
Then $\{X_n^{\alpha,\delta}\}_{n=1}^\infty$ is an approximation of the pyramid $\cal X_\alpha^\delta$.
Indeed, the natural projection
\[
\{y_0, \ldots, y_{2^n-1}\} \ni y_i \mapsto y_{i \, \mathrm{mod} \, 2^{n-1}} \in \{y_0, \ldots, y_{2^{n-1}-1}\}
\]
implies
\[
X_1^{\alpha,\delta} \prec X_2^{\alpha,\delta} \prec \cdots \prec X_n^{\alpha,\delta} \prec \cdots.
\]
Moreover, $\{X_n^{\alpha,\delta}\}_{n=1}^\infty$ $\delta$-dissipates with atoms $\alpha$, which implies that $\{X_n^{\alpha,\delta}\}_{n=1}^\infty$ converges weakly to $\cal X_\alpha^\delta$.
Thus $\{X_n^{\alpha,\delta}\}_{n=1}^\infty$ is an approximation of $\cal X_\alpha^\delta$.

Therefore, for any $\alpha, \beta \in V$, any $\delta, \epsilon> 0$, and any $p \in [1,\infty]$,
the sequence $\{X_n^{\alpha,\delta} \times_p X_n^{\beta,\epsilon}\}_{n=1}^\infty$ is an approximation of $\cal X_\alpha^\delta \otimes_p \cal X_\beta^\epsilon$ by Lemma \ref{lem:approx_product}.
For any $n$, since
\[
\diam{(X_n^{\alpha,\delta} \times_p X_n^{\beta,\epsilon})} = \|(\delta, \epsilon)\|_p,
\]
we have $X_n^{\alpha,\delta} \times_p X_n^{\beta,\epsilon} \in \cal X_{\alpha\beta}^{\|(\delta,\epsilon)\|_p}$,
which implies $\cal X_\alpha^\delta \otimes_p \cal X_\beta^\epsilon \subset \cal X_{\alpha \beta}^{\|(\delta,\epsilon)\|_p}$.
On the other hand, the sequence $\{X_n^{\alpha,\delta} \times_p X_n^{\beta,\epsilon}\}_{n=1}^\infty$ $\min\{\delta, \epsilon\}$-dissipates with atoms $\alpha \beta$.
Thus $\cal X_\alpha^\delta \otimes_p \cal X_\beta^\epsilon \supset \cal X_{\alpha \beta}^{\min\{\delta, \epsilon\}}$.
The proof is completed.
\end{proof}

\begin{corollary}
Let $\alpha \in V \setminus \{1\}$ and $\delta \in (0,\infty]$.
Then, $(\cal X_\alpha^\delta)^{\otimes_\infty n}$
is different for each $n \in \mathbb N \cup \{\infty\}$.
\end{corollary}

\end{document}